\documentclass[a4paper,10pt]{amsart}

\usepackage{
	amsmath, 
	amssymb, 
	amsthm,
	fullpage, 
	tikz,
	hyperref,
  color
}

\usepackage[USenglish]{babel}
\usepackage{wasysym}
\usepackage[toc,page]{appendix}
\usepackage{mathrsfs} 
\usetikzlibrary{matrix,arrows,patterns}

\newcommand{\CC}{\mathrm{\mathbb{C}}}
\newcommand{\NN}{\mathrm{\mathbb{N}}}
\newcommand{\ZZ}{\mathrm{\mathbb{Z}}}
\newcommand{\tensor}{\otimes}
\newcommand{\diag}{\mathrm{diag}}
\newcommand{\End}{\mathrm{End}}
\newcommand{\tr}{\mathrm{tr}}
\newcommand{\lc}{\mathrm{lc}}
\newcommand{\Mat}{\mathrm{Mat}}
\newcommand{\io}{\iota}
\newcommand{\de}{\delta}
\newcommand{\be}{\beta}

\newcommand{\at}[2][]{#1|_{#2}}
\newcommand{\SU}{\mathrm{SU}}
\newcommand{\U}{\mathrm{U}}
\newcommand{\su}{\mathfrak{su}}
\newcommand{\Uq}{\mathcal{U}_q}
\newcommand{\BB}{\mathcal{B}}

\def\pfq#1#2#3#4#5#6{%
	{}_{#1}\varphi_{#2}\biggl(\genfrac..{0pt}{}{#3}{#4};#5,#6\biggr)%
}

\def\qbin#1#2{\left[\genfrac..{0pt}{}{#1}{#2}\right]}

\newtheorem{theorem}{Theorem}[section]
\newtheorem{proposition}[theorem]{Proposition}
\newtheorem{lemma}[theorem]{Lemma}
\newtheorem{definition}[theorem]{Definition}
\newtheorem{corollary}[theorem]{Corollary}
\theoremstyle{definition}

\newtheorem{remark}[theorem]{Remark}

\numberwithin{equation}{section} 

\title[Matrix-valued orthogonal polynomials]{Matrix-valued orthogonal polynomials related to the 
quantum analogue of $(\SU(2) \times \SU(2), \diag)$}
\date{\today. Published version in  Ramanujan J. \textbf{43} (2017), no. 2, 243--311. This version 
contains an update of references and a correction 
to Remark \ref{rmk:lemldu1} of the published version.}
\author{Noud Aldenhoven, Erik Koelink, Pablo Rom\'an}
\begin{document}



\begin{abstract} Matrix-valued spherical functions related to the quantum symmetric pair 
for the quantum analogue of $(\SU(2) \times \SU(2), \diag)$ are introduced and studied in 
detail. The quantum symmetric pair is given in terms of a quantised universal enveloping
algebra with a coideal subalgebra. The matrix-valued spherical functions 
give rise
to matrix-valued orthogonal polynomials, which are matrix-valued analogues
of a subfamily of Askey-Wilson polynomials. 
For these matrix-valued orthogonal polynomials a number of properties are derived 
using this quantum group interpretation: the orthogonality relations from
the Schur orthogonality relations, the three-term recurrence relation
and the structure of the weight matrix in terms of Chebyshev polynomials from 
tensor product decompositions, the matrix-valued Askey-Wilson type $q$-difference operators from the 
action of the Casimir elements. 
A more analytic study of the weight gives an explicit LDU-decomposition in terms 
of continuous $q$-ultraspherical polynomials. The LDU-decomposition gives the 
possibility to find explicit expressions of the matrix entries of the matrix-valued 
orthogonal polynomials in terms of continuous $q$-ultraspherical polynomials and
$q$-Racah polynomials. 


\end{abstract}

\maketitle

\tableofcontents

\section{Introduction}

Shortly after the introduction of quantum groups it was realised that many special functions of 
basic hypergeometric type \cite{GaspR} have a natural relation to quantum groups, see e.g. 
\cite[Ch.~6]{CharP}, \cite{KlimS}, \cite{Koor1994} for references. 
In particular, many orthogonal polynomials in the $q$-analogue of the Askey scheme, see e.g. \cite{KoekLS}, 
have found an interpretation on compact quantum groups analogous to the interpretation of orthogonal polynomials 
of hypergeometric type from the Askey scheme on compact Lie groups and related structures, see e.g. 
\cite{VileK}, \cite{Vile}.

In case of the harmonic analysis on classical Gelfand pairs one studies spherical functions and related 
Fourier transforms, see \cite{vDijk}. For our purposes a Gelfand pair consists of 
a Lie group $G$ and a compact subgroup $K$ so that the trivial representation of $K$ in the decomposition of 
any irreducible representation of $G$ restricted to $K$ occurs with multiplicity at most one. 
The spherical functions are functions on $G$ which are left- and right-$K$-invariant. The zonal spherical 
functions are realised as matrix elements of irreducible $G$-representations with respect to a fixed $K$-vector.
For special cases the zonal spherical functions 
can be identified with explicit special functions of hypergeometric type, see \cite[Ch.~9]{vDijk}, \cite[\S IV]{Fara}.
The zonal spherical functions are eigenfunctions to an algebra of differential operators, which includes the differential 
operator arising from the Casimir operator in case $G$ is a reductive group. 
For special cases with $G$ compact, we obtain orthogonality relations and differential operators for
the spherical functions, which can be identified with orthogonal polynomials from the 
Askey-scheme. For the special case $G=\SU(2)\times \SU(2)$ with $K\cong\SU(2)$ 
embedded as the diagonal subgroup, 
the zonal spherical functions are the characters of $\SU(2)$, which are identified with the Chebyshev polynomials $U_n$ 
of the second kind by the Weyl character formula. 
The Gelfand pair situation has been generalised to the setting of quantum groups, mainly in the compact 
context, see e.g. Andruskiewitch and Natale \cite{AndrN} for the case of finite-dimensional Hopf-algebra 
with a Hopf sub-algebra,
Floris \cite{Flor}, Koornwinder \cite{Koor1995}, Vainermann \cite{Vain} for more general 
compact quantum groups, and, for a non-compact example, Caspers \cite{Casp}. 

The notions of matrix-valued and vector-valued spherical functions have already emerged at the beginning of the development 
of the theory of spherical functions, see e.g. \cite{GangV} and 
references given there. However, the focus on the relation with matrix-valued or 
vector-valued special functions only came later, see 
e.g. references given in \cite{HeckvP}, \cite{vPrui12}. Gr\"unbaum, Pacharoni, Tirao \cite{GrunPT} give a group theoretic approach to matrix-valued
orthogonal polynomials emphasising the role of the matrix-valued differential operators, which are manipulated in great detail. The paper 
\cite{GrunPT} deals with the case $(G,K) = (\SU(3), \U(2))$. 
Motivated by \cite{GrunPT} and the approach of Koornwinder \cite{Koor1985}, the group theoretic interpretation of 
matrix-valued orthogonal polynomials on  $(G,K)=(\SU(2)\times \SU(2),\SU(2))$ is studied from a different point of view, in particular with less
manipulation of the matrix-valued differential operators, in \cite{KoelvPR12}, \cite{KoelvPR13}, see also
\cite{HeckvP}, \cite{vPrui12}. The point of view is to construct the matrix-valued orthogonal polynomials using matrix-valued spherical 
functions, and next using this group theoretic interpretation to obtain properties of the matrix-valued orthogonal polynomials. 
This approach for the case $(G,K)=(\SU(2)\times \SU(2),\SU(2))$ leads to matrix-valued orthogonal polynomials for arbitrary size, which can be considered as analogues
of the Chebyshev polynomials of the second kind. 
A combination of the group theoretic approach and analytic considerations then allows us 
to understand these matrix-valued orthogonal polynomials completely, i.e. we have explicit orthogonality relations, three-term recurrence relations,
matrix-valued differential operators having the matrix-valued orthogonal polynomials as eigenfunctions, expression in terms of Tirao's \cite{Tira} 
matrix-valued hypergeometric functions, expression in terms of well-known scalar-valued orthogonal polynomials from the Askey-scheme, etc. 
This has been analytically extended to arbitrary size matrix-valued orthogonal Gegenbauer polynomials \cite{KoeldlRR}, see also 
\cite{PachZ} for related $2\times 2$-cases. 

The interpretation on quantum groups and related structures leads to many new results for special functions of basic hypergeometric type.
In this paper we use quantum groups in order to obtain matrix-valued orthogonal polynomials as analogues of a subclass of the 
Askey-Wilson polynomials. In particular, we consider the Chebyshev polynomials of the second kind,
recalled in \eqref{eq:defChebyshevpols}, as a special case of the 
Askey-Wilson polynomials \cite[(2.18)]{AskeW}. 
Moreover, we know that the Chebyshev polynomials occur as characters on the quantum $\SU(2)$ group,
see \cite[\S A.1]{Woro}. 
The approach in this paper is to establish the quantum analogue of the group-theoretic approach as presented in \cite{KoelvPR12}, \cite{KoelvPR13}, see also
\cite{HeckvP}, \cite{vPrui12}, for the example of the Gelfand pair $G=\SU(2)\times \SU(2)$ with $K\cong\SU(2)$. 
For this approach we need Letzter's approach \cite{Letz03}, \cite{Letz04},
\cite{Letz08} to quantum symmetric spaces 
using coideal subalgebras. We stick to the conventions 
as in Kolb \cite{Kolb} and we refer to \cite[\S 1]{KolbS} for a broader perspective on 
quantum symmetric pairs.
So we work with the quantised universal enveloping algebra 
$\Uq(\mathfrak{g})=\Uq(\su(2)\oplus\su(2))$, introduced in Section \ref{sec:quantizeduniversalenvelopingalg}, 
equipped with a right coideal subalgebra $\BB$, see Section \ref{sec:mainresults}. 
Once we have this setting established, the branching rules of the representations of 
$\Uq(\mathfrak{g})$ restricted to $\BB$ follow 
by identifying $\BB$ with the image of $\Uq(\su(2))$ (up to an isomorphism) under the comultiplication using the standard 
Clebsch-Gordan decomposition. In particular, it gives explicit intertwiners. 
Next we introduce matrix-valued spherical functions in Section \ref{sec:mainresults}. 
Using the matrix-valued spherical functions we introduce the matrix-valued orthogonal 
polynomials. 
Then we use a mix of quantum group theoretic and analytic approaches to study these matrix-valued orthogonal polynomials. 
So we find the orthogonality for the matrix-valued orthogonal polynomials from the Schur orthogonality relations, the three-term recurrence relation
follows from tensor-product decompositions of $\Uq(\mathfrak{g})$-representations, the matrix-valued $q$-difference operators for which these matrix-valued 
orthogonal polynomials are eigenvectors follow from the study of the Casimir elements in 
$\Uq(\mathfrak{g})$. More analytic properties follow from 
the LDU-decomposition of the matrix-valued weight function, and this allows to 
decouple the matrix-valued $q$-difference operators involved.
The decoupling gives the possibility to link the entries of the matrix-valued orthogonal polynomials with (scalar-valued) 
orthogonal polynomials from the $q$-analogue of the Askey-scheme, in particular the continuous $q$-ultraspherical polynomials
and the $q$-Racah polynomials. 
The approach of \cite{GrunPT} does not seem to work in the quantum case, because the possibilities to transform $q$-difference equations are very limited compared to 
transforming differential equations. 
We note that in 
\cite[\S 5]{AndrN} matrix-valued spherical functions are considered for finite-dimensional Hopf algebras with 
respect to a Hopf-subalgebra.

The approach to matrix-valued orthogonal polynomials from this quantum group setting also leads to 
identities in the quantised function algebra. This paper does not include the 
resulting identities 
after using infinite-dimensional representations of the quantised function algebra.
Furthermore, we have not supplied a proof of Lemma \ref{lemma:sphericalelements} using infinite-dimensional
representations and the direct integral decomposition of the Haar functional, but this should be possible as well.

In general, the notion of a quantum symmetric pair seems to be best-suited for the development 
of harmonic analysis in general and in particular of matrix-valued spherical functions on quantum groups, see e.g. \cite{KolbS}, \cite{Letz03},
\cite{Letz04}, \cite{Letz08}, \cite{Noum}, \cite{ObloS} and references given there.
When considering other quantum symmetric pairs in relation to matrix-valued 
spherical functions, the branching rule of a representation of the quantised universal enveloping
algebra to a coideal subalgebra seems to be difficult. In this paper it is reduced to the Clebsch-Gordan
decomposition, and there is a nice result by Oblomkov and Stokman \cite[Prop.~1.15]{ObloS}
on a special case of the branching rule for quantum symmetric pair of type AIII, but in general the lack
of the branching rule for the quantum symmetric pairs is an obstacle for the study of quantum analogues 
of matrix-valued spherical functions 
of e.g. \cite{HeckvP}, \cite{GrunPT}, \cite{vPrui12}, \cite{ObloS}.

The matrix-valued orthogonal polynomials resulting from the study in this paper are matrix-valued
analogues of the Chebyshev polynomials of the second kind viewed as an example of the 
Askey-Wilson polynomials. We expect that it is possible to obtain 
matrix-valued
analogues of the continuous $q$-ultraspherical polynomials viewed as subfamily of the 
Askey-Wilson polynomials using the approach of \cite{KoeldlRR} using the 
Askey-Wilson $q$-derivative instead of the ordinary derivative. 
We have not explicitly worked out the limit transition $q\uparrow 1$ of the results, but 
by the set-up it is clear that the formal limit gives back many of the results of 
\cite{KoelvPR12}, \cite{KoelvPR13}. 

The contents of the paper are as follows. In Section 
\ref{sec:genMVOP} we fix notation regarding matrix-valued orthogonal polynomials.
In Section \ref{sec:quantizeduniversalenvelopingalg} the notation for quantised universal 
enveloping algebras is recalled. 
Section \ref{sec:mainresults} states all the main results of this paper. It introduces the
quantum symmetric pair explicitly. Using the representations of the quantised universal enveloping algebra
and the coideal subalgebra, the 
matrix-valued polynomials are introduced. We continue to give 
explicit information on the orthogonality relations, three-term recurrence relations,
$q$-difference operators, the commutant of the weight, the LDU-decomposition of the weight, 
the decoupling of the $q$-difference equations and the 
link to scalar-valued orthogonal polynomials from the $q$-Askey scheme.  
The proofs of the statements of Section \ref{sec:mainresults} occupy the rest of the paper.
In Section \ref{sec:quantumgrouprelatedpropssphericalf} the main properties derivable from 
the quantum group set-up are derived, and we discuss in Appendix \ref{app:cgc} the 
precise relation of the branching rule for this quantum symmetric pair and the standard
Clebsch-Gordan decomposition. 
In Section \ref{sec:weightorthorel} we continue the study of the orthogonality relations,
in which we make the weight explicit. This requires several identities 
involving basic hypergeometric series, whose proofs we relegate to Appendix \ref{app:BproofsBHS}. 
Section  \ref{sec:qdifferenceoperators} studies the consequences of the explicit form of 
the matrix-valued $q$-difference 
operators of Askey-Wilson type to which the matrix-valued orthogonal polynomials are eigenfunctions.

In preparing this paper we have used computer algebra in order to verify the statements up to 
certain size of the matrix and up to 
certain degree of the polynomial in order to eliminate errors and typos. Note however, 
that all proofs are direct and do not use computer algebra. 
A computer algebra package used for this purpose can be found on the homepage of the second author.\footnote{\texttt{http://www.math.ru.nl/\~{}koelink/publist-ro.html}}

The convention on notation follows Kolb \cite{Kolb} for quantised universal enveloping algebras and right coideal subalgebras and we follow Gasper and Rahman \cite{GaspR} for the convention on basic hypergeometric series 
and we assume $0<q<1$.

\section{Matrix-valued orthogonal polynomials}\label{sec:genMVOP}

In this section we fix notation and give a short background to matrix-valued orthogonal polynomials, which were 
originally introduced by Krein in the forties, see e.g. references in \cite{Berg}, 
\cite{DamaPS}. General references for this section are \cite{Berg}, \cite{DamaPS}, \cite{GrunT}, and references given there.

Assume that we have a matrix-valued function $W\colon [a,b] \to M_{2\ell+1}(\CC)$, $2\ell+1\in \NN$, $a<b$,  so that 
$W(x)>0$ for $x\in [a,b]$ almost everywhere. We use the notation $A>0$ to denote a strictly positive definite
matrix. 
Moreover, we assume that all moments exist, where integration 
of a matrix-valued function means that each matrix entry is separately integrated. In particular, the integrals
are matrices in $M_{2\ell+1}(\CC)$. It then follows that for matrix-valued polynomials 
$P,Q \in M_{2\ell+1}(\CC)[x]$ the integral 
\begin{equation}\label{eq:defMVinnerpord}
\langle P, Q \rangle \, =\, \int_a^b P(x)^\ast\, W(x)\, Q(x)\, dx   \in M_{2\ell+1}(\CC).
\end{equation}
exists. This gives a matrix-valued inner product on the space $M_{2\ell+1}(\CC)[x]$ 
of matrix-valued polynomials, satisfying
\begin{gather*}
\langle P,Q\rangle =  \langle Q,P\rangle^\ast, \quad 
\langle P, QA+RB\rangle = \langle P, Q \rangle A + \langle P,R\rangle B, \\
\langle P,P\rangle = 0\in M_{2\ell+1}(\CC)\qquad \Longleftrightarrow \qquad P(x)=0\in M_{2\ell+1}(\CC) \ \forall\, x
\end{gather*}
for all $P,Q, R \in M_{2\ell+1}(\CC)[x]$ and 
$A,B \in M_{2\ell+1}(\CC)$. More general matrix-valued measures can be considered \cite{Berg}, \cite{DamaPS}, but for this paper the above set-up suffices.

A matrix-valued polynomial $P(x) = \sum_{r=0}^n x^r P^r$, $P^r\in M_{2\ell+1}(\CC)$ is of degree $n$ if the leading 
coefficient $P^n$ is non-zero. 
Given a weight $W$, there exists a family of matrix-valued polynomials $(P_n)_{n\in\NN}$ so that $P_n$ is a matrix-valued polynomial of degree $n$
and 
\begin{equation}\label{eq:genorthoMVOP}
\int_a^b \bigl(P_n(x)\bigr)^\ast\, W(x)\, P_m(x)\, dx\, =\, \delta_{n,m} G_n,
\end{equation}
where $G_n>0$. Moreover, the leading coefficient of $P_n$ is non-singular. Any other family of polynomials 
$(Q_n)_{n\in\NN}$ so that $Q_n$ is a matrix-valued polynomial of degree $n$ and $\langle Q_n, Q_m\rangle=0$ for $n\not= m$,
satisfies $P_n(x)=Q_n(x)E_n$ for some non-singular $E_n\in M_{2\ell+1}(\CC)$ for all $n\in \NN$. 
We call the matrix-valued polynomial $P_n$ monic in case the leading coefficient is the identity matrix $I$. The polynomials $P_n$ are 
called orthonormal in case the squared norm $G_n = I$
for all $n\in\NN$ in the orthogonality relations \eqref{eq:genorthoMVOP}.

The matrix-valued orthogonal polynomials $P_n$ always satisfy a matrix-valued three-term recurrence of the form
\begin{equation}\label{eq:gen3termrecurrence}
xP_n(x) = P_{n+1}(x)A_n + P_n(x) B_n + P_{n-1}(x) C_n
\end{equation}
for matrices $A_n,B_n,C_n\in M_{2\ell+1}(\CC)$ for all $n\in \NN$. Note that in particular $A_n$ is  
non-singular for all $n$. 
Conversely, assuming $P_{-1}(x)=0$ (by convention) 
and fixing the constant polynomial $P_0(x)\in M_{2\ell+1}(\CC)$ we can generate the polynomials $P_n$ 
from the recursion \eqref{eq:gen3termrecurrence}.  In case the polynomials are monic, the coefficient $A_n=I$ for all $n$ and $P_0(x)=I$ as the initial value. 
In general, the matrices satisfy $G_{n+1}A_n = C_{n+1}^\ast G_n$, $G_n B_n = B_n^\ast G_n$, so that
in the monic case $C_n = G_{n-1}^{-1} G_n$ for $n\geq 1$. 
In case the polynomials are orthonormal, we have $C_n=A_{n-1}^\ast$ and $B_n$ Hermitian.

Note that the matrix-valued `sesquilinear form'  \eqref{eq:defMVinnerpord} is antilinear in the first 
entry of the inner product, which leads to a three-term recurrence of the form 
\eqref{eq:gen3termrecurrence} where the multiplication by the constant matrices is from the right, 
see \cite{DamaPS} for a discussion. 

In case a subspace $V\subset \CC^{2\ell+1}$ is invariant for $W(x)$ for all $x$, 
$V^\perp$ is also invariant for $W(x)$ for all $x$. Let $\io_V\colon V\to \CC^{2\ell+1}$ be 
the embedding of $V$ into $\CC^{2\ell+1}$, so that $P_V=\io_V\io_V^\ast \in M_{2\ell+1}(\CC)$ is the 
corresponding orthogonal projection, then 
$W(x)P_V=P_VW(x)$ for all $x$. Let $P^V_n\colon [a,b]\to \text{End}(V)[x]$ be the matrix-valued 
polynomial defined by $P^V_n(x) = \io_V^\ast P_n(x) \io_V$, where $P_n$ are the monic matrix-valued 
orthogonal polynomials for the weight $W$. Then $P^V_n$ form a family of monic $V$-endomorphism-valued 
orthogonal polynomials, and $P_n(x) = P^V_n(x) \oplus P^{V^\perp}_n(x)$. The same decomposition 
can be written down for the orthonormal polynomials. 

The projections on invariant subspaces are in the commutant $\ast$-algebra 
$\{ T\in M_{2\ell+1}(\CC) \mid TW(x) = W(x)T \  \forall x\}$. In case the commutant algebra is trivial, the 
matrix-valued orthogonal polynomials are irreducible. 
The primitive idempotents correspond to 
the minimal invariant subspaces, and hence they determine the decomposition of the matrix-valued orthogonal
polynomials into irreducible cases.

\begin{remark}
In \cite{Tira15} the authors discuss non-orthogonal decompositions by considering, instead of the commutant algebra,
the real vector space
\begin{equation*}
\mathscr{A} = \{ Y \in \End(\mathcal{H}^{\ell}) : Y W(x) = W(x) Y^*, \forall x \in (-1, 1) \}.
\end{equation*}
It follows that if $I\mathbb{R} \subsetneq \mathscr{A}$, then the weight $W$ reduces, non-unitarily, to weights of smaller size. Koelink and Rom\'an \cite[Example 4.3]{KR15} showed that $\mathscr{A} = \{ W(x) : x \in (-1, 1) \}'$ so that, in our case, both decompositions coincide.
\end{remark}

We denote by $E_{i,j}\in M_{2\ell+1}(\CC)$ the matrix with zeroes except at the $(i,j)$-th entry where it is $1$.
So for the corresponding standard basis $\{e_k\}_{k=0}^{2\ell}$ we set $E_{i,j}e_k = \de_{j,k} e_i$. We usually use the basis
$\{e_k\}_{k=0}^{2\ell}$ in describing the results for the matrix-valued orthogonal polynomials, but occasionally the 
basis is relabelled $\{e^\ell_k\}_{k=-\ell}^{\ell}$, as is customary for the  $\Uq(\su(2))$-representations 
of spin $\ell$. In the latter case
we use superscripts to distinguish from the previous case; 
$E^\ell_{i,j}e^\ell_k = \de_{j,k} e^\ell_i$, $i,j,k\in \{-\ell,\cdots, \ell\}$.

\section{Quantised universal enveloping algebra}\label{sec:quantizeduniversalenvelopingalg}

We recall the setting for quantised universal enveloping algebras and quantised function algebras,
and this section is mainly meant to fix notation. The definitions can be found at various sources on
quantum groups, such as the books \cite{CharP}, \cite{EtinS}, \cite{KlimS}, and we follow Kolb \cite{Kolb}. 

Fix for the rest of this paper $0 < q < 1$. The quantised universal enveloping algebra 
can be associated to any root datum, but we only need the simplest cases $\mathfrak{g}=\mathfrak{sl}(2)$ and 
$\mathfrak{g}=\mathfrak{sl}(2)\oplus \mathfrak{sl}(2)$. The quantised universal enveloping
algebra is the unital associative algebra
generated by $k$, $k^{-1}$, $e$, $f$ subject to the relations
\begin{equation}\label{eq:defrelationsUqsl2}
kk^{-1}=1=k^{-1}k, \quad ke= q^2 ek \quad kf=q^{-2} fk, \quad ef-fe=\frac{k-k^{-1}}{q-q^{-1}},
\end{equation}
where we follow the convention as in \cite[\S 3]{Kolb}. For our purposes it is useful to extend 
the algebra with the roots of $k$ and $k^{-1}$, denoted by $k^{1/2}$, $k^{-1/2}$ satisfying
\begin{equation}\label{eq:defrelationsUqsl2-squareroots}
\begin{split}
&k^{1/2}k^{-1/2}=1=k^{-1/2}k^{1/2}, \quad k^{1/2}k^{1/2}=k, \quad  k^{-1/2}k^{-1/2}=k^{-1}, \\
&k^{1/2}e= q ek^{1/2}, \qquad  k^{1/2}f=q^{-1} fk^{1/2}.
\end{split}
\end{equation}
The extended algebra is denoted by $\Uq(\mathfrak{sl}(2))$, and it is a Hopf-algebra with 
comultiplication $\Delta$, counit $\varepsilon$, antipode $S$ defined on the generators by 
\begin{gather*} 
\Delta(e) = e\otimes 1 + k\otimes e, \quad 
\Delta(f) = f\otimes k^{-1} + 1\otimes f, \\
\Delta(k^{\pm 1/2}) = k^{\pm 1/2}\otimes k^{\pm 1/2}, \\
S(e) = -k^{-1}e, \quad S(f) = -fk, \quad S(k^{\pm 1/2})= k^{\mp 1/2}, \\
\varepsilon(e) = 0=\varepsilon(f), \quad \varepsilon(k^{\pm 1/2})=1.
\end{gather*}
The Hopf-algebra has a $\ast$-structure defined on the generators by 
\begin{equation*}
(k^{\pm 1/2})^\ast = k^{\pm 1/2}, \quad
e^\ast =  q^2fk, \quad f^\ast = q^{-2}k^{-1}e.
\end{equation*}
We denote the corresponding Hopf $\ast$-algebra by $\Uq(\mathfrak{su}(2))$. 

The identification as Hopf $\ast$-algebras with \cite{Koor1993}, \cite{Koel1996} is  
$(A,B,C,D) \leftrightarrow (k^{1/2}, q^{-1}k^{-1/2}e, qfk^{1/2}, k^{-1/2})$. 

The irreducible finite dimensional type 1  representations of the underlying $\ast$-algebra have been classified. 
Here type 1 means that the spectrum of $K^{1/2}$ is contained in $q^{\frac12\ZZ}$.
For each dimension $2\ell+1$ of spin $\ell \in \frac12\NN$ there is a representation in 
$\mathcal{H}^{\ell}\cong \CC^{2\ell+1}$ with orthonormal basis $\{e^\ell_{-\ell}, e^\ell_{-\ell+1}, \cdots, e^\ell_{\ell} \}$ 
and on which the action is given by 
\begin{equation}\label{eq:defrepresentationUqsu2}
\begin{split}
t^\ell(k^{1/2}) e^\ell_{p} &= q^{-p} e^\ell_{p}, \quad
t^\ell(e) e^\ell_{p} = q^{2-p} b^\ell(p) e^\ell_{p-1},  \quad
t^\ell(f) e^\ell_{p} = q^{p-1} b^\ell(p+1) e^\ell_{p+1}, \\
b^\ell(p) &= \frac{1}{q^{-1}-q} \sqrt{(q^{-\ell+p-1}-q^{\ell-p+1})(q^{-\ell-p}-q^{\ell+p})},
\end{split}
\end{equation}
where $t^\ell\colon \Uq(\mathfrak{su}(2))\to \text{End}(\mathcal{H}^{\ell})$ is the corresponding
representation. Note that $b^\ell(p)=b^\ell(1-p)$. 
Finally, recall that the centre $\mathcal{Z}(\Uq(\mathfrak{su}(2)))$ is generated by the 
Casimir element $\omega$,
\begin{equation}\label{eq:CasimirUqsu2}
\begin{split}
\omega &= \frac{q^{-1}k^{-1}+qk-2}{(q^{-1}-q)^2}+fe = \frac{qk^{-1}+q^{-1}k-2}{(q^{-1}-q)^2}+ef, \\
t^\ell(\omega) &= \left( \frac{q^{-\frac12 -\ell}-q^{\frac12 + \ell}}{q^{-1}-q}\right)^2 I. 
\end{split}
\end{equation}

We use the notation  $\Uq(\mathfrak{g})$ to denote the Hopf $\ast$-algebra 
$\Uq(\mathfrak{su}(2)\oplus\mathfrak{su}(2))$, 
which we identify with $\Uq(\mathfrak{su}(2))\otimes \Uq(\mathfrak{su}(2))$, where 
$K^{1/2}_i$, $K^{-1/2}_i$, $E_i$, $F_i$, $i=1,2$, are the generators. 
The relations \eqref{eq:defrelationsUqsl2} and \eqref{eq:defrelationsUqsl2-squareroots} hold 
with $(k^{1/2}, k^{-1/2},e,f)$ replaced by $(K^{1/2}_i, K^{-1/2}_i, E_i, F_i)$ for any fixed $i$
and the generators with different index $i$ commute. 
The tensor product of two Hopf $\ast$-algebras is again a Hopf $\ast$-algebra, where the maps on a simple
tensor $X_1\otimes X_2$ are given by, see e.g. \cite[Ch.~4]{CharP},
\begin{equation}\label{eq:defHopfstructureonUqg}
\begin{split}
&\Delta(X_1\otimes X_2) = \Delta_{13}(X_1)\Delta_{24}(X_2), \quad
\varepsilon(X_1\otimes X_2) = \varepsilon(X_1)\varepsilon(X_2), \\ 
&S(X_1\otimes X_2) = S(X_1)\otimes S(X_2), \quad
(X_1\otimes X_2)^\ast = X_1^\ast\otimes X_2^\ast,
\end{split}
\end{equation}
where we use leg-numbering notation. 

The irreducible finite dimensional type $1$ representations of $\Uq(\mathfrak{g})$ are labeled by $(\ell_1,\ell_2)\in \frac12\NN\times \frac12\NN$, and the 
representations $t^{\ell_1,\ell_2}$ from $\Uq(\mathfrak{g})$ to $\text{End}(\mathcal{H}^{\ell_1, \ell_2})$,  
$\mathcal{H}^{\ell_1, \ell_2} = \mathcal{H}^{\ell_1} \tensor \mathcal{H}^{\ell_2}$,
are obtained as the exterior
tensor product of the representations of spin $\ell_1$ and $\ell_2$ of $\Uq(\mathfrak{su}(2))$. 
Here type 1 means that the spectrum of $K_i^{1/2}$, $i=1,2$, is contained in $q^{\frac12\ZZ}$.

We have used the notation $\Delta$, $\varepsilon$, $S$ for the comultiplication, counit, and antipode for
all Hopf algebras. From the context it should be
clear which comultiplication, counit and antipode is meant. 
The corresponding dual Hopf $\ast$-algebra related to the quantised function algebra 
is not needed for the description of the results in Section \ref{sec:mainresults}, and it will
be recalled in Section \ref{ssec:MVSFonquantumgroup}.

\section[Matrix-valued orthogonal polynomials related to the quantum analogue of \\ ${(\SU(2) \times \SU(2),\text{diag})}$]{Matrix-valued orthogonal polynomials related to the quantum analogue of \texorpdfstring{${(\SU(2) \times \SU(2),\text{diag})}$}{(SU(2) x SU(2), SU(2))}}
\label{sec:mainresults}

In this section we state the main results of the paper. First we introduce the specific 
quantum symmetric pair, which is to be considered as the quantum analogue of a symmetric space $G/K$,
in this case $\SU(2) \times \SU(2)/\SU(2)$. 
Quantum symmetric spaces have been introduced and studied in detail by 
Letzter \cite{Letz03}, \cite{Letz04}, \cite{Letz08}, see also 
Kolb \cite{Kolb}, In particular, 
Letzter has shown that Macdonald polynomials occur as spherical functions on 
quantum symmetric pairs motivated by work by Koornwinder, Dijkhuizen, Noumi and others. 
In our case, $\mathcal{B}\subset \Uq(\mathfrak{g})$ as in Definition \ref{def:coidealB}, 
is the appropriate right coideal subalgebra. Using the explicit branching rules for 
$t^{(\ell_1,\ell_2)}\vert_{\mathcal B}$ of Theorem \ref{thm:cgc} 
we introduce matrix-valued spherical functions in Definition \ref{def:elementarysphericalfunction}.
To these matrix-valued spherical functions we associate matrix-valued polynomials 
in \eqref{eq:defpolsPn}, and we spend the remainder of this section to describe properties of these
matrix-valued polynomials. 
This includes the orthogonality relations, the three-term recurrence relation, and the matrix-valued polynomials
as eigenfunctions of a commuting set of matrix-valued $q$-difference equations of Askey-Wilson type. 
Moreover, we give two explicit descriptions of the matrix-valued weight function $W$, one in terms of 
spherical functions for this quantum symmetric pair, and one in terms of the LDU-decomposition. 
The LDU-decomposition gives the possibility to decouple  the matrix-valued $q$-difference 
operator, and this leads to an explicit expression for the matrix entries
of the matrix-valued orthogonal polynomials in terms of scalar-valued  orthogonal polynomials from
the $q$-Askey scheme in Theorem \ref{thm:explicit_Pn}. 

For the symmetric pair $(G,K)= (\SU(2) \times \SU(2), \SU(2))$, $K=\SU(2)$ corresponds to the 
fixed points of the Cartan involution $\theta$ flipping the order of the pairs in $G$. 
The quantised universal enveloping algebra 
associated to $G$ is $\Uq(\mathfrak{g})$ as introduced in 
Section \ref{sec:quantizeduniversalenvelopingalg}.
As the quantum analogue of $K$ we take the right coideal subalgebra 
$\mathcal{B}\subset \Uq(\mathfrak{g})$, i.e. $\mathcal{B}\subset \Uq(\mathfrak{g})$ is an algebra 
satisfying 
$\Delta(\mathcal{B})\subset \mathcal{B} \otimes \Uq(\mathfrak{g})$, 
as in Definition \ref{def:coidealB}. Letzter \cite[Section 7, (7.2)]{Letz03} has introduced the 
corresponding 
left coideal subalgebra, and we follow Kolb \cite[\S 5]{Kolb} in using right coideal subalgebras
for quantum symmetric pairs. 
Note that we have modified the generators slightly in order to have $B_1^\ast=B_2$.

\begin{definition}\label{def:coidealB}
The right coideal subalgebra  $\mathcal{B}\subset \Uq(\mathfrak{g})$ is the subalgebra 
generated by $K^{\pm 1/2}$, where $K=K_1K_2^{-1}$, and
\begin{equation*}
\begin{split}
B_1 &= q^{-1} K_1^{-1/2} K_2^{-1/2} E_1 + q F_2 K_1^{-1/2} K_2^{1/2}, \\
B_2 &= q^{-1} K_1^{-1/2} K_2^{-1/2} E_2 + q F_1 K_1^{1/2} K_2^{-1/2}.
\end{split}
\end{equation*}
\end{definition}

\begin{remark}\label{rmk:identification}
(i) $\mathcal{B}$ is a right coideal as follows from the general construction,
see \cite[Prop.~5.2]{Kolb}. It can be verified directly by checking it for the generators.  
Note $\Delta(K^{\pm 1/2})=K^{\pm 1/2}\otimes K^{\pm 1/2}$ is immediate, and 
\begin{equation*}
\Delta(B_1)= B_1 \otimes (K_1K_2)^{-1/2} + K^{1/2}\otimes q^{-1}(K_1K_2)^{-1/2} E_1 + K^{-1/2} \otimes q F_2 K^{-1/2}
\end{equation*}
is in $\mathcal{B} \tensor \Uq(\mathfrak{g})$ by a straightforward calculation. Since $B_2=B_1^\ast$, it also follows for $B_2$, since 
$K^{\pm 1/2}$ are self-adjoint. 
The relations, cf. \cite[Lemma~5.15]{Kolb}, 
\begin{equation}\label{eq:relationsinB}
K^{1/2} B_1 = q B_1 K^{1/2}, \quad
K^{1/2} B_2 = q^{-1} B_2 K^{1/2}, \quad
[B_1, B_2] = \cfrac{K - K^{-1}}{q - q^{-1}},
\end{equation}
hold in $\Uq(\mathfrak{g})$ as can also be checked directly. 

(ii) Let $\Psi\colon \Uq(\mathfrak{su}(2)) \to \Uq(\mathfrak{su}(2))$ be defined by 
\begin{equation*}
\Psi(k^{1/2}) = k^{-1/2},\quad  \Psi(k^{-1/2}) = k^{1/2}, \quad
\Psi(e)=q^3 f, \qquad \Psi(f) = q^{-3}e,
\end{equation*}
then $\Psi$ extends to an involutive $\ast$-algebra isomorphism. Consider the map 
\[
\iota\circ (\Psi \otimes \text{Id})\circ \Delta \colon \Uq(\mathfrak{su}(2))\to \Uq(\mathfrak{g})
\]
where $\iota$ is the algebra morphism mapping $x\otimes y\in \Uq(\mathfrak{su}(2))\otimes \Uq(\mathfrak{su}(2))$ 
to the corresponding element $X_1Y_2\in \Uq(\mathfrak{g})$ for $x$ and $y$ generators of $\Uq(\mathfrak{su}(2))$.
Then we see that $k^{1/2} \mapsto K^{-1/2}$, 
$qfk^{1/2}\mapsto B_1$, and $q^{-1}k^{-1/2}e \mapsto B_2$ under the map 
$\iota\circ (\Psi \otimes \text{Id})\circ \Delta$. 
In particular, the relations \eqref{eq:relationsinB} follow.
We conclude that $\mathcal{B}$ is isomorphic as a $\ast$-algebra to 
$\Delta(U_q(\mathfrak{su}(2))\subset \Uq(\mathfrak{g})$ by the 
$\ast$-isomorphism $\iota \circ (\Psi \otimes \text{Id})$.

(iii) In particular, $\mathcal{B}\cong \Uq(\mathfrak{su}(2))$ as 
$\ast$-algebras. So the irreducible type $1$ representations 
of $\mathcal{B}$ are labeled by the spin $\ell\in \frac12\NN$. This can be made explicit by 
$t^\ell\colon \BB \to \text{End}(\mathcal{H}^\ell)$ and setting
\begin{equation}\label{eq:representationsB}
t^\ell(K^{1/2}) e^\ell_{p} = q^{-p} e^\ell_{p}, \quad
t^\ell(B_1) e^\ell_{p} =  b^\ell(p) e^\ell_{p-1},  \quad
t^\ell(B_2) e^\ell_{p} = b^\ell(p+1) e^\ell_{p+1} 
\end{equation}
with the notation of \eqref{eq:defrepresentationUqsu2}. We use the same notation $t^\ell$ for these representations here
and in \eqref{eq:defrepresentationUqsu2}, since they correspond under the identification of $\BB$ as $\Uq(\mathfrak{su}(2))$. 

(iv) Let $\sigma$ be the $\ast$-algebra isomorphism on $\Uq(\mathfrak{g})= \Uq(\mathfrak{su}(2))\otimes\Uq(\mathfrak{su}(2))$
by flipping the order in the tensor product, or equivalently by flipping the subscripts $1\leftrightarrow 2$. 
Then $\sigma \colon \mathcal{B} \to \mathcal{B}$ is an involution $B_1\leftrightarrow B_2$, $K\leftrightarrow K^{-1}$. 
On the level of representations of $\Uq(\mathfrak{g})$ and $\mathcal{B}$ it follows 
$t^{\ell_1,\ell_2}(\sigma(X)) = P^\ast t^{\ell_2,\ell_1}(X) P$, $X\in \Uq(\mathfrak{g})$,  where $P\colon \mathcal{H}^{\ell_1,\ell_2} \to \mathcal{H}^{\ell_2,\ell_1}$
is the flip, and $t^{\ell}(\sigma(Y)) = (J^\ell)^\ast t^{\ell}(Y)J^\ell$, $Y\in \mathcal{B}$, where 
$J^\ell \colon \mathcal{H}^\ell\to \mathcal{H}^\ell$, $J^\ell\colon e^\ell_p \mapsto e^\ell_{-p}$.
\end{remark}

\begin{theorem} \label{thm:cgc}
The finite dimensional representation $t^{\ell_1, \ell_2}$ of $\Uq(\mathfrak{g})$ restricted to $\mathcal{B}$ 
decomposes multi\-pli\-ci\-ty-free into irreducible representations $t^{\ell}$ of $\mathcal{B}$;
\begin{align*}
t^{\ell_1, \ell_2}\vert_{\BB}
  &\simeq \bigoplus_{\ell = |\ell_1 - \ell_2|}^{\ell_1 + \ell_2} t^{\ell}, &
\mathcal{H}^{\ell_1, \ell_2} 
	&\simeq \bigoplus_{\ell = |\ell_1 - \ell_2|}^{\ell_1 + \ell_2} \mathcal{H}^{\ell}.
\end{align*}
With respect to the orthonormal basis $\{e^{\ell}_p\}_{p=-\ell}^\ell$ of $\mathcal{H}^{\ell}$ and  
the orthogonal basis $\{ e^{\ell_1}_{i} \tensor e^{\ell_2}_{j} \}_{i=-\ell_1, j=-\ell_2}^{\ell_1,\ell_2}$ for $\mathcal{H}^{\ell_1, \ell_2}$ 
the $\mathcal{B}$-intertwiner $\beta^{\ell}_{\ell_1, \ell_2} \colon \mathcal{H}^{\ell} \rightarrow \mathcal{H}^{\ell_1, \ell_2}$ 
is given by
\begin{align*} 
\beta^{\ell}_{\ell_1, \ell_2} \colon e^{\ell}_p \mapsto 
  \sum_{i = -\ell_1}^{\ell_1} \sum_{j = -\ell_2}^{\ell_2}
  C^{\ell_1, \ell_2, \ell}_{i, j, p} 
	e^{\ell_1}_i \tensor e^{\ell_2}_{j},
\end{align*}
where $C^{\ell_1, \ell_2, \ell}_{i, j, p}$ are Clebsch-Gordan coefficients satisfying 
$C^{\ell_1, \ell_2, \ell}_{i, j, p} = 0$ if $i - j \neq p$.
\end{theorem}

The proof of Theorem \ref{thm:cgc} is a reduction to the well-known Clebsch-Gordan decomposition for the 
quantised universal enveloping algebra $\Uq(\mathfrak{su}(2))$, see e.g. \cite{CharP}, \cite{KlimS}, using 
Remark \ref{rmk:identification}. 
The proof is presented in Appendix \ref{app:cgc}.  In particular, $(\beta^{\ell}_{\ell_1, \ell_2})^\ast \beta^{\ell}_{\ell_1, \ell_2}$
is the identity on $\mathcal{H}^\ell$.
Note that 
\begin{equation*}
(\beta^{\ell}_{\ell_1, \ell_2})^\ast \colon \mathcal{H}^{\ell_1, \ell_2}
\rightarrow \mathcal{H}^{\ell}, \qquad
e^{\ell_1}_{n_1} \otimes e^{\ell_2}_{n_2} \mapsto 
\sum_{p=-\ell}^\ell C^{\ell_1, \ell_2, \ell}_{n_1, n_2, p} \, e^\ell_p.
\end{equation*}

In general the decomposition of an irreducible representation restricted to a right coideal subalgebra seems a difficult problem.
In this particular case we can reduce to the Clebsch-Gordan decomposition, and yet 
another known special case is by  Oblomkov and Stokman \cite[Prop.~1.15]{ObloS}, but in 
general this is an open problem.

In particular, for fixed $\ell\in \frac12\NN$, we have $[t^{\ell_1, \ell_2}\vert_{\BB}\colon t^\ell]=1$ if
and only if 
\begin{equation}\label{eq:conditionsonl1l2}
(\ell_1,\ell_2)\in \frac12\NN\times\frac12\NN, \qquad|\ell_1-\ell_2|\leq \ell \leq \ell_1+\ell_2,
\quad \ell_1+\ell_2-\ell\in \ZZ.
\end{equation}
We use the reparametrisation of \eqref{eq:conditionsonl1l2} by 
\begin{equation}\label{eq:defxi}
\xi = \xi^{\ell} \colon \NN \times \{0, 1, \ldots, 2\ell\} \to \frac{1}{2} \NN \times \frac{1}{2} \NN, \quad 
\xi(n, k) = \left( \cfrac{n + k}{2}, \ell + \cfrac{n - k}{2} \right),
\end{equation}
see also Figure \ref{fig:parametrization} and \cite[Fig. 1, 2]{KoelvPR12}.
In case $\ell=0$, we have $t^0=\varepsilon\vert_{\BB}$, where $\varepsilon$ is the counit of 
$\Uq(\mathfrak{g})$, is the trivial representation of $\BB$, and the condition  
\eqref{eq:conditionsonl1l2} gives $\ell_1=\ell_2$ and $\xi^0(n,0)=(\frac12 n, \frac12 n)$.

With these preparations we can introduce the matrix-valued spherical functions associated 
to a fixed representation $t^\ell$ of $\BB$, where we use the notation of Theorem \ref{thm:cgc}.

\begin{definition}\label{def:elementarysphericalfunction}
Fix $\ell \in \frac{1}{2}\NN$ and let $(\ell_1, \ell_2) \in \frac{1}{2}\NN \times \frac{1}{2}\NN$ 
so that $[t^{\ell_1, \ell_2}\vert_{\BB}\colon t^\ell]=1$.
The spherical function of type $\ell$ associated to $(\ell_1, \ell_2)$ is defined by
\begin{align*}
\Phi^{\ell}_{\ell_1, \ell_2} \colon \Uq(\mathfrak{g}) \to \End(\mathcal{H}^{\ell}), \qquad
 Z \mapsto (\beta^{\ell}_{\ell_1, \ell_2})^\ast \circ t^{\ell_1, \ell_2}(Z) \circ \beta^{\ell}_{\ell_1, \ell_2}.
\end{align*}
\end{definition}

\begin{remark}\label{rmk:def:elementarysphericalfunction}
(i) Note that the requirement on $(\ell_1,\ell_2)$ in
Definition \ref{def:elementarysphericalfunction}  corresponds to the condition \eqref{eq:conditionsonl1l2}.
Since $\beta^{\ell}_{\ell_1, \ell_2}$ is a $\BB$-intertwiner, we have 
\begin{equation}\label{eq:invarianceproperties}
\Phi^{\ell}_{\ell_1, \ell_2}(XZY) = t^\ell(X) \Phi^{\ell}_{\ell_1, \ell_2}(Z) t^\ell(Y), 
\qquad \forall\, X,Y\in \BB, \ \forall\, Z\in \Uq(\mathfrak{g}).
\end{equation}

(ii) Note that the condition \eqref{eq:conditionsonl1l2} is symmetric in $\ell_1$ and $\ell_2$.
With the notation of Remark \ref{rmk:identification}(iv) we have $\Phi^{\ell}_{\ell_2, \ell_1}(Z)
= J^\ell \Phi^{\ell}_{\ell_1, \ell_2}(\sigma(Z)) J^\ell$ for $Z\in \Uq(\mathfrak{g})$. 
This follows from $\beta^{\ell}_{\ell_2, \ell_1} = P \beta^{\ell}_{\ell_1, \ell_2} J^\ell$, 
which is a consequence of \eqref{eq:CGCsymml1l2}.
\end{remark}

In case $\ell = 0$, $\mathcal{H}^0\cong \CC$, we need $\ell_1 = \ell_2$. Then $\Phi^0_{\ell_1,\ell_1}$ are 
linear maps $\Uq(\mathfrak{g}) \to \CC$. In particular,   $\Phi^{0}_{0,0}$ equals the counit $\varepsilon$, and  
the spherical function $\varphi  = \frac12(q^{-1}+q) \Phi^{0}_{1/2, 1/2}$ is scalar-valued linear map on $\Uq(\mathfrak{g})$. 
The elements $\Phi^{0}_{n/2, n/2}$ can be written as a multiple of $U_n(\varphi)$, 
where $U_n$ denotes the 
Chebyshev polynomial of the second kind of degree $n$, see Proposition \ref{prop:sphericalcaseasChebyshevpols}. 
This statement can be considered as a special case of 
Theorems \ref{thm:ortho},
but we need the identification with the Chebyshev polynomials in the spherical case $\ell=0$ 
in order to obtain the weight function in Theorem \ref{thm:ortho}. 
Proposition \ref{prop:sphericalcaseasChebyshevpols} will follow from Theorem \ref{thm:spherical_recurrence}. 
The identification of the spherical functions for $\ell=0$ with Chebyshev polynomials 
corresponds to the classical case, since the spherical functions on $G\times G/G$ are the characters on $G$ and the 
characters on $\SU(2)$ are Chebyshev polynomials of the second kind, as the simplest case of the Weyl character formula. 
It also corresponds to the computation of the characters on the quantum $\SU(2)$ group by Woronowicz \cite{Woro}, 
since the characters are identified with Chebyshev polynomials as well.

Next Theorem \ref{thm:spherical_recurrence} gives the  possibility to associate polynomials in $\varphi$ to 
spherical functions of Definition \ref{def:elementarysphericalfunction}. Theorem \ref{thm:spherical_recurrence} 
essentially follows from 
the tensor product decomposition of representations of $\Uq(\mathfrak{g})$, which in turn follows from tensor product 
decomposition for $\Uq(\mathfrak{su}(2))$, and some explicit knowledge of Clebsch-Gordan coefficients.

\begin{theorem} \label{thm:spherical_recurrence}
Fix $\ell \in \frac{1}{2}\NN$ and let $(\ell_1, \ell_2) \in \frac{1}{2}\NN \times \frac{1}{2}\NN$ 
satisfy \eqref{eq:conditionsonl1l2}, then for constants $A_{i, j}$ we have 
\begin{align*}
\varphi \Phi^{\ell}_{\ell_1, \ell_2}  
  = \sum_{i, j = \pm 1/2} A_{i, j} \Phi^{\ell}_{\ell_1 + i, \ell_2 + j}, \qquad A_{1/2, 1/2} \neq 0.
\end{align*}
\end{theorem}

In order to interpret the result of Theorem \ref{thm:spherical_recurrence} we evaluate both sides at an 
arbitrary $X\in \Uq(\mathfrak{g})$. The right hand side is a linear combination of linear maps from
$\mathcal{H}^\ell$ to itself after evaluating at $X$. For the left hand side we use the pairing of Hopf algebras, 
so that multiplication and comultiplication are dual to each other and the left hand side has to be interpreted
as
\begin{equation}\label{eq:interpretationphitimesPhi}
\Bigl(\varphi \Phi^{\ell}_{\ell_1, \ell_2}\Bigr)(X) = 
\sum_{(X)} \varphi(X_{(1)})\, \Phi^{\ell}_{\ell_1, \ell_2}(X_{(2)}) \in \text{End}(\mathcal{H}^\ell),
\end{equation}
which is a linear combination of linear maps from
$\mathcal{H}^\ell$ to itself, 
using $\Delta(X)= \sum_{(X)} X_{(1)}\otimes X_{(2)}$. 
The convention in Theorem \ref{thm:spherical_recurrence} is that $A_{i,j}$ is zero in case 
$(\ell_1+i,\ell_2+j)$ does not satisfy \eqref{eq:conditionsonl1l2}. 
The proof of Theorem \ref{thm:spherical_recurrence} can be found in Section \ref{subsec:sph_rec}. 

Since $\BB$ is a right coideal subalgebra we see that the left hand side of Theorem \ref{thm:spherical_recurrence} has the 
same transformation behaviour as \eqref{eq:invarianceproperties}. Indeed, for $X\in \BB$ and $Y\in \Uq(\mathfrak{g})$ we
have 
\begin{equation}\label{eq:leftinvariancepropertyproductThmsphericalrecurrrence}
\begin{split}
\Bigl(\varphi \Phi^{\ell}_{\ell_1, \ell_2}\Bigr)(XY) 
  &= \sum_{(X),(Y)} \varphi(X_{(1)}Y_{(1)})\, \Phi^{\ell}_{\ell_1, \ell_2}(X_{(2)}Y_{(2)}) \\ 
  &= \sum_{(X),(Y)} \varepsilon(X_{(1)})\varphi(Y_{(1)})\, \Phi^{\ell}_{\ell_1, \ell_2}(X_{(2)}Y_{(2)}) \\
  &= \sum_{(Y)} \varphi(Y_{(1)})\, \Phi^{\ell}_{\ell_1, \ell_2}(\sum_{(X)} \varepsilon(X_{(1)})X_{(2)}Y_{(2)}) \\
  &= \sum_{(Y)} \varphi(Y_{(1)})\, \Phi^{\ell}_{\ell_1, \ell_2}(XY_{(2)}) \\
  &=  \sum_{(Y)} \varphi(Y_{(1)})\, t^\ell(X) \Phi^{\ell}_{\ell_1, \ell_2}(Y_{(2)})
  =  t^\ell(X) \Bigl(\varphi \Phi^{\ell}_{\ell_1, \ell_2}\Bigr)(Y),
\end{split}
\end{equation}
where we have used that $X_{(1)}\in \BB$ by the right coideal property and \eqref{eq:invarianceproperties} for $\varphi= \Phi^0_{1/2,1/2}$ in 
the second equality, and the counit axiom $\sum_{(X)} \varepsilon(X_{(1)})X_{(2)}=X$ in the fourth equality and then 
\eqref{eq:invarianceproperties} for $\Phi^\ell_{\ell_1,\ell_2}$ and the fact that $\varphi(Y_{(1)})$ is a scalar. 
Similarly, the invariance property from the right can be proved.

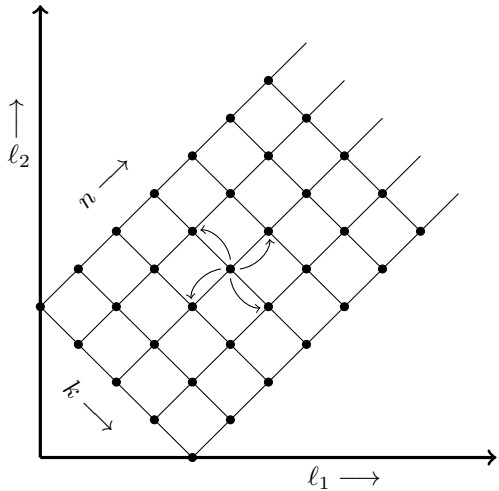
\begin{figure}[ht]
\begin{center}
\begin{tikzpicture}[scale=1.0]

\draw[very thick,->] (0,0) -- (6,0) ;
\draw[very thick,->] (0,0) -- (0,6) ;

\foreach \d in {0,...,6} {
  \foreach \k in {0,...,4} {
    \path[draw, fill, color=black] (\k*0.5+\d*0.5, 2-\k*0.5+\d*0.5) circle (1.5pt);
    \path[draw, color=black] (\k*0.5+\d*0.5, 2-\k*0.5+\d*0.5) -- (\k*0.5+\d*0.5+0.5, 2.5-\k*0.5+\d*0.5);
  }
  \foreach \k in {0,...,3} {
    \path[draw, color=black] (\k*0.5+\d*0.5, 2-\k*0.5+\d*0.5) -- (\k*0.5+\d*0.5+0.5, 1.5-\k*0.5+\d*0.5);
  }
}

\node[above, rotate=-45] at (0.5, 0.5) {$k \longrightarrow$};
\node[above, rotate=45] at (1, 3.5) {$n \longrightarrow$};

\node[below] at (4.0, 0) {$\ell_1 \longrightarrow$};
\node[left] at (0, 4.0) {$\ell_2$};
\node[above, rotate=90] at (-0.1, 4.5) {$\longrightarrow$};

\node at (2.5, 2.5) {}
  edge[->, shorten >= 3pt, bend right = 45] (3.0, 3.0)
  edge[->, shorten >= 3pt, bend right = 45] (3.0, 2.0)
  edge[->, shorten >= 3pt, bend right = 45] (2.0, 2.0)
  edge[->, shorten >= 3pt, bend right = 45] (2.0, 3.0);
\end{tikzpicture}
\end{center}
\caption{The spherical functions $\Phi^{\ell}_{\ell_1, \ell_2}$ when $\ell = 2$ and interpretation of
$\varphi \cdot \Phi^{\ell}_{5/2, 5/2}$ in terms of the matrix-valued spherical functions. 
The reparametrization $\xi$ is depicted.}
\label{fig:parametrization}
\end{figure}

Theorem \ref{thm:spherical_recurrence} leads to polynomials in $\varphi$ by iterating the result and using that 
$A_{1/2,1/2}$ is non-zero. 

\begin{corollary} \label{cor:thm:spherical_recurrence}
There exist $2\ell + 1$ polynomials $r^{\ell, k}_{n,m}$, $0 \leq k \leq 2\ell$, 
of degree at most $n$ so that
\begin{equation*}
\Phi^{\ell}_{\xi(n,m)} = \sum_{k = 0}^{2\ell} r^{\ell, k}_{n,m}(\varphi)\, \Phi^{\ell}_{\xi(0, k)},
\qquad n\in \NN, \ 0 \leq m \leq 2\ell.
\end{equation*}
\end{corollary}

The aim of the paper is to show that the polynomials $r^{\ell,k}_{n,m}$ give rise to matrix-valued orthogonal polynomials. Put 
\begin{equation}\label{eq:defpolsPn}
P_n = P^\ell_n \in \End(\mathcal{H}^\ell)[x] \qquad
(P_n)_{i,j}=\overline{r^{\ell,i}_{n,j}}, \qquad 0\leq i,j\leq 2\ell, 
\end{equation}
where the matrix-valued polynomials $P_n$ are taken with respect to the relabeled standard basis 
$e_p=e^\ell_{p-\ell}$, $p\in \{0,1,\cdots,2\ell\}$, so that 
$P_n = \sum_{i,j=0}^{2\ell}  \overline{r^{\ell,i}_{n,j}} \otimes E_{i,j}$. 
From Corollary \ref{cor:monic-3-term} or Theorem \ref{thm:explicit_Pn} we see that 
the polynomial $r^{\ell,i}_{n,j}$ has real coefficients. 
The case $\ell=0$ corresponds to a three-term recurrence relation for (scalar-valued) orthogonal polynomials, 
and then the polynomials coincide with the Chebyshev polynomials $U_n$ 
viewed as a subclass of Askey-Wilson polynomials \cite[(2.18)]{AskeW}, see Proposition \ref{prop:sphericalcaseasChebyshevpols}. 

We show that the matrix-valued polynomials $(P_n)_{n=0}^\infty$ are orthogonal with respect to an explicit 
matrix-valued weight function $W$, see Theorem \ref{thm:ortho}, arising from the Schur orthogonality relations. 
The expansion of the entries of weight function in terms of Chebyshev polynomials is given by quantum group theoretic 
considerations except for the calculation of the coefficients in this expansion. 
The matrix-valued orthogonal polynomials satisfy a matrix-valued three-term recurrence relation as follows from 
Theorem \ref{thm:spherical_recurrence}, which in turn is a consequence of the decomposition of tensor product representations of 
$\Uq(\mathfrak{g})$. 
However, in order to determine the matrix coefficients in the matrix-valued three-term recurrence we use analytic methods.
The existence of two Casimir elements in $\Uq(\mathfrak{g})$ leads to the matrix-valued orthogonal polynomials 
being eigenfunctions of two commuting matrix-valued $q$-difference operators, see \cite{KoelvPR13} for the group case. 
This extends Letzter \cite{Letz04} to the matrix-valued set-up for this particular case.
The $q$-difference operators are the key to determining the entries of the matrix-valued orthogonal 
polynomials explicitly in terms of scalar-valued orthogonal polynomials from the $q$-Askey-scheme \cite{KoekLS}, 
namely the continuous $q$-ultraspherical polynomials and the $q$-Racah polynomials.
In this deduction the LDU-decomposition of the matrix-valued weight function $W$ is essential, 
since the conjugation with $L$ allows us to decouple the matrix-valued $q$-difference operator. 

In the remainder of Section \ref{sec:mainresults} we state these results explicitly, and we present the proofs in the remaining sections. 
First we give the main statements which essentially 
follow from the quantum group theoretic set-up, except for explicit calculations, 
and these are Theorems \ref{thm:ortho}, \ref{thm:monic-3-term}, \ref{thm:diff_eqn_Pn}. 
The remaining Theorems \ref{thm:ldu}, \ref{thm:explicit_Pn} are obtained using 
scalar orthogonal polynomials from the $q$-analogue of the Askey scheme \cite{KoekLS} and
transformation and summation formulas for basic hypergeometric series \cite{GaspR}.

We start by stating that the matrix-valued polynomials $(P_n)_{n=0}^\infty$ introduced in \eqref{eq:defpolsPn}  are 
orthogonal with the conventions of Section \ref{sec:genMVOP}. 
The orthogonality relations of Theorem \ref{thm:ortho} are due to the Schur orthogonality relations. The expansion of the entries of 
the weight function in terms of Chebyshev polynomials follows from the fact that the entries are spherical functions, i.e. 
correspond to the case $\ell=0$ so that they are polynomial in $\varphi$. 
The non-zero entries follow by considering tensor product decompositions, but the explicit values for the coefficients $\alpha_t(m,n)$ in
Theorem \ref{thm:ortho} require summation and transformation formulae for basic hypergeometric series. 

\begin{theorem} \label{thm:ortho}
The polynomials $(P_n)_{n=0}^\infty$ of \eqref{eq:defpolsPn} form a family of 
matrix-valued orthogonal polynomials, so that $P_n$ is of degree $n$ with non-singular leading coefficient. 
The orthogonality for the matrix-valued polynomials $(P_n)_{n \geq 0}$ is given by
\begin{align*}
\frac{2}{\pi} \int_{-1}^1 P_m(x)^* W(x) P_n(x) \sqrt{1 - x^2} dx &= G_m \delta_{m, n},
\end{align*}
where the squared norm matrix $G_m$ is diagonal;
\begin{align*}
(G_n)_{i, j} &= \delta_{i, j} q^{2n - 2\ell}
\frac{  (1-q^{4\ell+2})^2}{(1-q^{2n+2i+2})(1-q^{4\ell-2i+2n+2})}.
\end{align*}
Moreover, for $0 \leq m \leq n \leq 2\ell$ the weight matrix is given by 
\begin{align*}
W(x)_{m, n} &= \sum_{t = 0}^n \alpha_t(m, n)\, U_{m + n - 2t}(x),
\end{align*}
where
\begin{align*} 
\alpha_t(m, n) &= 
  q^{
    2n(2\ell+1) - n^2 - (4\ell + 3)t + t^2 - 2\ell + m
  }
  \frac{1 - q^{4\ell + 2}}{1 - q^{2m + 2}}
  \frac{(q^2; q^2)_{2\ell - n} (q^2; q^2)_{n}}{(q^2; q^2)_{2\ell}} \nonumber \\
  &\qquad \times
  \frac{(-1)^{m - t} (q^{2m - 4\ell}; q^{2})_{n - t}}{(q^{2m + 4}; q^2)_{n-t}}
  \frac{(q^{4\ell + 4 - 2t}; q^2)_{t}}{(q^2; q^2)_{t}},
\end{align*}
and $W(x)_{m, n} = W(x)_{n, m}$ if $m > n$.
\end{theorem}

The proof of Theorem \ref{thm:ortho} proceeds in steps. First we study explicitly the case $\ell=0$, motivated by 
works by Koornwinder \cite{Koor1993}, Letzter \cite{Letz03}, \cite{Letz04}, \cite{Letz08}, and others.
Secondly, we show that taking traces of a 
matrix-valued spherical function of type $\ell$ associated to $(\ell_1,\ell_2)$ times the 
adjoint of a spherical function of type $\ell$ associated to $(\ell_1',\ell_2')$ gives, up to an
action by an invertible group-like element of $\Uq(\mathfrak{g})$, 
a polynomial in the generator for the case $\ell=0$. Then the explicit expression of the Haar functional 
on this polynomial algebra, stated in Lemma \ref{lemma:sphericalelements}, gives the 
matrix-valued orthogonality relations. Finally, the explicit expression for the weight is obtained by analysing the 
explicit expression of $W$ in terms of the matrix entries of the intertwiners $\beta^\ell_{\ell_1,\ell_2}$ in case $\ell_1+\ell_2=\ell$. These matrix entries are Clebsch-Gordan coefficients. 

The leading coefficient of $P_n$ can be calculated explicitly from the proof of Theorem \ref{thm:ortho}.

\begin{corollary} \label{cor:ortho-cgc-2}
The leading coefficient of $P_n$ is a non-singular diagonal matrix;
\begin{align*}
\lc(P_n)_{i,j} &= \delta_{i,j} 2^n q^n \frac{
  (q^{2i + 2}, q^{4\ell - 2i + 2}; q^2)_{n}
}{
  (q^2, q^{4\ell + 4}; q^2)_{n}
}.
\end{align*}
\end{corollary}

The weight $W$ is not irreducible, see Section \ref{sec:genMVOP}, but splits into two irreducible block matrices.
The symmetry $J$ of the weight function of Theorem \ref{thm:ortho} is essentially a consequence of 
Remark \ref{rmk:def:elementarysphericalfunction}(ii), but we need the explicit expression of the weight in order to prove that 
the commutant algebra is not bigger, see also \cite[\S 4]{KR15}.

\begin{proposition} \label{prop:commutant}
The commutant algebra
\begin{align*}
\{ W(x) \mid x \in [-1, 1] \}'
  &= \{ Y \in \End(\mathcal{H}^{\ell}) \mid W(x) Y = Y W(x), \forall x \in (-1, 1) \},
\end{align*}
is spanned by $I$ and $J$, where $J\colon e_p\mapsto e_{2\ell-p}$, $p\in \{0,\cdots, 2\ell\}$,
is a self-adjoint involution.
Then $JP_n(x)J = P_n(x)$ and $JG_nJ = G_n$. 
Moreover, the weight $W$ decomposes into two irreducible block matrices $W_+$ and $W_-$, 
where $W_+$, respectively $W_-$, acts in the $+1$-eigenspace, respectively $-1$-eigenspace, of $J$.
So for $P_++P_-=I$, where $P_+$, $P_-$ are the orthogonal self-adjoint projections $P_+=\frac12(I+J)$, 
$P_-=\frac12(I-J)$, we have that $W_+$, respectively $W_-$, corresponds to $P_+W(x)P_+$, 
respectively $P_-W(x)P_-$, restricted
to the $+1$-eigenspace, respectively $-1$-eigenspace, of $J$.
\end{proposition}

The special cases for $\ell=\frac12$ and $\ell=1$ are given at the end of this section. 
In particular, we identify all scalar-valued orthogonal polynomials occurring in this framework 
explicitly in terms of Askey-Wilson polynomials.

Theorem \ref{thm:spherical_recurrence} can be used to find a three-term recurrence relation for the 
matrix-valued orthogonal polynomials $P_n$, cf. Section \ref{sec:genMVOP}, so the underlying tensor product
decompositions provide the three-term recurrence relation. 
However, the resulting expressions for the entries of the coefficients of the matrices are 
rather complicated expressions in terms of Clebsch-Gordan coefficients. 
For the corresponding matrix-valued monic polynomials $Q_n(x) = P_n(x) \lc(P_n)^{-1}$, see 
Corollary \ref{cor:ortho-cgc-2} for the explicit expression for the leading coefficient, 
we can derive a simple expression for the matrices in the three-term recurrence relation once we 
have obtained more explicit expressions for the matrix entries of $Q_n$.
This is obtained in Section \ref{sec:qdifferenceoperators} using an explicit link of the matrix
entries to scalar orthogonal polynomials in the $q$-Askey scheme. 

\begin{theorem} \label{thm:monic-3-term}
The monic matrix-valued orthogonal polynomials $(Q_n)_{n \geq 0}$ satisfy the three-term recurrence relation
\begin{align*}
xQ_n(x) = Q_{n+1}(x) + Q_n(x) X_n + Q_{n-1}(x) Y_n, 
\end{align*}
where $Q_{-1}(x) = 0$, $Q_0(x) = I$ and
\begin{align*}
X_n &= 
  \sum_{i=0}^{2\ell-1} 
  \frac{
    q^{2n+1} (1-q^{2i+2})^2 (1-q^{4\ell+2n+2})^2
  }{
    2 (1-q^{2i+2n})(1-q^{4\ell+2n-2i})(1-q^{2n+2i+2})(1-q^{4\ell-2i+2n+2})
  } E_{i,i+1} \\
  &+ \sum_{i=1}^{2\ell}
    \frac{
      q^{2n+1} (1-q^{2n})^2(1-q^{4\ell+2n+2})^2
    }{
      2 (1-q^{2n+2i})(1-q^{4\ell+2n-2i})(1-q^{2n+2i+2})(1-q^{4\ell-2i+2n+2})} E_{i,i-1}, \\
Y_n &=
  \sum_{i=0}^{2\ell}
  \frac{1}{4}
  \frac{
    (1-q^{2n})^2(1-q^{4\ell+2n+2})^2
  }{
    (1-q^{2n+2i})(1-q^{4\ell+2n-2i})(1-q^{2n+2i+2})(1-q^{4\ell-2i+2n+2})
  } E_{i,i}.
\end{align*} 
\end{theorem}

Note that $X_n\to 0$, $Y_n\to \frac14$ as $n\to \infty$. 

The three-term recurrence relation for the matrix-valued orthogonal polynomials 
$P_n$ is given in Corollary \ref{cor:monic-3-term}, which follows from
Theorem \ref{thm:monic-3-term}, since we have
$G_{n+1}A_n=\lc(P_{n+1})^\ast \lc(P_n)$, $G_nB_n= \lc(P_n)^\ast X_n \lc(P_n)$, 
and $G_{n-1}C_n=\lc(P_{n-1})^\ast Y_n\lc(P_n)$.
For 
future reference we give the explicit expressions in Corollary \ref{cor:monic-3-term}. 

\begin{corollary}\label{cor:monic-3-term}
The matrix-valued orthogonal polynomials $(P_n)_{n \geq 0}$ satisfy the three term recurrence relation
\begin{align*}
x P_n(x) = P_{n+1}(x) A_n + P_{n}(x) B_n + P_{n-1}(x) C_n,
\end{align*}
where $P_{-1}(x) = 0$, $P_0(x) = I$ and
\begin{align*}
A_n &= \sum_{i = 0}^{2\ell} \frac{1}{2q}
  \frac{
    (1 - q^{2n + 2})(1 - q^{4\ell + 2n + 4})
  }{
    (1 - q^{2i + 2n + 2})(1 - q^{4\ell - 2i + 2n + 2})
  } E_{i, i}, \\
B_n &= 
  \sum_{i=0}^{2\ell-1}
    \frac{q^{2n+1}}{2}
    \frac{
      (1-q^{4\ell-2i})(1-q^{2i+2})
    }{
      (1-q^{4\ell+2n-2i})(1-q^{2n+2i+4})
    } E_{i,i+1} \\
  &\qquad + \sum_{i=1}^{2\ell}
    \frac{q^{2n+1}}{2}
    \frac{
      (1-q^{2i})(1-q^{4\ell-2i+2})
    }{
      (1-q^{2n+2i})(1-q^{4\ell-2i+2n+4})
    } E_{i,i-1}, \\
C_n &= \sum_{i=0}^{2\ell}
  \frac{q}{2}
  \frac{
    (1-q^{2n})(1-q^{4\ell+2n+2})
  }{
    (1-q^{2n+2i+2})(1-q^{4\ell+2n-2i+2})
  } E_{i,i}.
\end{align*}
\end{corollary}

Note that the case $\ell=0$ gives a three-term recurrence relation that can be solved in terms
of the Chebyshev polynomials, see Proposition \ref{prop:sphericalcaseasChebyshevpols}.

In the group case, the spherical functions are eigenfunctions of $K$-invariant differential operators on $G/K$,
see e.g. \cite{CassM}, \cite{GangV}. For matrix-valued spherical functions this is also the case, see 
\cite{Tira77}, and this has been exploited in the special cases studied in \cite{GrunPT}, \cite{KoelvPR12},
\cite{KoelvPR13}. In the quantum group case the action of the Casimir operator gives rise to a $q$-difference operator 
for the corresponding spherical functions, see \cite{Letz04}. The first occurrence of an Askey-Wilson $q$-difference operator, see \cite{AskeW}, \cite{GaspR}, \cite{Isma}, in this context is due to Koornwinder \cite{Koor1993}. 
For the matrix-valued orthogonal polynomials we have a matrix-valued analogue of the 
Askey-Wilson $q$-difference operator, as given in Theorem \ref{thm:diff_eqn_Pn}. 
We obtain two of these operators, one arising from the Casimir operator for $\Uq(\mathfrak{su}(2))$ 
in the first leg of $\Uq(\mathfrak{g})$ and one from the $\Uq(\mathfrak{su}(2))$ Casimir operator of the second leg. 
This is related to a kind of Cartan decomposition of $\Uq(\mathfrak{g})$, cf. \eqref{eq:invarianceproperties}, which, however, does not exist
in general for quantised universal enveloping algebras. We can still resolve this problem using techniques
based on \cite[\S 2]{CassM}, see the first part of the proof in Section \ref{sec:quantumgrouprelatedpropssphericalf}. The proof of Theorem \ref{thm:diff_eqn_Pn} is 
completed in Section \ref{sec:qdifferenceoperators}. 

\begin{theorem} \label{thm:diff_eqn_Pn}
Define two matrix-valued $q$-difference operators by
\begin{align*}
D_i = \mathcal{M}_i(z) \eta_{q} + \mathcal{M}_i(z^{-1}) \eta_{q^{-1}}, \quad i = 1, 2,
\end{align*}
where the multiplication by the matrix-valued functions $\mathcal{M}_i(z)$ and $\mathcal{M}_i(z^{-1})$ is from the left, and 
where $\eta_q$ is the shift operator defined by $(\eta_{q} \breve{f})(z) = \breve{f}(qz)$, 
$\breve{f}(z) =f(\mu(z))$, where $x=\mu(z) = \frac12(z+z^{-1})$.  
The matrix-valued function $\mathcal{M}_1$ is given by
\begin{equation*}
\begin{split}
\mathcal{M}_1(z) 
  &= - \sum_{i=0}^{2\ell-1}
    \frac{q^{1-i}(1 - q^{2i+2})}{(1 - q^2)^2} 
    \frac{z}{(1 - z^2)} E_{i,i+1} \\
  &\qquad + \sum_{i=0}^{2\ell} 
    \frac{q^{1-i}}{(1 - q^2)^2}
    \frac{(1 - q^{2i + 2} z^2)}{(1 - z^2)} E_{i,i},
\end{split}
\end{equation*}
and $\mathcal{M}_2(z) = J \mathcal{M}_1(z) J$, where $J e_{p} = e_{2\ell - p}$.
The matrix-valued orthogonal polynomials $P_n$ are eigenfunctions for the operators $D_i$ with 
eigenvalue matrices given by $\Lambda_n(i)$ such that $D_i P_n = P_n \Lambda_n(i)$ and
\begin{align*}
\Lambda_{n}(1) &= \sum_{j=0}^{2\ell}  
    \frac{
      q^{-j - n - 1} + q^{j + n + 1}
    }{
      (q^{-1} - q)^2
    } E_{j,j}, \quad
\Lambda_{n}(2) = J\Lambda_{n}(1)J.
\end{align*}
Explicitly
\begin{align*}
(D_i P_n)(\mu(z)) = \mathcal{M}_i(z) (\eta_{q} \breve{P_n})(z) + \mathcal{M}_i(z^{-1}) (\eta_{q^{-1}} \breve{P_n})(z) = P_n(\mu(z)) \Lambda_n(i),
\end{align*}
where $\eta_q$ and $\eta_{q^{-1}}$ are applied entry-wise to the matrix-valued orthogonal polynomials $P_n$.
\end{theorem}

Theorem \ref{thm:diff_eqn_Pn} shows that $JD_1J=D_2$, since $J$ is constant. In particular,
$D_1+D_2$ commutes with $J$ and reduces to a $q$-difference operator for 
the matrix-valued orthogonal polynomials associated with the weight $W_+$ or $W_-$,
see Proposition \ref{prop:commutant}. Similarly, $D_1-D_2$ anticommutes with $J$.

Note that for the expression  $\mathcal{M}_i(z) \breve{P}(qz) + \mathcal{M}_i(z^{-1}) \breve{P}(z/q)$ is symmetric in $z\leftrightarrow z^{-1}$ for any matrix-valued polynomial $P$, and hence again is a function in $x=\mu(z)$. 
The case $\ell=0$ corresponds to only one $q$-difference operator, which we rewrite as
\begin{equation}\label{eq:l=0ofdifferenceeqtn}
\left(\frac{1-q^2z^2}{1-z^2}\eta_{q} +\frac{1-q^2z^{-2}}{1-z^{-2}}\eta_{q^{-1}}\right)\breve p_n = (q^{-n} +q^{n+2}) \breve p_n. 
\end{equation}
For the Chebyshev polynomials $U_n(x) = (q^{n+2};q)_n^{-1} p_n(x;q,-q,q^{1/2},-q^{1/2}|q)$ re\-written as Askey-Wilson polynomials \cite[(2.18)]{AskeW} are solutions for the relation \eqref{eq:l=0ofdifferenceeqtn}, see \cite[\S 14.1]{KoekLS}, \cite[\S 7.7]{GaspR}, \cite[Ch.~15-16]{Isma}. 
In particular, we consider the operators 
of Theorem \ref{thm:diff_eqn_Pn} as matrix-valued analogues of the Askey-Wilson operator, see 
Askey and Wilson \cite{AskeW}, or \cite{AndrAR}, \cite{GaspR}, \cite{Isma}.

\begin{corollary}\label{cor:thm:diff_eqn_Pn}
The $q$-difference operators $D_1$ and $D_2$ are symmetric with respect to the matrix-valued weight $W$, i.e. 
for all matrix-valued polynomials $P$, $Q$, we have 
\begin{equation*}
\int_{-1}^1 \bigl( (D_iP)(x)\bigr)^\ast W(x) Q(x) \, dx = \int_{-1}^1 \bigl( P(x)\bigr)^\ast W(x) (D_iQ)(x) \, dx,
\quad i=1,2.
\end{equation*}
\end{corollary}

By \cite[\S 2]{GrunT} it suffices to check Corollary \ref{cor:thm:diff_eqn_Pn} for $P=P_n$, $Q=P_m$, so 
that by Theorem \ref{thm:diff_eqn_Pn} and Theorem \ref{thm:ortho} we need to check that 
$\Lambda_n(i)^\ast G_n\delta_{m,n} = G_n \Lambda_m(i) \delta_{m,n}$, which is true since the 
matrices involved are real and diagonal. 

In order to study the matrix-valued orthogonal polynomials and the weight function in more detail we need 
the continuous $q$-ultraspherical polynomials \cite[Chapter 2]{AndrAR}, \cite{GaspR}, \cite[Chapter 20]{Isma}, \cite{KoekLS};
\begin{equation} \label{eqn:cont_q_ultra_poly}
C_n(x;\beta | q) 
  = \sum_{r = 0}^{n}
  \frac{
    (\beta; q)_r (\beta; q)_{n - r}
  }{
    (q; q)_{r} (q; q)_{n - r}
  } e^{i(n - 2r)\theta},
  \qquad x = \cos\theta.
\end{equation}
The continuous $q$-ultraspherical polynomials are orthogonal polynomials
for $|\beta|<1$. The orthogonality measure is a positive measure in case $0<q<1$ and $\beta$ real 
with $|\beta|<1$. Explicitly
\begin{equation}\label{eq:ortho-contqultraspherpols}
\begin{split}
&\frac{1}{2\pi}\int_{-1}^1 C_n(x;\beta | q) C_m(x;\beta | q) \frac{w(x;\beta\mid q)}{\sqrt{1-x^2}}\, dx = \de_{nm}
 \frac{(\be, \be q;q)_\infty}{(\be^2,q;q)_\infty} \frac{(\be^2;q)_n}{(q;q)_n} \frac{1-\be}{1-\be q^n}, \\
&\qquad\qquad  w(\cos\theta ;\beta| q) = \frac{(e^{2i\theta}, e^{-2i\theta};q)_\infty}{(\be e^{2i\theta}, \be e^{-2i\theta};q)_\infty}.
\end{split}
\end{equation}
Note that in the special case $\be= q^{1+k}$, $k\in \NN$, the weight function is polynomial in $x=\cos\theta$, and
\begin{equation}\label{eq:contqultraspherpolsspecialweight}
w(\cos\theta ;q^{1+k}| q) = 4(1-\cos^2\theta) \, (qe^{2i\theta}, qe^{-2i\theta};q)_k
\end{equation}

We use the continuous $q$-ultraspherical polynomials \eqref{eqn:cont_q_ultra_poly} 
for any $\beta\in \CC$. In particular, for $\beta=q^{-k}$ with $k\in \NN$ the sum in 
\eqref{eqn:cont_q_ultra_poly} is restricted to $n-k\leq r\leq k$, and in particular 
$C_n(x;q^{-k};q)=0$ 
in case $n-k>k$. 
With this convention we can now describe the LDU-decomposition of the weight matrix, and state
the inverse of the unipotent lower triangular matrix $L$ in Theorem \ref{thm:ldu}.

\begin{theorem} \label{thm:ldu}
The matrix-valued weight $W$ as in Theorem \ref{thm:ortho}  has the following LDU-decomposition:
\begin{align*}
W(x) &=  L(x) T(x) L(x)^{t}, \qquad x \in [-1, 1],
\end{align*}
where $L \colon [-1, 1] \to M_{2\ell + 1}(\CC)$ is the unipotent lower triangular matrix
\begin{align*}
L(x)_{m k} &= q^{m - k}
  \cfrac{
    (q^2; q^2)_{m} (q^2; q^2)_{2k + 1}
  }{
    (q^2; q^2)_{m + k + 1} (q^2; q^2)_{k}
  } C_{m - k}(x; q^{2k + 2} | q^2), & 0 \leq k \leq m \leq 2\ell,
\end{align*}
and $T \colon [-1, 1] \to M_{2\ell + 1}(\CC)$ is the diagonal matrix, $0\leq k\leq 2\ell$, 
\begin{align*}
T(x)_{k k} &= c_k(\ell) \frac{w(x;q^{2k+2}|q^2)}{1-x^2}, \\
c_{k}(\ell) &= 
\frac{q^{-2\ell}}{4}
\frac{
  (1 - q^{4k + 2})
  (q^2; q^2)_{2\ell + k + 1}
  (q^2; q^2)_{2\ell - k}
  (q^2; q^2)_{k}^4
}{
  (q^2; q^2)_{2k + 1}^2
  (q^2; q^2)_{2\ell}^2
}.
\end{align*}
The inverse of $L$ is given by
\begin{align*}
\bigl( L(x)\bigr)^{-1}_{k, n} &= 
  q^{(2k + 1)(k - n)} \cfrac{
    (q^2; q^2)_{k} (q^2; q^2)_{k + n} 
  }{
    (q^2; q^2)_{2k} (q^2; q^2)_{n}
  } C_{k - n}(x; q^{-2k} | q^2), & 0 \leq n \leq k.
\end{align*}
\end{theorem}

Note that $T$, $L$ and $L^{-1}$ are matrix-valued polynomials, which is clear from the explicit expression and 
\eqref{eq:contqultraspherpolsspecialweight}. It is remarkable that the LDU-decomposition is for arbitrary size 
$2\ell+1$, but that there is no dependence of $L$ on the spin $\ell$ and that the dependence of 
$T$ on the spin  $\ell$ is only in the constants $c_k(\ell)$. 

We prove the first part of Theorem \ref{thm:ldu} in Section \ref{sec:weightorthorel}. 
The proof of Theorem \ref{thm:ldu} is analytic in nature, and a quantum group theoretic proof would 
be desirable. 
The statement on the inverse 
of $L(x)$ is taken from \cite{Alde}, where the inverse of a
lower triangular matrix with matrix entries continuous $q$-ultraspherical polynomials in a more general situation
is derived. The inverse $L^{-1}$ is derived in \cite[Example 4.2]{Alde}.
The inverse of $L$ in the limit case $q\uparrow 1$ was derived by Cagliero and Koornwinder \cite{CaglK}, 
and the proof of \cite{Alde} is of a different nature than the proof presented in 
\cite{CaglK}. 

Theorem \ref{thm:ldu} shows that $\det(W(x))$ is the product of the diagonal
entries of $T(x)$. Since all coefficients $c_k(\ell)>0$ and the weight functions are positive, we obtain 
Corollary \ref{cor:thmldu}. 

\begin{corollary}\label{cor:thmldu}
The matrix-valued weight $W(x)$ is strictly positive definite for $x \in [-1, 1]$. 
In particular, the matrix-valued weight $W(x)\sqrt{1-x^2}$ of Theorem \ref{thm:ortho} 
is strictly positive definite for $x \in (-1, 1)$.
\end{corollary}

Using the lower triangular matrix $L$ of the 
LDU-decomposition of Theorem \ref{thm:ldu} 
we are able te decouple $D_1$ of Theorem \ref{thm:diff_eqn_Pn} after conjugation with $L^t(x)$. 
We get a scalar $q$-difference equation for each of the  matrix entries of $L^t(x)P_n(x)$, 
which is solved by continuous $q$-ultraspherical polynomials up to a constant. Since we have yet another
matrix-valued $q$-difference operator for $L^t(x)P_n(x)$, namely  $L^tD_2(L^t)^{-1}$ with $D_2$ as in Theorem \ref{thm:diff_eqn_Pn},
we get a relation for the constants involved. This relation turns out to be 
a three-term recurrence relation along columns, which can be identified with the three-term recurrence 
for $q$-Racah polynomials.
Finally, use $(L^t(x))^{-1}$ to obtain an explicit expression for the matrix entries of the 
matrix-valued orthogonal polynomials of Theorem \ref{thm:explicit_Pn}. 

Before stating Theorem \ref{thm:explicit_Pn}, 
recall that the $q$-Racah polynomials, see e.g. \cite[\S 7.2]{GaspR}, \cite[\S 15.6]{Isma}, \cite[\S 14.2]{KoekLS}, 
are defined by
\begin{equation}\label{eqn:q-racah}
R_n(\mu(x);\alpha, \beta, \gamma, \delta ; q)
  = \pfq{4}{3}{
    q^{-n}, \alpha \beta q^{n+1}, q^{-x}, \gamma \delta q^{x+1}
  }{
    \alpha q, \beta \delta q, \gamma q
  }{q}{q}, 
\end{equation}
where $n \in \{0, 1, 2, \ldots, N\}$, $N\in\NN$, $\mu(x) = q^{-x} + \gamma \delta q^{x+1}$ and so that one of the  
conditions $\alpha q = q^{-N}$, or $\beta \delta q = q^{-N}$, or $\gamma q = q^{-N}$ holds.

\begin{theorem} \label{thm:explicit_Pn}
For $0 \leq i, j \leq 2\ell$ we have
\begin{align*}
&P_n(x)_{i, j} = \sum_{k = i}^{2\ell}
  (-1)^k q^{n + (2k+1)(k-i) + j(2k+1) + 2k(2\ell + n + 1) - k^2} \\[0.2cm]
  &\qquad \times \frac{
    (q^2; q^2)_{k} (q^2; q^2)_{k+i}
  }{
    (q^2; q^2)_{2k} (q^2; q^2)_{i}
  }
  \frac{
    (q^{-4\ell}, q^{-2j-2n}; q^2)_{k}
  }{
    (q^2, q^{4\ell+4}; q^2)_{k}
  }
  \frac{
    (q^2; q^2)_{n+j-k}
  }{
    (q^{4k+4}; q^2)_{n+j-k}
  } \\[0.2cm]
  &\qquad \times R_{k}(\mu(j);1,1,q^{-2n-2j-2},q^{-4\ell-2};q^2) \\
  &\qquad \times C_{k-i}(x;q^{-2k}|q^2) C_{n+j-k}(x;q^{2k+2}|q^2).
\end{align*}
\end{theorem}

Note that the left hand side is a polynomial of degree at most $n$, whereas the right hand side 
is of degree $n+j-i$. In particular, for $j>i$ the leading coefficient of the right hand side of
Theorem \ref{thm:explicit_Pn} has to vanish, leading to Corollary \ref{cor:thmexplicit_Pn}.

\begin{corollary} \label{cor:thmexplicit_Pn} 
With the notation of Theorem \ref{thm:explicit_Pn} we have for $j>i$ 
\begin{align*}
&\sum_{k = i}^{2\ell}
  (-1)^k q^{(2k+1)(k-i) + j(2k+1) + 2k(2\ell + n + 1) - k^2} \\
  &\qquad \times \frac{
    (q^2; q^2)_{k} (q^2; q^2)_{k+i}
  }{
    (q^2; q^2)_{2k} 
  }
  \frac{
    (q^{-4\ell}, q^{-2j-2n}; q^2)_{k}
  }{
    (q^2, q^{4\ell+4}; q^2)_{k}
  }\\[0.2cm]
  &\qquad \times 
  \frac{
    (q^{2+2k}; q^2)_{n+j-k}
  }{
    (q^{4k+4}; q^2)_{n+j-k}
  }  R_{k}(\mu(j);1,1,q^{-2n-2j-2},q^{-4\ell-2};q^2) = 0.
\end{align*}
\end{corollary}

By evaluating Corollary \ref{cor:thm:spherical_recurrence} at $1\in \Uq(\mathfrak{g})$ 
we obtain Corollary \ref{cor:normalisation}, which is not clear from Theorem
\ref{thm:explicit_Pn}. 

\begin{corollary}\label{cor:normalisation}
For $m\in \{0,\cdots, 2\ell\}$ we have $\sum_{k=0}^{2\ell} (P_n\bigl(\frac12(q+q^{-1})\bigr))_{k,m}=1$.
\end{corollary}

\subsection{Examples} 

We end this section by specialising the results for low-dimensional cases.
The case $\ell=0$ reduces to the Chebyshev polynomials $U_n(x)$ of the second kind as observed
following Theorem \ref{thm:diff_eqn_Pn}. This is proved in
Proposition \ref{prop:sphericalcaseasChebyshevpols}, which is required for the proofs
of the general statements of Section \ref{sec:mainresults}.

\subsubsection{Example: $\ell=\frac12$}
In this case we work with $2 \times 2$ matrices.
By Proposition \ref{prop:commutant} we know that the weight is block-diagonal, so that in
this case we have an orthogonal decomposition to scalar-valued orthogonal polynomials.
To be explicit, the matrix-valued weight $W$ is given by
\begin{align*}
YW(x)Y^{t} &= \sqrt{1-x^2} \begin{pmatrix}
  2x + (q + q^{-1}) & 0 \\
  0 & -2x + (q + q^{-1})
\end{pmatrix}, \\
\quad Y &= \frac{1}{2} \sqrt{2} \begin{pmatrix}
  1 & 1 \\
  -1 & 1
\end{pmatrix},
\quad x \in [-1, 1].
\end{align*}
In this case, see Section \ref{sec:genMVOP}, the polynomials $P_n$ diagonalise
since the leading coefficient is diagonalised by conjugation with the
orthogonal matrix $Y$, and
we write $YP_n(x)Y^{t} = \diag(p^+_{n}(x), p^-_{n}(x))$.
Then we can identify $p^{\pm}_n$ by any of the results given in this section,
and we do this using the three-term recurrence relation of Corollary \ref{cor:monic-3-term}.
After conjugation, the three-term recurrence relation for $p^+_{n}$ is given by
\begin{align*}
xp^+_{n}(x) &= \frac{1}{2q} \frac{(1 - q^{2n+6})}{(1 - q^{2n+4})} p^+_{n+1}(x)
  + \frac{q^{2n+1}}{2} \frac{(1 - q^{2})^2}{(1 - q^{2n+2})(1 - q^{2n+4})} p^+_{n}(x) \\
  &\qquad + \frac{q}{2} \frac{(1 - q^{2n})}{(1 - q^{2n+2})} p^+_{n-1}(x),
\end{align*}
and the three-term recurrence relation for $p^-_{n}$ is obtained by substituting $x \mapsto -x$ into the
three-term recurrence relation for $p^+_{n}$.
The explicit expressions for $p^+_{n}$ and $p^-_{n}$ are given in terms of continuous $q$-Jacobi polynomials $P_n^{(\alpha, \beta)}(x|q)$ for $(\alpha, \beta)=(\frac{1}{2}, \frac{3}{2})$,
see \cite[\S 4]{AskeW}, \cite[\S 14.10]{KoekLS}.
From \cite[Exercise~7.32.(ii)]{GaspR} we have $P_n^{(\frac{1}{2}, \frac{3}{2})}(-x|q^2) = (-1)^n q^{-n} P_n^{(\frac{3}{2}, \frac{1}{2})}(x|q^2)$.
So we obtain
\begin{align*}
p^+_{n}(x) &= \frac{(1 - q^4)}{(1 - q^{2n+4})} \frac{(q^2,-q^3,-q^4;q^2)_n}{(q^{2n+6};q^2)_n} P_n^{(\frac{1}{2}, \frac{3}{2})}(x|q^2), \\
p^-_{n}(x) &= (-1)^n q^{-n} \frac{(1 - q^4)}{(1 - q^{2n+4})} \frac{(q^2, -q^3, -q^4; q^2)_n}{(q^{2n+6}; q^2)_n} P^{(\frac{3}{2}, \frac{1}{2})}_n(x|q^2),
\end{align*}
which is a $q$-analogue of \cite[\S 8.2]{KoelvPR12}. Moreover,
writing down the conjugation of the $q$-difference operator $D_1+D_2$ of Theorem \ref{thm:diff_eqn_Pn}
for the case $\ell=\frac12$ for the
conjugated polynomials gives back the Askey-Wilson $q$-difference for the
continuous $q$-Jacobi polynomials $P_n^{(\alpha, \beta)}(x|q)$ for $(\alpha,\beta)=
(\frac{1}{2}, \frac{3}{2})$ and $(\frac{3}{2},\frac{1}{2})$. Working out the eigenvalue equation for
$D_1-D_2$ gives a simple $q$-analogue of the contiguous relations of \cite[p.~5708]{KoelvPR12}.

\subsubsection{Example: $\ell=1$}
For $\ell = 1$ we work with $3 \times 3$ matrices.
By Proposition \ref{prop:commutant} we can block-diagonalise the matrix-valued weight:
\begin{align*}
YW(x)Y^{t} &= \sqrt{1-x^2} \begin{pmatrix}
  W_+(x) & 0 \\
  0 & W_-(x)
\end{pmatrix}, \\
Y &= \frac{1}{2}\sqrt{2} \begin{pmatrix}
  1 & 0 & 1 \\
  0 & \sqrt{2} & 0 \\
  -1 & 0 & 1
\end{pmatrix},
\quad x \in [-1,1],
\end{align*}
where $W_+$ is a $2 \times 2$ matrix-valued weight and $W_-$ is a scalar-valued  weight.
Explicitly,
\begin{align*}
W_+(x) &= \begin{pmatrix}
  4x^2 + (q^2 + q^{-2}) & \displaystyle 2\sqrt{2}q^{-1}\frac{(1 + q^2 + q^4)}{(1 + q^2)}x \\
  \displaystyle 2\sqrt{2}q^{-1}\frac{(1 + q^2 + q^4)}{(1 + q^2)}x & \displaystyle \frac{4q^2}{(1 + q^2)^2}x^2 + (q^2 + q^{-2})
\end{pmatrix}, \\
W_-(x) &= -4x^2 + (q^2 + 2 + q^{-2}).
\end{align*}
The polynomials diagonalise by $YP_n(x)Y^t = \diag(P^+_{n}(x), p^{-}_{n}(x))$,
where $P^+_{n}$ is a $2 \times 2$ matrix-valued polynomial and $p^{-}_{n}$ is a scalar-valued polynomial.
Conjugating Corollary \ref{cor:monic-3-term}, the three-term
recurrence relations for $P^+_{n}$ and $p^{-}_{n}$ are
\begin{align*}
xP^+_{n}(x) &= P^+_{n+1}(x) A_n + P^+_{n}(x) B_n + P^+_{n-1}(x) C_n, \\
xp^{-}_{n}(x) &= \frac{1}{2q} \frac{(1 - q^{2n+8})}{(1 - q^{2n+6})} p^{-}_{n+1}(x) + \frac{q}{2} \frac{(1 - q^{2n})}{(1 - q^{2n+2})} p^{-}_{n-1}(x),
\end{align*}
where
\begin{align*}
A_n &= \frac{1}{2q} \begin{pmatrix}
  \displaystyle \frac{(1 - q^{2n+8})}{(1 - q^{2n+6})} & 0 \\
  0 & \displaystyle \frac{(1 - q^{2n+8})(1 - q^{2n+2})}{(1 - q^{2n+4})^2}
\end{pmatrix}, \\[0.3cm]
B_n &= \frac{1}{2}\sqrt{2} \begin{pmatrix}
  0 & \displaystyle q^{2n+1} \frac{(1 - q^2)(1 - q^4)}{(1 - q^{2n+4})^2} \\
  \displaystyle q^{2n+1} \frac{(1 - q^2)(1 - q^4)}{(1 - q^{2n+2})(1 - q^{2n+6})} & 0
\end{pmatrix}, \\[0.3cm]
C_n &= \frac{q}{2} \begin{pmatrix}
  \displaystyle \frac{(1 - q^{2n})}{(1 - q^{2n+2})} & 0 \\
  0 & \displaystyle \frac{(1 - q^{2n})(1 - q^{2n+6})}{(1 - q^{2n+4})^2}
\end{pmatrix}.
\end{align*}
The scalar-valued polynomial $p^{-}_{n}$ can be identified with the continuous
$q$-ultra\-spherical polynomials;
\begin{align*}
p^{-}_{n}(x) = q^n \frac{(1 - q^2)(1 - q^6)}{(1 - q^{2n+2})(1 - q^{2n+6})} C_n(x;q^4|q^2).
\end{align*}
The $2 \times 2$ matrix-valued polynomials $P^{+}_{n}$ are solutions to the
matrix-valued $q$-difference equation $DP^{+}_{n,1} = P^{+}_{n} \Lambda_n$.
Here $D= M(z) \eta_q + M(z^{-1}) \eta_{q^{-1}}$ is the restriction of
the conjugated $D_1 + D_2$ to the $+1$-eigenspace of $J$.
The explicit expressions for $M(z)$ and
$\Lambda_n$ are
\begin{align*}
M(z) &= \frac{q^{-1}}{(1 - q^2) (1 - z^2)} \begin{pmatrix}
  \displaystyle \frac{(1 + q^2)}{(1 - q^2)} (1 - q^4 z^2) & -\sqrt{2}q^2z \\
  -\sqrt{2}q(1 + q^2)z & \displaystyle 2q\frac{(1 - q^4z^2)}{(1 - q^2)}
\end{pmatrix}, \\[0.3cm]
\Lambda_n &= \begin{pmatrix}
  \displaystyle q^{-n-1} \frac{(1 + q^2)(1 + q^{2n+4})}{(1 - q^2)^2} & 0 \\
  0 & \displaystyle 2q^{-n} \frac{(1 + q^{2n+4})}{(1 - q^2)^2}
\end{pmatrix}.
\end{align*}
These results are $q$-analogues of some of the results given in \cite[\S 8.3]{KoelvPR12}, see also \cite{PachZ}.
Note moreover that $W_-(0)$ is a multiple of the identity, so that the
commutant of $W_-$ equals the commuting algebra of Tirao and Zurri\'an \cite{Tira15}, see also \cite{KR15}.
Since the commutant is trivial, the weight $W_-$ is irreducible, which can also be checked directly.

\section{Quantum group related properties of spherical functions}\label{sec:quantumgrouprelatedpropssphericalf}

In this section we start with the proofs of the statements of Section \ref{sec:mainresults} which can be obtained using
the interpretation of matrix-valued spherical functions on $\Uq(\mathfrak{g})$. 

\subsection{Matrix-valued spherical functions on the quantum group}\label{ssec:MVSFonquantumgroup}

In this subsection we study some of the properties of the matrix-valued spherical functions which
follow from the quantum group theoretic interpretation. In particular, we 
derive Theorem \ref{thm:cgc} from Remark \ref{rmk:identification}. The precise identification
with the literature and the standard Clebsch-Gordan coefficients is made in Appendix \ref{app:cgc},
and we use the intertwiner and the Clebsch-Gordan coefficients as presented there.

We also need the matrix elements of the type 1 irreducible finite dimensional representations. Define
\begin{equation*} 
t^\ell_{m,n}\colon \Uq(\mathfrak{su}(2)) \to \mathbb{C}, \quad
t^\ell_{m,n}(X) = \langle t^\ell(X) e^\ell_n, e^\ell_m\rangle, \quad n,m\in\{-\ell,\cdots, \ell\},
\end{equation*}
where we take the inner product in the representation space $\mathcal{H}^\ell$ for which the 
basis $\{e^\ell_n\}_{n=-\ell}^\ell$ is orthonormal. 
Denoting 
\begin{equation*}
\begin{pmatrix}t^{1/2}_{-1/2,-1/2} & t^{1/2}_{-1/2,1/2} \\ t^{1/2}_{1/2,-1/2} & t^{1/2}_{1/2,1/2}
\end{pmatrix} = 
\begin{pmatrix}
\alpha & \beta \\ \gamma & \delta
\end{pmatrix},
\end{equation*}
then $\alpha$, $\beta$, $\gamma$, $\delta$ generate a Hopf algebra, where
the Hopf algebra structure is determined by duality of Hopf algebras.  
Moreover, it is a Hopf $\ast$-algebra with $\ast$-structure defined by $\alpha^\ast = \delta$,
$\beta^\ast= -q\gamma$, which 
we denote by $\mathcal{A}_q(SU(2))$. 
Then the Hopf $\ast$-algebra $\mathcal{A}_q(SU(2))$ is in duality as Hopf $\ast$-algebras  
with $\Uq(\mathfrak{su}(2))$. 
In particular, the matrix elements $t^\ell_{m,n}\in \mathcal{A}_q(SU(2))$ can 
be expressed in terms of the generators and span $\mathcal{A}_q(SU(2))$. 
Moreover, the matrix elements $t^\ell_{m,n}\in \mathcal{A}_q(SU(2))$ form a basis for 
the underlying vector space of $\mathcal{A}_q(SU(2))$.
The left action of $\Uq(\mathfrak{g})$ on $\mathcal{A}_q(SU(2))$ is given by
$(X\cdot \xi)(Y) = \xi(YX)$ for $X,Y\in\Uq(\mathfrak{g})$ and $\xi\in\mathcal{A}_q(SU(2))$.
Similarly the right action is given by
$(\xi\cdot X)(Y) = \xi(XY)$ for $X,Y\in\Uq(\mathfrak{g})$ and $\xi\in\mathcal{A}_q(SU(2))$.
A calculation gives $k^{1/2}\cdot t^\ell_{m,n}= q^{-n}t^\ell_{m,n}$ 
and $t^\ell_{m,n}\cdot k^{1/2}= q^{-m}t^\ell_{m,n}$, so that  
$\alpha\cdot k^{1/2} = q^{1/2}\alpha$, 
$\beta\cdot k^{1/2} = q^{-1/2}\beta$, 
$\gamma\cdot k^{1/2} = q^{1/2}\gamma$,
$\delta\cdot k^{1/2} = q^{-1/2}\delta$. 

Since $k^{1/2}$ and its powers are group-like elements
of $\Uq(\mathfrak{g})$, it follows that the left and right action of $k^{1/2}$ and its powers are 
algebra homomorphisms. 
See e.g. \cite{CharP}, \cite{EtinS}, \cite{KlimS}, \cite{Koel1996}, and references given there. 

In the same way we view
\begin{equation*} 
t^{\ell_1}_{m_1,n_1}\otimes t^{\ell_2}_{m_2,n_2} \colon \Uq(\mathfrak{g}) \to \mathbb{C} 
\end{equation*}
were the functions are taken with respect to the Hopf algebra tensor product $\Uq(\mathfrak{g}) = \Uq(\mathfrak{su}(2))\otimes \Uq(\mathfrak{su}(2))$. 
In particular for $\lambda, \mu \in \frac{1}{2} \ZZ$ we find the expression $t^{\ell_1}_{m_1,n_1}\otimes t^{\ell_2}_{m_2,n_2}(K_1^\lambda K_2^\mu) = \delta_{m_1,n_1}\delta_{m_2,n_2} q^{-\lambda m_1 -\mu m_2}$.
Similarly, the Hopf $\ast$-algebra spanned by all the matrix elements $t^{\ell_1}_{m_1,n_1}\otimes t^{\ell_2}_{m_2,n_2}$ 
is isomorphic to $\mathcal{A}_q(SU(2))\otimes \mathcal{A}_q(SU(2))$. 
We set $\mathcal{A}_q(G)=\mathcal{A}_q(SU(2))\otimes \mathcal{A}_q(SU(2))$ for $G=SU(2)\times SU(2)$. 

Define $A = K_1^{1/2} K_2^{1/2}$ and let $\mathcal{A}$ be the commutative subalgebra of $\Uq(\mathfrak{g})$ 
generated by $A$ and $A^{-1}$.
Recall the spherical function $\Phi^{\ell}_{\ell_1, \ell_2}$ from Definition \ref{def:elementarysphericalfunction}, and recall the transformation property 
\eqref{eq:invarianceproperties}. 

\begin{definition}\label{def:sphericalfunctionoftypeell}
The linear map $\Phi\colon \Uq(\mathfrak{g})\to \text{\rm End}(\mathcal{H}^\ell)$ is 
a spherical function of type $\ell$ if 
\[
\Phi(XZY) = t^\ell(X) \Phi(Z) t^\ell(Y), 
\qquad \forall\, X,Y\in \BB, \ \forall\, Z\in \Uq(\mathfrak{g}).
\] 
\end{definition}

So the spherical function $\Phi^{\ell}_{\ell_1, \ell_2}$ is a spherical function of type $\ell$ by 
\eqref{eq:invarianceproperties}.

\begin{proposition} \label{prop:usefull}
Fix  $\ell \in \frac{1}{2} \NN$, and assume $(\ell_1,\ell_2)$ satisfies \eqref{eq:conditionsonl1l2}.

\noindent
\textrm{\rm (i)} Write $\Phi^{\ell}_{\ell_1, \ell_2} = \sum_{m, n=-\ell}^\ell (\Phi^{\ell}_{\ell_1, \ell_2})_{m, n} \otimes E^\ell_{m, n}$, 
where $E^\ell_{m, n}$ are the elementary matrices, then
\begin{align*} 
\left(\Phi^{\ell}_{\ell_1, \ell_2} \right)_{m, n}
  &= \sum_{m_1, n_1 = -\ell_1}^{\ell_1}
    \sum_{m_2, n_2 = -\ell_2}^{\ell_2}
      C^{\ell_1, \ell_2, \ell}_{m_1, m_2, m}
      C^{\ell_1, \ell_2, \ell}_{n_1, n_2, n}\, 
      t^{\ell_1}_{m_1, n_1} \tensor t^{\ell_2}_{m_2, n_2}.
\end{align*}
\noindent
\textrm{\rm (ii)} A spherical function $\Phi$ of type $\ell$ restricted to $\mathcal{A}$ is diagonal with respect to 
the basis $\{e^\ell_p\}_{p=-\ell}^\ell$. 
Moreover, for each $\lambda \in \ZZ$,
\begin{align*}
\left( \Phi^{\ell}_{\ell_1, \ell_2}(A^{\lambda}) \right)_{m, n}
  &= \delta_{m, n} 
    \sum_{i = -\ell_1}^{\ell_1}
    \sum_{j = -\ell_2}^{\ell_2}
    \left( C^{\ell_1, \ell_2, \ell}_{i, j, n} \right)^2
    q^{-\lambda(i + j)},
\end{align*}
so that $\Phi^{\ell}_{\ell_1, \ell_2}(1)$ is the identity. \par
\noindent
\textrm{\rm (iii)} Assume $\Phi$ is a spherical function of type $\ell$ and that 
\begin{equation*}
\begin{split}
\Phi= \sum_{m, n=-\ell}^\ell \Phi_{m, n} \otimes E^\ell_{m, n}\colon \Uq(\mathfrak{g})\to \textrm{\rm End}(\mathcal{H}^\ell), 
\end{split}
\end{equation*}
with all linear maps $\Phi_{m, n}$ on $\Uq(\mathfrak{g})$ in the linear span of the matrix-elements $t^{\ell_1}_{i_1,j_1}\otimes
t^{\ell_2}_{i_2,j_2}$, $-\ell_1\leq i_1,j_1\leq \ell_1$, $-\ell_2\leq i_2,j_2\leq \ell_2$,
then $\Phi$ is a multiple of $\Phi^{\ell}_{\ell_1, \ell_2}$.  
\end{proposition}

\begin{proof}
Note $(\Phi^{\ell}_{\ell_1, \ell_2}(X))_{m,n} = \langle \Phi^{\ell}_{\ell_1, \ell_2}(X) e^{\ell}_n, e^{\ell}_m \rangle$ 
for $X\in\Uq(\mathfrak{g})$, therefore we first compute $\Phi^{\ell}_{\ell_1, \ell_2}(X) e^{\ell}_n$.
For $X \in \Uq(\mathfrak{g})$ we have
\begin{align*}
\Phi^{\ell}_{\ell_1, \ell_2}(X) e^{\ell}_n
  &= (\beta^\ell_{\ell_1,\ell_2})^\ast \circ t^{\ell_1, \ell_2}(X)
    \left(
      \sum_{n_1 = -\ell_1}^{\ell_1} \sum_{n_2 = -\ell_2}^{\ell_2}
      C^{\ell_1, \ell_2, \ell}_{n_1, n_2, n}
      e^{\ell_1}_{n_1} \tensor e^{\ell_2}_{n_2}
    \right) \\
  &= (\beta^\ell_{\ell_1,\ell_2})^\ast \Bigg(
    \sum_{m_1, n_1 = -\ell_1}^{\ell_1}
    \sum_{m_2, n_2 = -\ell_2}^{\ell_2}
    C^{\ell_1, \ell_2, \ell}_{n_1, n_2, n} \\
  &\qquad \qquad \quad \times \langle t^{\ell_1, \ell_2}(X) e^{\ell_1}_{n_1} \tensor e^{\ell_2}_{n_2},
      e^{\ell_1}_{m_1} \tensor e^{\ell_2}_{m_2} \rangle
    e^{\ell_1}_{m_1} \tensor e^{\ell_2}_{m_2}
  \Bigg) \\
  &= \sum_{m = -\ell}^{\ell}
    \sum_{m_1, n_1 = -\ell_1}^{\ell_1}
    \sum_{m_2, n_2 = -\ell_2}^{\ell_2}
    C^{\ell_1, \ell_2, \ell}_{m_1, m_2, m} 
    C^{\ell_1, \ell_2, \ell}_{n_1, n_2, n} \\
    & \qquad \qquad \qquad \qquad \times (t^{\ell_1}_{m_1, n_1} \tensor t^{\ell_2}_{m_2, n_2})(X) e^{\ell}_{m}.
\end{align*}
This proves (i).

To obtain (ii) write $\Phi= \sum_{m, n=-\ell}^\ell \Phi_{m, n} \otimes E^\ell_{m, n}$, 
and observe 
\begin{equation*}
t^\ell(K^\mu) \Phi(A^\lambda)
  = \Phi(K^\mu A^\lambda)
  = \Phi(A^\lambda K^\mu)
  = \Phi(A^\lambda)t^\ell(K^\mu),
\end{equation*}
for all $\lambda\in\ZZ$, $\mu\in\frac12 \ZZ$ that implies 
$q^{-m\mu} \Phi_{m,n}(A^\lambda)= q^{-n\mu} \Phi_{m,n}(A^\lambda)$ since $t^\ell(K^\mu)$
is diagonal. This gives $\Phi_{m,n}(A^\lambda)=0$ for $m\not=n$. 
Next pair $\Phi^{\ell}_{\ell_1, \ell_2}$  with $A^{\lambda}$ using (i) and the observation made
before Proposition \ref{prop:usefull} and the fact $C^{\ell_1, \ell_2, \ell}_{m_1, m_2, m} =0$ unless
$m_1-m_2=m$.  
Next take $\lambda=0$, and use
\eqref{eq:CGCorthorel-partial}.

Finally, for (iii) note that for $X\in \mathcal{B}$ we have $\Phi(X)= t^\ell(X) \Phi(1) = 
\Phi(1) t^\ell(X)$. Since $t^\ell\colon \mathcal{B} \to \text{End}(\mathcal{H}^\ell)$ is an irreducible
unitary representation, Schur's Lemma implies that $\Phi(1)=cI$ is a multiple of the identity. 
Theorem \ref{thm:cgc} implies that for $X\in \mathcal{B}$
\begin{equation*}
t^\ell(X)_{m,n} = 
 \sum_{m_1, n_1 = -\ell_1}^{\ell_1}
    \sum_{m_2, n_2 = -\ell_2}^{\ell_2}
    C^{\ell_1, \ell_2, \ell}_{m_1, m_2, m} 
    C^{\ell_1, \ell_2, \ell}_{n_1, n_2, n}
    (t^{\ell_1}_{m_1, n_1} \otimes t^{\ell_2}_{m_2, n_2})(X)
\end{equation*}
and, by the multiplicity free statement in Theorem \ref{thm:cgc}, this is the only (up to a constant) 
possible linear combination of the matrix elements $t^{\ell_1}_{m_1, n_1} \otimes t^{\ell_2}_{m_2, n_2}$
for fixed $(\ell_1,\ell_2)$ which has this property. Hence, (iii) follows. 
\end{proof}

The special case $\Phi^0_{1/2,1/2}\colon \Uq(\mathfrak{g})\to \CC$ can 
now be calculated explicitly using the 
Clebsch-Gordan coefficients $C^{1/2,1/2,0}_{m,-m,0}$ from \eqref{eq:CGCforl1isl2ishalf}.
Explicitly,
\begin{equation}\label{eq:expressionvarphi}
\begin{split}
\varphi  
  &= \frac12 (q^{-1}+q) \Phi^{0}_{1/2, 1/2} \\ 
  &= \frac12 q t^{\frac12}_{-\frac12,-\frac12}\otimes t^{\frac12}_{-\frac12,-\frac12} 
    - \frac12 t^{\frac12}_{-\frac12,\frac12}\otimes t^{\frac12}_{-\frac12,\frac12} \\
    &\qquad - \frac12 t^{\frac12}_{\frac12,-\frac12}\otimes t^{\frac12}_{\frac12,-\frac12} 
    + \frac12 q^{-1} t^{\frac12}_{\frac12,\frac12}\otimes t^{\frac12}_{\frac12,\frac12} \\
  &= \frac12 q \alpha\otimes \alpha - \frac12\beta\otimes \beta 
    - \gamma\otimes \gamma + \frac12 q^{-1} \delta\otimes \delta.
\end{split}
\end{equation}
This element is not self-adjoint in $\mathcal{A}_q(SU(2))\otimes \mathcal{A}_q(SU(2))$.
Recall the right action of $\Uq(\mathfrak{g})$ on the map $\Phi \colon \Uq(\mathfrak{g}) 
\to \text{End}(\mathcal{H}^\ell)$, including the case $\ell=0$, 
by $(\Phi\cdot X)(Y) = \Phi(XY)$ for all 
$X, Y \in \Uq(\mathfrak{g})$. This is analogous to the construction discussed at the beginning 
of this subsection. Then 
\begin{equation}\label{eq:expressionpsi}
\psi =  \varphi \cdot A^{-1} = \frac12\alpha\otimes \alpha - \frac12 q^{-1} \beta\otimes \beta - 
\frac12 q \gamma\otimes \gamma +\frac12\delta\otimes \delta =\psi^\ast
\end{equation}
is self-adjoint for the $\ast$-structure of $\mathcal{A}_q(SU(2))\otimes \mathcal{A}_q(SU(2))$. 
Then by construction
\begin{equation}\label{eq:invarianceproppsi}
\psi((AXA^{-1})YZ) = \varepsilon(X) \psi(Y) \varepsilon(Z), \qquad \forall\, X,Z\in \BB, \ Y\in \Uq(\mathfrak{g}).
\end{equation}

\subsection{The recurrence relation for spherical functions of type $\ell$}\label{subsec:sph_rec}

The proof of Theorem \ref{thm:spherical_recurrence} in the group case can be found 
in \cite[Prop.~3.1]{KoelvPR12}, where the constants $A_{i,j}$ are explicitly given in terms
of Clebsch-Gordan coefficients. This proof can also be applied in this case, giving the 
coefficients $A_{i,j}$ explicitly in terms of Clebsch-Gordan coefficients. 
A more general set-up  
can be found in \cite[Proposition 3.3.17]{vPrui12}.
The approach given here is related \cite[Prop.~3.1]{KoelvPR12}, except that it 
differs in its way of establishing $A_{1/2,1/2}\not=0$. 

\begin{proof}[Proof of Theorem \ref{thm:spherical_recurrence}]
As $\Uq(\mathfrak{g})$-representations the tensor product decomposition 
\begin{equation*}
t^{1/2,1/2}\otimes t^{\ell_1,\ell_2} \cong \sum_{i,j=\pm 1/2} t^{\ell_1+i,\ell_2+j}
\end{equation*}
follows from the standard tensor product decomposition for $\Uq(\mathfrak{su}(2))$, see 
\eqref{eq:standardCGC}. 
It follows that 
\begin{equation*}
\varphi \Phi^{\ell}_{\ell_1, \ell_2} = \sum_{i,j=\pm 1/2} \Psi^{i,j}, \quad
\Psi^{i,j} = \sum_{m,n=-\ell}^\ell \Psi^{i,j}_{m,n} \otimes E^\ell_{m,n},
\end{equation*}
where $\Psi^{i,j}_{m,n}$ is in the span of the matrix elements $t^{\ell_1+i}_{r_1,s_1}\otimes t^{\ell_1+j}_{r_2,s_2}$. 
Note that \eqref{eq:leftinvariancepropertyproductThmsphericalrecurrrence}, and a similar calculation 
for multiplication by an element from $\mathcal{B}$ from the other side,
show that $\varphi \Phi^{\ell}_{\ell_1, \ell_2}$ 
has the required transformation behaviour \eqref{eq:invarianceproperties} for the action of $\mathcal{B}$ from the left and the right. 
Since the matrix elements $t^{\ell_1}_{r_1,s_1}\otimes t^{\ell_1}_{r_2,s_2}$ form a basis for $\mathcal{A}_q(G)$, it 
follows that each $\Psi^{i,j}$ satisfies \eqref{eq:invarianceproperties}, so that by 
Proposition \ref{prop:usefull}(iii) $\Psi^{i,j}= A_{i,j} \, \Phi^{\ell}_{\ell_1+i, \ell_2+j}$. 
Here $\Psi^{i,j}=0$ in case $(\ell_1+i,\ell_2+j)$ does not satisfy  the conditions \eqref{eq:conditionsonl1l2}. 

It remains to show that $A_{1/2,1/2}\not=0$. In order to do so we evaluate the identity of Theorem \ref{thm:spherical_recurrence} 
at a suitable element of $\Uq(\mathfrak{g})$.
For the $\Uq(\mathfrak{su}(2))$-representations in \eqref{eq:defrepresentationUqsu2} it is immediate that 
$t^\ell(e^k)=0$ for $k>2\ell$ and that $t^\ell_{r,s}(e^{2\ell})=0$ except for the case $(r,s)=(-\ell,\ell)$ and 
then $t^\ell_{-\ell,\ell}(e^{2\ell})\not=0$. Extending to $\Uq(\mathfrak{g})$ we find that 
$\Phi^{\ell}_{\ell_1,\ell_2} (E_1^{k_1}E_2^{k_2})=0$ if $k_1>2\ell_1$ or $k_2>2\ell_2$, and 
\begin{equation}\label{eq:PhievaluatedatpowersofE1E2}
\begin{split}
\bigl( \Phi^{\ell}_{\ell_1,\ell_2}\bigr)_{m,n} (E_1^{2\ell_1}E_2^{2\ell_2}) 
  &= \delta_{m,-\ell_1+\ell_2}\delta_{n,\ell_1-\ell_2} 
    C^{\ell_1, \ell_2, \ell}_{-\ell_1, -\ell_2, m}     
    C^{\ell_1, \ell_2, \ell}_{\ell_1, \ell_2, n} \\ 
  &\qquad \times t^{\ell_1}_{-\ell_1,\ell_1}(E_1^{2\ell_1}) t^{\ell_2}_{-\ell_2,\ell_2}(E_2^{2\ell_2}),
\end{split}
\end{equation}
where the right hand side is non-zero in case $m=-\ell_1+\ell_2$, $n=\ell_1-\ell_2$.

So if we evaluate the identity of Theorem \ref{thm:spherical_recurrence} at $E_1^{2\ell_1+1}E_2^{2\ell_2+1}$, it follows that
only the term with $(i,j)=(1/2,1/2)$ on the right hand side of Theorem \ref{thm:spherical_recurrence} is non-zero, and the specific matrix element is given by 
\eqref{eq:PhievaluatedatpowersofE1E2} with $(\ell_1,\ell_2)$ replaced by $(\ell_1+\frac12,\ell_2+\frac12)$. 
It suffices to check that the left hand of Theorem \ref{thm:spherical_recurrence} is non-zero when evaluated at $E_1^{2\ell_1+1}E_2^{2\ell_2+1}$. 
By \eqref{eq:interpretationphitimesPhi} we need to calculate the comultiplication on $E_1^{2\ell_1+1}E_2^{2\ell_2+1}$. 
Using the non-commutative $q$-binomial theorem \cite[Exercise~1.35]{GaspR} twice we get 
\begin{equation*} 
\begin{split}
\Delta(E_1^{2\ell_1+1}E_2^{2\ell_2+1}) &= \Delta(E_1)^{2\ell_1+1}\Delta(E_2)^{2\ell_2+1} \\ 
&= \sum_{k_1=0}^{2\ell_1+1} \sum_{k_2=0}^{2\ell_2+1} \qbin{2\ell_1+1}{k_1}_{q^2} \qbin{2\ell_2+1}{k_2}_{q^2} \\
& \qquad \times E_1^{k_1}E_2^{k_2} K_1^{2\ell_1+1-k_1}K_2^{2\ell_2+1-k_2} \otimes E_1^{2\ell_1+1-k_1}E_2^{2\ell_2+1-k_2}.
\end{split}
\end{equation*}
From \eqref{eq:expressionvarphi} we find that $\varphi(E_1^{k_1}E_2^{k_2} K_1^{2\ell_1+1-k_1}K_2^{2\ell_2+1-k_2})=0$
unless $k_1=k_2=1$, and in that case the term $\beta\otimes \beta$ of \eqref{eq:expressionvarphi} gives a non-zero
contribution, and the other terms give zero. So we find
\begin{equation*}
\begin{split}
\bigl( \varphi \Phi^{\ell}_{\ell_1,\ell_2}\bigr) (E_1^{2\ell_1+1}E_2^{2\ell_2+1}) &= \qbin{2\ell_1+1}{1}_{q^2} \qbin{2\ell_2+1}{1}_{q^2} \\
&\qquad \times \varphi(E_1E_2 K_1^{2\ell_1}K_2^{2\ell_2}) \bigl( \Phi^{\ell}_{\ell_1,\ell_2}\bigr) (E_1^{2\ell_1}E_2^{2\ell_2}),
\end{split}
\end{equation*}
and this is non-zero for the same matrix element $(m,n) = (-\ell_1+\ell_2,\ell_1-\ell_2)$ by
\eqref{eq:PhievaluatedatpowersofE1E2}. 
\end{proof}

Note that we can derive the explicit value of $A_{1/2,1/2}$ from the proof of Theorem \ref{thm:spherical_recurrence} 
by keeping track of the constants involved. However, we don't need the explicit value 
except in the case $\ell=0$. 

In the special case $\ell=0$, the recurrence of Theorem \ref{thm:spherical_recurrence} has two terms and we obtain
\begin{equation}\label{eq:spherical_recurrenceell=0}
\begin{split}
&\varphi \Phi^0_{\ell_1,\ell_1} = A_{1/2,1/2} \, \Phi^0_{\ell_1+1/2,\ell_1+1/2} +  A_{-1/2,-1/2} \, \Phi^0_{\ell_1-1/2,\ell_1-1/2}, \\
& A_{1/2,1/2} = \frac12 q^{-1}\frac{1-q^{4\ell_1+4}}{1-q^{4\ell_1+2}}, \qquad A_{-1/2,-1/2} = 
\frac12 q\frac{1-q^{4\ell_1}}{1-q^{4\ell_1+2}}.
\end{split}
\end{equation}
The value of $A_{1/2,1/2}$ can be obtained by evaluating at the group like element $A^\lambda$, using Proposition 
\ref{prop:usefull}(ii) and $2 \varphi(A^\lambda) = q^{\lambda+1}+q^{-1-\lambda}$ and comparing leading coefficients of the Laurent polynomials in $q^\lambda$. 
This gives $\frac12 q^{-1} (C^{\ell_1,\ell_1,0}_{-\ell_1,-\ell_1,0})^2= A_{1/2,1/2} (C^{\ell_1+1/2,\ell_1+1/2,0}_{-\ell_1-1/2,-\ell_1-1/2,0})^2$, and the 
value  for $A_{1/2,1/2}$ follows from \eqref{eq:CGCforl1isl2jisminell}.
Having $A_{1/2,1/2}$, we evaluate at $1$, i.e. the case $\lambda = 0$, which gives $A_{1/2,1/2}+A_{-1/2,-1/2}=\frac12(q+q^{-1})$ by Proposition \ref{prop:usefull}(ii), 
from which 
$A_{-1/2,-1/2}$ follows. 

\begin{proposition}\label{prop:sphericalcaseasChebyshevpols}
For $n\in\NN$ we have 
$\displaystyle{\Phi^0_{\frac{1}{2}n,\frac{1}{2}n} = q^{n}\frac{1-q^2}{1-q^{2n+2}} U_n(\varphi)}$. 
\end{proposition}

Recall that the Chebyshev polynomials of the second kind are orthogonal polynomials;
\begin{equation}\label{eq:defChebyshevpols}
\begin{split}
U_n(\cos\theta) = \frac{\sin((n+1)\theta)}{\sin\theta}, \quad
\int_{-1}^1 U_n(x) U_m(x) \sqrt{1-x^2}\, dx = \delta_{m,n} \frac12 \pi, \\
x\, U_n(x) = \frac12 U_{n+1}(x) + \frac12 U_{n-1}(x), \qquad U_{-1}(x)=0, \ U_0(x)=1, 
\end{split}
\end{equation}
see e.g. \cite{GaspR}, \cite{Isma}, \cite{KoekLS}. 

\begin{proof} From \eqref{eq:spherical_recurrenceell=0} it follows that 
$\Phi^0_{\frac{1}{2}n,\frac{1}{2}n} = p_n(\varphi)$ for a polynomial $p_n$ of degree
$n$ satisfying 
\begin{equation*}
x \, p_n(x) = \frac12 q^{-1}\frac{1-q^{2n+4}}{1-q^{2n+2}} p_{n+1}(x) 
+\frac12 q\frac{1-q^{2n}}{1-q^{2n+2}} p_{n-1}(x), 
\end{equation*}
with initial conditions $p_0(x)=1$, $p_1(x)=\frac{2}{q+q^{-1}}x$. 
Set 
\begin{equation*}
p_n(x) = q^{n}\frac{1-q^2}{1-q^{2n+2}}r_n(x), 
\end{equation*}
then $r_0(x)=1$, $r_1(x)=2x$ and $2x\, r_n(x) = r_{n+1}(x) + r_{n-1}(x)$. So $r_n(x) = U_n(x)$. 
\end{proof}

\subsection{Orthogonality relations}\label{subsec:orthogonality}

In this subsection we prove Theorem \ref{thm:ortho} from the quantum group theoretic 
interpretation up to the calculation of certain explicit coefficients in the expansion of the 
entries of the weight. 

Recall the Haar functional on $\mathcal{A}_q(SU(2))$. It is the unique left and right invariant 
positive functional 
$h_0\colon \mathcal{A}_q(SU(2)) \to \mathbb{C}$ normalised by $h_0(1)=1$,
see e.g. \cite{CharP}, \cite[\S 4.3.2]{KlimS}, \cite{Koel1996}, \cite{Woro}.
The Schur orthogonality relations state 
\begin{equation}\label{eq:HaaronquantumSU(2)}
h_0\bigl( t^{\ell_1}_{m,n} (t^{\ell_2}_{r,s})^\ast \bigr) 
= \delta_{\ell_1, \ell_2} \delta_{m,r } \delta_{n,s} 
q^{2(\ell_1 +n )} \frac{(1 - q^2)}{(1 - q^{4\ell_1 + 2})}.
\end{equation}
We identify $\mathcal{A}_q(G)$ with $\mathcal{A}_q(SU(2))\otimes \mathcal{A}_q(SU(2))$.
Then the functional $h=h_0\otimes h_0$ is the Haar functional on $\mathcal{A}_q(G)$.  
We can identify the analogue of the algebra of bi-$K$-invariant polynomials on $G$ as
the algebra generated by the self-adjoint element $\psi$, and give the analogue 
of the restriction of the invariant integration in Lemma \ref{lemma:sphericalelements}. 

\begin{lemma}\label{lemma:sphericalelements}
A functional $\tau\colon \Uq(\mathfrak{g})\to \CC$ that satisfies the transformation behavior $\tau((AXA^{-1})YZ) = \varepsilon(X) \psi(Y) \varepsilon(Z)$ for all $X, Z \in \mathcal{B}$ 
and all $Y \in \Uq(\mathfrak{g})$ is a polynomial in $\psi$ as in \eqref{eq:expressionpsi}. Moreover, 
the Haar functional on the $\ast$-algebra $\CC[\psi]\subset \mathcal{A}_q(G)$ is given 
by 
\begin{equation*}
h(p(\psi)) = \frac{2}{\pi} \int_{-1}^1 p(x) \, \sqrt{1-x^2}\, dx. 
\end{equation*}
\end{lemma}

\begin{proof}
From Proposition \ref{prop:usefull}(iii) and \eqref{eq:invarianceproppsi}
we see that any functional on $\Uq(\mathfrak{g})$ satisfying the invariance property 
is a polynomial in $\psi$, since this space is spanned by 
$\Phi^0_{\frac{1}{2}n, \frac{1}{2}n}\cdot A^{-1}$ for $n\in \mathbb{N}$.
From Proposition \ref{prop:sphericalcaseasChebyshevpols} we have 
$\Phi^0_{\frac12 n,\frac12 n}\cdot A^{-1}$ is a multiple of
$U_n(\psi)$, since the right action of $A^{-1}$ is an algebra homomorphism. 
The Schur orthogonality relations give
$$h((\Phi^0_{\frac12 n,\frac12 n}\cdot A^{-1}) (\Phi^0_{\frac12 m,\frac12 m}\cdot A^{-1})^\ast)=0,\qquad \text{for }m\not= n,$$
and since the argument of $h$ is polynomial in $\psi$, we see that it has to correspond to
the orthogonality relations \eqref{eq:defChebyshevpols} for the Chebyshev polynomials.
So we find the expression for the Haar functional on the $\ast$-algebra
generated by $\psi$. 
\end{proof}

\begin{theorem} \label{thm:bi-B-inv}
Assume 
$\Psi, \Phi \colon \Uq(\mathfrak{g}) \to \End(\mathcal{H}^{\ell})$ are spherical functions of 
type $\ell$, see Definition \ref{def:sphericalfunctionoftypeell}. 
Then the map
\begin{equation*}
\tau \colon \Uq(\mathfrak{g}) \to \CC, 
\qquad X \mapsto \tr\bigl( (\Psi\cdot A^{-1}) (\Phi\cdot A^{-1})^\ast\bigr) (X),
\end{equation*}
satisfies $\tau((AXA^{-1})YZ) = \varepsilon(X) \psi(Y) \varepsilon(Z)$ for all $X, Z \in \mathcal{B}$ 
and all $Y \in \Uq(\mathfrak{g})$.
\end{theorem}

In particular, any such trace is a polynomial in the generator $\psi$ by Lemma \ref{lemma:sphericalelements}. 

\begin{corollary}\label{cor:thmbi-B-inv} 
Fix $\ell\in \frac12\NN$, then for $k,p\in\{0,1,\cdots, 2\ell\}$, 
\begin{equation*}
W(\psi)_{k,p} := \tr\bigl( (\Phi^\ell_{\xi(0,k)}\cdot A^{-1}) (\Phi^\ell_{\xi(0,p)}\cdot A^{-1})^\ast\bigr) 
= \sum_{r=0}^{p\wedge k} \alpha_r(k,p) \, U_{k+p-2r}(\psi),
\end{equation*}
with $\bigl( W(\psi)_{k,p}\bigr)^\ast =W(\psi)_{p,k}$. 
\end{corollary}

\begin{proof}[Proof of Theorem \ref{thm:bi-B-inv}]
Write $\Psi = \sum_{m,n=-\ell}^\ell \Psi_{m, n}\otimes E^\ell_{m,n}$, 
$\Phi = \sum_{m,n=-\ell}^\ell \Phi_{m, n}\otimes E^\ell_{m,n}$, with $\Psi_{m, n}$ and $\Phi_{m, n}$ linear
functionals on $\Uq(\mathfrak{g})$, so that 
\begin{equation}\label{eq:traceexplicitexpression}
\begin{split}
\tr\bigl((\Psi\cdot A^{-1}) (\Phi\cdot A^{-1})^\ast\bigr) &= 
\sum_{m, n=-\ell}^\ell (\Psi_{m, n}\cdot A^{-1}) (\Phi_{m, n}\cdot A^{-1})^\ast  \\ 
\Longrightarrow\quad \tau(Y) =  \sum_{m, n=-\ell}^\ell \sum_{(Y)} &\Psi_{m, n}(A^{-1}Y_{(1)}) \overline{\Phi_{m, n}(A^{-1} S(Y_{(2)})^\ast)},
\end{split}
\end{equation}
using the standard notation $\Delta(Y)= \sum_{(Y)} Y_{(1)}\otimes  Y_{(2)}$ and $\xi^\ast(X)= \overline{\xi(S(X)^\ast)}$ 
for $\xi\colon \Uq(\mathfrak{g})\to \mathbb{C}$ and $X\in \Uq(\mathfrak{g})$. 

Let $Z \in \mathcal{B}$ and $Y \in \Uq(\mathfrak{g})$, 
then we find
\begin{equation*}
\begin{split}
\tau(YZ) &= \sum_{m, n=-\ell}^\ell \sum_{(Y), (Z)} \Psi_{m, n}(A^{-1}Y_{(1)} Z_{(1)})\,
  \overline{\Phi_{m, n} (A^{-1}S(Y_{(2)})^\ast S(Z_{(2)})^\ast)}. \\
\end{split}
\end{equation*}
Since $\mathcal{B}$ is a right coideal, we can assume that $Z_{(1)}\in \mathcal{B}$, so, 
by the transformation property \eqref{eq:invarianceproperties}, 
$\Psi_{m, n}(A^{-1} Y_{(1)} Z_{(1)})= \sum_{k=-\ell}^\ell \Psi_{m,k}(A^{-1}Y_{(1)})t^\ell_{k,n}(Z_{(1)})$. 
Using this, and $t^\ell_{k,n}(Z_{(1)})= \overline{t^\ell_{n,k}(Z_{(1)}^\ast)}$ by the 
$\ast$-invariance of $\BB$ and the unitarity of $t^\ell$, and next move this to 
the $\Phi$-part, and summing over $n$  
and the transformation property \eqref{eq:invarianceproperties} for $\Phi$, we obtain 
\begin{equation*}
\begin{split}
\tau(YZ) &= \sum_{m,k=-\ell}^\ell \sum_{(Y), (Z)} \Psi_{m,k}(A^{-1}Y_{(1)}) \\
  &\qquad \qquad \times \sum_{n=-\ell}^\ell \overline{ \Phi_{m, n} (A^{-1} S(Y_{(2)})^\ast S(Z_{(2)})^\ast)t^\ell_{n,k}(Z_{(1)}^\ast)} \\
&= \sum_{n,k=-\ell}^\ell \sum_{(Y)} \Psi_{k, n}(A^{-1}Y_{(1)})\,
  \overline{\Phi_{k, n} \bigl( A^{-1} S(Y_{(2)})^\ast \sum_{(Z)} S(Z_{(2)})^\ast Z_{(1)}^\ast\bigr)} \\
&= \varepsilon(Z)\sum_{n,k=-\ell}^\ell \sum_{(Y)} \Psi_{k, n}(A^{-1} Y_{(1)})\,
  \overline{\Phi_{k, n} (A^{-1} S(Y_{(2)})^\ast)} = \varepsilon(Z) \tau(Y),
\end{split}
\end{equation*}
using $\sum_{(Z)} S(Z_{(2)})^\ast Z_{(1)}^\ast = 
\bigl(\sum_{(Z)}Z_{(1)} S(Z_{(2)})\bigr)^\ast= \overline{\varepsilon(Z)}$ 
by the antipode axiom in a Hopf algebra.

For the invariance property from the left, we proceed similarly using that $A$ is a group-like 
element. So for $Y\in \Uq(\mathfrak{g})$ and $X\in \mathcal{B}$ we have 
\begin{equation*}
\begin{split}
\tau(AXA^{-1}Y) &= \sum_{m, n=-\ell}^\ell \sum_{(X), (Y)} 
\Psi_{m, n}(X_{(1)}A^{-1}Y_{(1)}) \\
  &\qquad \qquad \times \overline{\Phi_{m, n} (A^{-1} S(A)^\ast S(X_{(2)})^\ast S(A^{-1})^\ast S(Y_{(2)})^\ast)}.
\end{split}
\end{equation*}
Now $S(A)^\ast=A^{-1}$, $S(A^{-1})^\ast=A$.    
Proceeding as in the previous paragraph using $X_{(1)}\in\BB$ we obtain
\begin{equation*}
\begin{split}
\tau(AXA^{-1}Y) &= \sum_{n, k=-\ell}^\ell \sum_{(X), (Y)} 
\Psi_{k, n}(A^{-1} Y_{(1)}) \\
&\qquad \times \sum_{m=-\ell}^\ell \overline{t^\ell_{k,m}(X_{(1)}^\ast)} \Phi_{m, n} (A^{-2} S(X_{(2)})^\ast A S(Y_{(2)})^\ast)  \\
= \sum_{n, k=-\ell}^\ell & \sum_{(Y)} 
\Psi_{k,n}(A^{-1} Y_{(1)}) \sum_{(X)} \overline{\Phi_{k,n} (X_{(1)}^\ast A^{-2} S(X_{(2)})^\ast A S(Y_{(2)})^\ast )}.
\end{split} 
\end{equation*}
The result follows if we prove $\sum_{(X)} X_{(1)}^\ast A^{-2} S(X_{(2)})^\ast A = A^{-1} \overline{\varepsilon(X)}$. 
In order to prove this we need the observation that $S(X^\ast) = A^{-2} S(X)^\ast A^2$
for all $X\in \Uq(\mathfrak{g})$, 
which can be verified on the generators and follows since the operators are antilinear homomorphisms,
see Remark \ref{rmk:thm:bi-B-inv}.
Now we obtain the required identity;
\begin{equation*}
\sum_{(X)} X_{(1)}^\ast A^{-2} S(X_{(2)})^\ast A = \sum_{(X)} X_{(1)}^\ast S(X_{(2)}^\ast) A^{-1} = \varepsilon(X^\ast) A^{-1}
= \overline{\varepsilon(X)} A^{-1}. 
\qedhere
\end{equation*}
\end{proof}

\begin{remark}\label{rmk:thm:bi-B-inv}
The required identity $S(X^\ast) = A^{-2} S(X)^\ast A^2$
for all $X\in \Uq(\mathfrak{g})$ can be generalised to arbitrary semisimple $\mathfrak{g}$. 
Indeed, since the square of the antipode $S$ is given by conjugation with an explicit element
of the Cartan subalgebra associated to $\rho =\frac12 \sum_{\alpha>0} \alpha$,  
see e.g. \cite[Ex.~4.1.1]{KoroV}, and since in a Hopf $\ast$-algebra $S\circ \ast$ is an involution, 
we find that $S(X^\ast) = K_{-\rho} S(X)^\ast K_{\rho}$ if $S^2(X) = K_{-\rho} X K_{\rho}$ 
as in \cite[Ex.~4.1.1]{KoroV}. 
\end{remark}

\begin{proof}[Proof of Corollary \ref{cor:thmbi-B-inv}] 
By Theorem \ref{thm:bi-B-inv} and Lemma \ref{lemma:sphericalelements}, the trace is a linear
combination of $\Phi^0_{\frac12 n,\frac12 n}\cdot A^{-1}$, hence a polynomial in $\psi$. To obtain the 
expression we need to find those $n$'s for which $\Phi^0_{\frac12 n,\frac12 n}\cdot A^{-1}$ occurs in 
$W(\psi)_{k,p}$ by Proposition \ref{prop:sphericalcaseasChebyshevpols}. 
Using \eqref{eq:defxi}, \eqref{eq:traceexplicitexpression} and Proposition \ref{prop:usefull}, we find that in 
$W(\psi)_{k,p}$ 
matrix only matrix elements of the form 
$t^{\frac12 k}_{a_1,b_1}(t^{\frac12 p}_{c_1,d_1})^\ast\otimes t^{\ell- \frac12 k}_{a_2,b_2}(t^{\ell-\frac12 p}_{c_2,d_2})^\ast$ occur.
Using the Clebsch-Gordan decomposition we see that $t^{\frac12 k}_{a_1,b_1}(t^{\frac12 p}_{c_1,d_1})^\ast$ 
and similarly $t^{\ell- \frac12 k}_{a_2,b_2}(t^{\ell-\frac12 p}_{c_2,d_2})^\ast$ can be written as 
a sum of matrix elements from $t^{\frac12 (k+p)-r}$, $r\in\{0,1,\cdots, k\wedge p\}$, and similarly 
$t^{2\ell- \frac12 (k+p)-s}$, $s\in\{0,1,\cdots, (2\ell-k)\wedge(2\ell-p)\}$. 
By Proposition \ref{prop:sphericalcaseasChebyshevpols} and Proposition \ref{prop:usefull}, the only $n$'s that 
can occur are $n=k+p-2r$, $r\in\{0,1,\cdots, k\wedge p\}$. 

The statement on the adjoint follows immediately from \eqref{eq:traceexplicitexpression} and $\psi$ being self-adjoint, see
\eqref{eq:expressionpsi}.
\end{proof}

We now can start the first part of the proof of Theorem \ref{thm:ortho}, except for the fact that we have to determine 
certain constants. This is contained in Lemma \ref{lem:ortho-cgc-1}. 

\begin{lemma} \label{lem:ortho-cgc-1}
We have
\begin{align*}
\sum_{i = -\ell}^{\ell} 
\sum_{m_1 = -\ell_1}^{\ell_1} 
\sum_{m_2 = -\ell_2}^{\ell_2}
\left( C^{\ell_1, \ell_2, \ell}_{m_1, m_2, i}\right)^2 q^{2(m_1 + m_2)}
= q^{-2\ell} \frac{
  (1 - q^{4\ell + 2})
}{
  (1 - q^2)
}. 
\end{align*}
\end{lemma}

The proof of Lemma \ref{lem:ortho-cgc-1} is a calculation using Theorem \ref{thm:spherical_recurrence}, 
which we postpone to Subsection \ref{subsec:ortho-calc}.

\begin{proof}[First part of the proof of Theorem \ref{thm:ortho}]
Using the notation of Section \ref{sec:mainresults} we find from 
Corollary \ref{cor:thm:spherical_recurrence} and \eqref{eq:defpolsPn} 
\begin{equation*} 
\begin{split}
&\tr \bigl( (\Phi^\ell_{\xi(n,i)}\cdot A^{-1}) (\Phi^\ell_{\xi(m,j)}\cdot A^{-1})^\ast\bigr) \\
&\quad = \sum_{k=0}^{2\ell} \sum_{p=0}^{2\ell} r^{\ell,k}_{n,i}(\psi) 
\tr \bigl( (\Phi^\ell_{\xi(0,k)}\cdot A^{-1}) (\Phi^\ell_{\xi(0,p)}\cdot A^{-1})^\ast\bigr) 
\overline{r^{\ell,p}_{m,j}}(\psi) 
\\ 
&\quad = \sum_{k=0}^{2\ell} \sum_{p=0}^{2\ell} P_n(\psi)^\ast_{i,k} W(\psi)_{k,p}  P_m(\psi)_{p,j}
=  \Bigl( (P_n(\psi))^\ast W(\psi) P_m(\psi)\Bigr)_{i,j} 
\end{split}
\end{equation*}
where we use that the action by $A^{-1}$ from the right is an algebra homomorphism, since $A^{-1}$ is group like,
and that $\psi$ is self-adjoint. 
By Proposition \ref{prop:usefull}, Lemma \ref{lemma:sphericalelements} and \eqref{eq:HaaronquantumSU(2)} we have
\begin{equation}\label{eq:matrixentryoforthogonalityastrace}
\begin{split}
&\frac{2}{\pi} \int_{-1}^1\Bigl( (P_n(x))^\ast W(x) P_m(x)\Bigr)_{i,j} \sqrt{1-x^2}\, dx \\
&\quad =  h\Bigl( \tr \bigl( (\Phi^\ell_{\xi(n,i)}\cdot A^{-1}) (\Phi^\ell_{\xi(m,j)}\cdot A^{-1})^\ast\bigr) \Bigr) \\
&\quad =\, \delta_{n,m}\delta_{i,j}
\frac{q^{2\ell+2n}(1-q^2)^2}{(1-q^{2n+2i+2})(1-q^{4\ell +2n-2i+2})} \\
&\qquad \times \left( \sum_{r=-\ell}^\ell \sum_{a=-\frac12 (n+i)}^{\frac12 (n+i)}
\sum_{b=-\ell -\frac12 (n-i)}^{\ell +\frac12 (n-i)}|C^{\frac12 (n+i), \ell +\frac12 (n-i), \ell}_{a,b,r}|^2 q^{2(a+b)}\right)^2
\end{split}
\end{equation}
after a straightforward calculation. 
Plugging in Lemma \ref{lem:ortho-cgc-1} in \eqref{eq:matrixentryoforthogonalityastrace}, and rewriting 
proves the result using Corollary \ref{cor:thmbi-B-inv}.
\end{proof}

Note that we have not yet determined the explicit values of $\alpha_t(m,n)$ in
Theorem \ref{thm:ortho} and we have not shown that $W$ is a matrix-valued weight function in the sense of Section
\ref{sec:genMVOP}. 
The values of the constants $\alpha_t(m,n)$ will be determined in Section \ref{subsec:weight} 
and the positivity of $W(x)$ for $x\in (-1,1)$ will follow from Theorem \ref{thm:ldu}.

\subsection{$q$-Difference equations} 
It is well-known, see e.g. Koornwinder \cite{Koor1993}, Letzter \cite{Letz04},  Noumi \cite{Noum}, that
the centre of the quantised enveloping algebra can be used to determine a commuting family of 
$q$-difference operators to which the corresponding spherical functions are eigenfunctions.
In this subsection we derive the matrix-valued $q$-difference operators corresponding to
central elements to which we find
matrix-valued eigenfunctions. 

The centre of $\Uq(\mathfrak{g})$ is generated by two Casimir elements, see 
Section \ref{sec:quantizeduniversalenvelopingalg};
\begin{equation*} 
  \Omega_1 = \frac{q K_{1}^{-1} + q^{-1} K_{1} - 2}{(q - q^{-1})^2} + E_1 F_1, \qquad
  \Omega_2 = \frac{q K_{2}^{-1} + q^{-1} K_{2} - 2}{(q - q^{-1})^2} + E_2 F_2.
\end{equation*}
Because of Proposition \ref{prop:usefull} and \eqref{eq:CasimirUqsu2} we find 
\begin{equation}\label{eq:actionOMEGA12onsphericalfunctionsell12}
\Omega_i\cdot \Phi^\ell_{\ell_1,\ell_2} = 
 \Phi^\ell_{\ell_1,\ell_2} \cdot\Omega_i = \left( \frac{q^{-\frac12 -\ell_i}-q^{\frac12 + \ell_i}}{q^{-1}-q}\right)^2 \Phi^\ell_{\ell_1,\ell_2}, \qquad i=1,2.
\end{equation}
 
The goal is to compute the radial parts of the Casimir elements acting on 
arbitrary spherical functions of type $\ell$ in terms of an explicit $q$-difference operator. 
In order to derive such a $q$-difference operator, we find a $\mathcal{B}\mathcal{A}\mathcal{B}$-decomposition
for suitable elements in $\Uq(\mathfrak{g})$ 
in Proposition \ref{prop:BABcasimir}. 
This special case of a $\mathcal{B}\mathcal{A}\mathcal{B}$-decomposition is the analogue of the 
$KAK$-decomposition, which has
not a general quantum algebra analogue. 
For this purpose, we first establish Lemma \ref{lem:cm}, which can be viewed as a quantum analogue of 
\cite[Lemma~2.2]{CassM}, and gives the $\mathcal{B}\mathcal{A}\mathcal{B}$-decomposition of $F_2A^{\lambda}$. 

\begin{lemma} \label{lem:cm}
Recall $A=K_1^{1/2}K_2^{1/2}$. For $\lambda \in \ZZ\setminus\{0\}$ we have
\[
F_2 A^{\lambda}  = \frac{q^{-1}}{q^{1-\lambda} - q^{1+\lambda}} \left(
    K^{1/2} A^{\lambda} B_1 - q^{\lambda} K^{1/2} B_1 A^{\lambda}
  \right).
\]
\end{lemma}

\begin{proof}
Recall Definition \ref{def:coidealB}, so the result follows from 
\begin{equation*}
\begin{split}
A^{\lambda} B_1 
  &= q^{-1} A^{\lambda} K_1^{-1/2} K_2^{-1/2} E_1 + q A^{\lambda} F_2 K_1^{-1/2} K_2^{1/2} \\
  &= q(q^{1-\lambda} - q^{1+\lambda}) K_1^{-1/2} K_2^{1/2} F_2 A^{\lambda} + q^{\lambda} B_1 A^{\lambda},
\end{split}
\end{equation*}
and using $K^{1/2}= K_1^{1/2} K_2^{-1/2}\in \mathcal{B}$. 
\end{proof}

\begin{proposition} \label{prop:BABcasimir}
The $\mathcal{B} \mathcal{A} \mathcal{B}$-decomposition for the Casimir elements $\Omega_1$ and $\Omega_2$ is given by
\begin{equation*}
\begin{split}
\Omega_1 A^{\lambda} =& \frac{q(1-q^{2\lambda+4})}{(1-q^2)^2 (1-q^{2\lambda+2})} K^{1/2}A^{\lambda+1} 
- \frac{2q^2}{(1 - q^2)^2} A^{\lambda} \\
&+ \frac{q^3(1-q^{2\lambda})}{(1-q^2)^2 (1-q^{2\lambda+2})} K^{-1/2}A^{\lambda-1}
- \frac{q^{2\lambda+1}}{(1-q^{2\lambda+2})^2} B_1K^{-1/2}B_2 A^{\lambda+1} \\
&- \frac{q^{2\lambda+2}}{(1-q^{2\lambda+2})^2} A^{\lambda+1}K^{-1/2}B_2 B_1
+ \frac{q^\lambda}{(1-q^{2\lambda+2})^2} B_1K^{-1/2} A^{\lambda+1}B_2 \\
&+ \frac{q^{3\lambda+3}}{(1-q^{2\lambda+2})^2} K^{-1/2}B_2 A^{\lambda+1}B_1,
\end{split}
\end{equation*}
and
\begin{equation*}
\begin{split}
\Omega_2 A^{\lambda} =& \frac{q(1-q^{2\lambda+4})}{(1-q^2)^2 (1-q^{2\lambda+2})} K^{-1/2}A^{\lambda+1} 
- \frac{2q^2}{(1 - q^2)^2} A^{\lambda} \\
&+ \frac{q^3(1-q^{2\lambda})}{(1-q^2)^2 (1-q^{2\lambda+2})} K^{1/2}A^{\lambda-1}
- \frac{q^{2\lambda+1}}{(1-q^{2\lambda+2})^2} B_2K^{1/2}B_1 A^{\lambda+1}  \\
&- \frac{q^{2\lambda+2}}{(1-q^{2\lambda+2})^2} A^{\lambda+1}K^{1/2}B_1 B_2
+ \frac{q^\lambda}{(1-q^{2\lambda+2})^2} B_2K^{1/2} A^{\lambda+1}B_1 \\
&+ \frac{q^{3\lambda+3}}{(1-q^{2\lambda+2})^2} K^{1/2}B_1 A^{\lambda+1}B_2.
\end{split}
\end{equation*}
\end{proposition}

\begin{proof}
We first concentrate on $\Omega_2$, the statement for $\Omega_1$ follows by flipping the order 
using $\sigma$ as in Remark \ref{rmk:identification}(iv). 
We need to rewrite $E_2F_2A^\lambda$, which we do in terms of $F_2A^\lambda$ and next using 
Lemma \ref{lem:cm}. The details are as follows. 

Using Definition \ref{def:coidealB}, the commutation relations and pulling through $F_2$ to the left,  we get
\begin{equation}\label{eq:let1}
B_2F_2A^\lambda = q^{-1} E_2F_2 A^{\lambda-1} + q^2 F_1F_2  K^{1/2} A^\lambda.
\end{equation}
Similarly, and only slightly more involved, we obtain
\begin{equation}\label{eq:let2}
F_2A^\lambda B_2 = q^{\lambda-2} E_2F_2 A^{\lambda-1} - q^{\lambda-2} \frac{K_2 -K_2^{-1}}{q-q^{-1}} A^{\lambda-1} 
+ q^{1-\lambda} F_1F_2 K^{1/2}  A^\lambda.
\end{equation}
Using \eqref{eq:let1} and \eqref{eq:let2} we eliminate the term with $F_1F_2$, and shifting $\lambda$ to $\lambda+1$ gives  
\begin{equation}\label{eq:let3}
 (q^{-1}-q^{2\lambda+1}) E_2F_2A^\lambda = B_2F_2A^{\lambda+1} -q^{2+\lambda} F_2 A^{\lambda+1}B_2 
 - q^{2\lambda+1} \frac{K_2 -K_2^{-1}}{q-q^{-1}} A^{\lambda}.
\end{equation}
Apply Lemma \ref{lem:cm} on the first two terms on the right hand side of \eqref{eq:let3} 
and note that the remaining terms in \eqref{eq:let3} and in $\Omega_2 A^\lambda$ can be dealt with by observing that 
$K_2=K^{-1/2}A= AK^{-1/2}$. Taking corresponding terms together proves the $\mathcal{B}\mathcal{A}\mathcal{B}$-decomposition
for $\Omega_2 A^\lambda$ after a short calculation. 
\end{proof}

Our next task is to translate Proposition \ref{prop:BABcasimir} into an operator for spherical functions
of type $\ell$, from which we derive eventually, see Theorem \ref{thm:diff_eqn_Pn}, an Askey-Wilson $q$-difference type operator
for the matrix-valued orthogonal polynomials $P_n$. 
Let $\Phi$ be a spherical function of type $\ell$, then we immediately obtain from
Proposition \ref{prop:BABcasimir} 
\begin{equation}\label{eq:actionCasimirAlambdaonPhi}
\begin{split}
\Phi(\Omega_1 A^{\lambda}) &=
\frac{q(1-q^{2\lambda+4})}{(1-q^2)^2 (1-q^{2\lambda+2})} t^\ell(K^{1/2})\Phi(A^{\lambda+1}) 
- \frac{2q^2}{(1 - q^2)^2} \Phi(A^{\lambda}) \\ &+ 
\frac{q^3(1-q^{2\lambda})}{(1-q^2)^2 (1-q^{2\lambda+2})} t^\ell(K^{-1/2})\Phi(A^{\lambda-1}) \\
&- \frac{q^{2\lambda+1}}{(1-q^{2\lambda+2})^2} t^\ell(B_1K^{-1/2}B_2) \Phi(A^{\lambda+1})  \\
&- \frac{q^{2\lambda+2}}{(1-q^{2\lambda+2})^2} \Phi(A^{\lambda+1}) t^\ell(K^{-1/2}B_2 B_1) \\
&+ \frac{q^\lambda}{(1-q^{2\lambda+2})^2} t^\ell(B_1K^{-1/2}) \Phi(A^{\lambda+1})t^\ell(B_2) \\ 
&+ \frac{q^{3\lambda+3}}{(1-q^{2\lambda+2})^2} t^\ell(K^{-1/2}B_2) \Phi(A^{\lambda+1})t^\ell(B_1).
\end{split}
\end{equation}
The analogous expression for $\Phi(\Omega_2 A^{\lambda})$ of \eqref{eq:actionCasimirAlambdaonPhi} can be obtained
using the flip $\sigma$, see Remark \ref{rmk:identification}(iv). In particular, it suffices to 
replace all $t^\ell(X)$ in \eqref{eq:actionCasimirAlambdaonPhi} by $J^\ell t^\ell(X)J^\ell$, see Remark \ref{rmk:identification}(iv),
to get the corresponding expression. 

By Proposition \ref{prop:usefull}(ii) we know that $\Phi(A^{\lambda})$ is diagonal. Note that $\Phi\cdot \Omega_1: \Uq\to \End(H_\ell)$ is also a spherical function of type $\ell$ by the centrality of the Casimir operator $\Omega_1$. Hence $(\Phi\cdot\Omega_1)(A^\lambda)=\Phi(\Omega_1A^\lambda)$ is diagonal, which can also be seen directly from \eqref{eq:actionCasimirAlambdaonPhi}.
We can calculate the matrix entries of $\Phi(\Omega_1 A^\lambda)$ using the upper triangular matrices $t^\ell(B_1K^{-1/2})$, $t^\ell(B_1)$ having only non-zero entries on the superdiagonal, the lower triangular matrices $t^\ell(K^{-1/2}B_2)$, $t^\ell(B_2)$ having only non-zero entries on the subdiagonal and the diagonal matrices $t^\ell(K^{\pm1/2})$, $t^\ell(B_1K^{-1/2}B_2)$, $t^\ell(K^{-1/2}B_2B_1)$, see \eqref{eq:representationsB}, \eqref{eq:defrelationsUqsl2}.

For $\Phi$ a spherical function of type $\ell$ we view the diagonal restricted to $\mathcal{A}$ as a 
vector-valued function $\hat{\Phi}$; 
\begin{equation*} 
\hat\Phi \colon \mathcal{A} \to \mathcal{H}^\ell, 
\qquad A^\lambda \mapsto \sum_{m=-\ell}^\ell \Phi(A^\lambda)_{m,m}\, e^\ell_m.
\end{equation*}
So we can regard the Casimir elements as acting on the vector-valued function $\hat{\Phi}$, 
and the action of the Casimir is made explicit in Proposition \ref{prop:diffeqforhatPhi}. 

\begin{proposition} \label{prop:diffeqforhatPhi}
Let $\Phi$ be a spherical function of type $\ell$, with corresponding vector-valued function
$\hat\Phi \colon \mathcal{A} \to \mathcal{H}^\ell$ 
representing the diagonal when restricted to $\mathcal{A}$. Then 
\begin{equation*}
(\widehat{\Phi\cdot\Omega_1})\, (A^{\lambda}) = M_1^\ell(q^{\lambda}) \hat\Phi(A^{\lambda + 1}) - \frac{2q^2}{(1-q^2)^2}\hat\Phi(A^{\lambda}) 
+ N_1^\ell(q^\lambda) \hat\Phi(A^{\lambda - 1}),
\end{equation*}
where $M_1^\ell(z)$ is a tridiagonal and $N_1^\ell(z)$ is a diagonal matrix with respect to 
the basis $\{e^\ell_n\}_{n=-\ell}^\ell$ of $\mathcal{H}^\ell$ with coefficients
\begin{align*}
(N_1^\ell(z))_{m, m} &= \frac{q^{3+m} (1-z^2)}{(1-q^2)^2 (1-q^{2}z^2)}, \\
(M_1^\ell(z))_{m, m+1} &= \frac{zq^{m+1}}{(1 - q^2z^2)^2}(b^{\ell}(m+1))^2, \\
(M_1^\ell(z))_{m, m} &= \frac{q^{1-m} (1-q^4z^2)}{(1-q^2)^2 (1-q^{2}z^2)} \\
 &\qquad - \frac{z^2q^{2+m}}{(1-q^{2}z^2)^2} \left( (b^\ell(m))^2+ (b^\ell(m+1))^2\right), \\
(M_1(z))_{m, m-1} &= \frac{z^3q^{m+3}}{(1 - q^2z^2)^2}(b^{\ell}(m))^2.
\end{align*}
Moreover, for $\Omega_2$ the action is
\begin{equation*}
\widehat{(\Phi\cdot\Omega_2)} (A^{\lambda})  = M_2^\ell(q^{\lambda}) \hat\Phi(A^{\lambda + 1}) 
- \frac{2q^2}{(1-q^2)^2}\hat\Phi(A^{\lambda}) 
+ N_2^\ell(q^\lambda) \hat\Phi(A^{\lambda - 1}),
\end{equation*}
where $M_2^\ell(z)=J^\ell M_1^\ell(z) J^\ell$ is a tridiagonal, and $M_2^\ell(z)=J^\ell M_1^\ell(z) J^\ell$ is a diagonal matrix.
\end{proposition}

\begin{remark} 
The proof of Theorem \ref{thm:diff_eqn_Pn} does not explain why the matrix coefficients of
$\eta_q$ and $\eta_{q^{-1}}$ in Theorem \ref{thm:diff_eqn_Pn} are related by $z\leftrightarrow z^{-1}$.
In Proposition \ref{prop:diffeqforhatPhi} there is a lack of symmetry between the up and down shift in $\lambda$, and only after suitable multiplication with $\hat{\Phi}_0$ from the left and right the symmetry of
Theorem \ref{thm:diff_eqn_Pn} pops up.
It would be desirable to have an explanation of this symmetry from the quantum group theoretic interpretation.
Note that the symmetry can be translated to the requirement
$\Psi(z) N_1(q^{-2}z^{-1}) = M_1(z) \Psi(qz)$ with $\Psi(q^\lambda) = \hat{\Phi}^\ell_0(A^{\lambda})\hat{\Phi}^\ell_0(A^{-1-\lambda})^{-1}$.
\end{remark} 

The remark following \eqref{eq:actionCasimirAlambdaonPhi} on how to switch to the 
second Casimir operator gives the conjugation between $M_1^\ell$ and $M_2^\ell$, respectively
$N_1^\ell$ and $N_2^\ell$. 
Note that in case $\ell=0$, $\Phi$ and $\hat\Phi$ are equal, and we find that Proposition \ref{prop:diffeqforhatPhi} gives the operator 
\begin{equation*} 
\begin{split}
\Phi(\Omega_2 A^{\lambda}) = \Phi(\Omega_1 A^{\lambda})
&= \frac{q(1-q^{2\lambda+4})}{(1-q^2)^2 (1-q^{2\lambda+2})} \Phi(A^{\lambda+1}) 
- \frac{2q^2}{(1 - q^2)^2} \Phi(A^{\lambda}) \\
&\qquad + \frac{q^3(1-q^{2\lambda})}{(1-q^2)^2 (1-q^{2\lambda+2})} \Phi(A^{\lambda-1}) 
\end{split}
\end{equation*}
for $\Phi\colon \Uq(\mathfrak{g})\to \CC$ a spherical function (of type $0$), which
should be compared to \cite[Lemma~5.1]{Koor1993}, see also \cite{Letz04}, \cite{Noum}.

\begin{proof} 
Consider \eqref{eq:actionCasimirAlambdaonPhi} and calculate the $(m,m)$-entry. 
Using the explicit expressions for the elements $t^\ell(X)$ for $X\in \mathcal{B}$,
see \eqref{eq:representationsB}, \eqref{eq:defrelationsUqsl2}, in \eqref{eq:actionCasimirAlambdaonPhi},
we find 
\begin{equation*}
\begin{split}
\Phi(\Omega_1 A^{\lambda})_{m,m} 
  &= \frac{q^{3+m}(1-q^{2\lambda})}{(1-q^2)^2 (1-q^{2\lambda+2})} \Phi(A^{\lambda-1})_{m,m}
    - \frac{2q^2}{(1 - q^2)^2} \Phi(A^{\lambda})_{m,m} \\  
  & + \frac{q^{1-m}(1-q^{2\lambda+4})}{(1-q^2)^2 (1-q^{2\lambda+2})} \Phi(A^{\lambda+1})_{m,m} \\
  & - \frac{q^{2\lambda+2+m} (b^\ell(m+1))^2}{(1-q^{2\lambda+2})^2} \Phi(A^{\lambda+1})_{m,m} \\
  & - \frac{q^{2\lambda+2+m} (b^\ell(m))^2}{(1-q^{2\lambda+2})^2} \Phi(A^{\lambda+1})_{m,m} \\
  & + \frac{q^{\lambda+m+1} (b^\ell(m+1))^2}{(1-q^{2\lambda+2})^2} \Phi(A^{\lambda+1})_{m+1,m+1} \\ 
  & + \frac{q^{3\lambda+3+m} (b^\ell(m))^2 }{(1-q^{2\lambda+2})^2} \Phi(A^{\lambda+1})_{m-1,m-1}, 
\end{split}
\end{equation*}
which we can rewrite as stated with the matrices $M_1^\ell(q^{\lambda})$ and $N_1^\ell(q^{\lambda})$.
The case for $\Omega_2$ follows from the observed symmetry, or it can be obtained by an analogous computation. 
\end{proof}

In particular, we can apply Proposition \ref{prop:diffeqforhatPhi} to $\Phi^\ell_{\xi(n,m)}$ using 
\eqref{eq:actionOMEGA12onsphericalfunctionsell12} to find an eigenvalue equation. In order to find
an eigenvalue equation for the matrix-valued polynomials, we first introduce the 
full spherical functions $\hat\Phi^\ell_n\colon \mathcal{A} \to \text{End}(\mathbb{C}^{2\ell+1})$ defined by
\begin{equation}\label{eq:deffullsphericalfunction}
\hat\Phi^\ell_n = \sum_{i,j=0}^{2\ell} (\hat\Phi^\ell_n)_{i,j} \otimes E_{i,j}, 
\ (\hat\Phi^\ell_n)_{i,j}(A^\lambda) =  (\Phi^\ell_{\xi(n,j)}(A^\lambda))_{i-\ell,i-\ell}
\end{equation}
for $n\in \mathbb{N}$. 
So we put the vectors $(\hat{\Phi}^\ell_{\xi(n,0)}, \cdots,\hat{\Phi}^\ell_{\xi(n,2\ell)})$ as columns in 
a matrix, and we relabel in order to have the matrix entries labeled by $i,j\in\{0,\cdots, 2\ell\}$. 
We reformulate Proposition \ref{prop:diffeqforhatPhi} and 
\eqref{eq:actionOMEGA12onsphericalfunctionsell12} as the eigenvalue equations 
\begin{equation}\label{eq:q_differenceeqnforPhin}
\begin{split}
\hat\Phi_n(A^\lambda) \Lambda_n(i) &= M_i(q^\lambda) \hat\Phi^\ell_n(A^{\lambda+1}) + N_i(q^\lambda) \hat\Phi^\ell_n(A^{\lambda-1}), \\
\Lambda_n(1) &= \sum_{j=0}^{2\ell} \frac{q^{1-n-j}+q^{3+n+j}}{(1-q^2)^2}E_{j,j}, 
\quad \Lambda_n(2) = J \Lambda_n(1)J,
\end{split}
\end{equation}
where $\bigl(M_i(z)\bigr)_{m,n}= \bigl(M_i^\ell(z)\bigr)_{m-\ell,n-\ell}$ for $m,n\in\{0,1,\cdots, 2\ell\}$, and similarly for $N_i$, 
are the matrices of Proposition \ref{prop:diffeqforhatPhi} shifted to the standard matrix with respect to the standard
basis $\{e_n\}_{n=0}^{2\ell}$ of $\mathbb{C}^{2\ell+1}$. Note that the symmetry of Proposition \ref{prop:diffeqforhatPhi} then
rewrites as $M_2(z) = J M_1(z)J$, $N_2(z) = J N_1(z)J$, with $J\colon e_n\mapsto e_{2\ell-n}$. 

Now we rewrite Corollary \ref{cor:thm:spherical_recurrence} after pairing with $A^\lambda$ as 
\begin{equation}\label{eq:fullsphericalispoltimeszerofullspherical}
\hat\Phi^\ell_n(A^\lambda) = \hat\Phi^\ell_0(A^\lambda)  \overline{P_n}\bigl( \mu(q^{\lambda+1})\bigr) \in \text{End}(\mathbb{C}^{2\ell+1}),
\end{equation}
where $\mu(x) =\frac12(x+x^{-1})$, using that pairing with the group-like elements $A^\lambda$ is a homomorphism and $\varphi(A^\lambda)=\mu(q^{\lambda+1})$ by
\eqref{eq:expressionvarphi}. 
Using \eqref{eq:fullsphericalispoltimeszerofullspherical} in \eqref{eq:q_differenceeqnforPhin} proves Corollary \ref{cor:diffeqPn}. 

\begin{corollary}\label{cor:diffeqPn}
Assuming $\hat\Phi^\ell_0(A^\lambda)$ is invertible, the matrix-valued polynomials $P_n$ satisfy the 
eigenvalue equations
\begin{equation*}
\begin{split}
P_n(\mu(q^{\lambda}))\, \Lambda_n(i) &= \tilde{M}_i(q^{\lambda})\, P_n(\mu(q^{\lambda+1})) + 
\tilde{N}_i(q^{\lambda})\, P_n(\mu(q^{\lambda-1})), \qquad i=1,2,
\end{split}
\end{equation*}
where
\begin{equation*}
\begin{split}
\tilde{M}_i(q^{\lambda}) &= \overline{\bigl( \hat\Phi^\ell_0(A^{\lambda-1})\bigr)^{-1} M_i(q^{\lambda-1}) \hat\Phi^\ell_0(A^{\lambda})}, \\
\tilde{N}_i(q^{\lambda}) &= \overline{\bigl( \hat\Phi^\ell_0(A^{\lambda-1})\bigr)^{-1} N_i(q^{\lambda-1}) \hat\Phi^\ell_0(A^{\lambda-2})}, 
\end{split}
\end{equation*}
and $\Lambda_n(i)$ are defined in \eqref{eq:q_differenceeqnforPhin} and $\mu(x) = \frac12(x+x^{-1})$. 
\end{corollary}

It remains to prove the assumption in Corollary \ref{cor:diffeqPn} for sufficiently many $\lambda$, and to
calculate the coefficients in the eigenvalue equations
explicitly. This is done in Section \ref{sec:qdifferenceoperators}. 

Having established \eqref{eq:fullsphericalispoltimeszerofullspherical}, we can prove Corollary \ref{cor:ortho-cgc-2} 
by considering coefficients in the Laurent expansion. 

\begin{proof}[Proof of Corollary \ref{cor:ortho-cgc-2}]
The left hand side of \eqref{eq:fullsphericalispoltimeszerofullspherical} can be expanded as 
a Laurent series in $q^\lambda$ by Proposition \ref{prop:usefull}(ii) and 
\eqref{eq:deffullsphericalfunction}. The leading coefficient of degree $\ell +n$ is 
an antidiagonal  matrix
\begin{equation*}
\bigl( \hat\Phi^\ell_n(A^\lambda)\bigr)_{i,j} = q^{\lambda(\ell+n)} 
\left( C^{\frac12(n+j), \ell+\frac12 (n-j), \ell}_{-\frac12(n+j), -\ell-\frac12 (n-j), i-\ell}\right)^2 
+ \text{lower order terms},
\end{equation*}
since the Clebsch-Gordan coefficient is zero unless $i+j=2\ell$. With a similar expression for  $\hat\Phi^\ell_0(A^\lambda)$ on the right hand side
and expanding $\overline{P_n}\bigl( \mu(q^{\lambda+1})\bigr) = \overline{\text{lc}(P_n)} q^{n(\lambda+1)} 2^{-n} + 
\text{lower order terms}$ gives
\begin{equation*}
\text{lc}(P_n)_{i,j} = \delta_{i,j} q^n 2^{-n} 
\left( C^{\frac12(n+j), \ell+\frac12 (n-j), \ell}_{-\frac12(n+j), -\ell-\frac12 (n-j), \ell-j}\right)^2
\left( C^{\frac12j, \ell-\frac12j, \ell}_{-\frac12 j, -\ell+\frac12 j, \ell-j}\right)^{-2},
\end{equation*}
and \eqref{eq:CGCforends} gives the result. 
\end{proof}

Corollary \ref{cor:ortho-cgc-2} gives $P_0(x)=I$, so Corollary \ref{cor:diffeqPn} gives
\begin{equation}\label{eq:Lambda0expressedinMN}
\Lambda_0(i) = \tilde{M}_i(q^{\lambda}) + 
\tilde{N}_i(q^{\lambda}), \qquad \forall\, \lambda \in \ZZ.
\end{equation}

\subsection{Symmetries}\label{subsec:symmetries}

Even though we can use the explicit expression of the weight $W$ to establish Proposition \ref{prop:commutant},
we show the occurrence of $J$ in the commutant from Remark \ref{rmk:def:elementarysphericalfunction}(ii).

Observe that $(\ell_1,\ell_2)=\xi(n,k)$ gives $(\ell_2, \ell_1)= \xi(n,2\ell-k)$ and that $\sigma(A)=A$, so 
Remark \ref{rmk:def:elementarysphericalfunction}(ii) yields 
$\bigl(\Phi^\ell_{\xi(n,i)}\cdot A^{-1}\bigr)(\sigma(Z)) = J^\ell \bigl(\Phi^\ell_{\xi(n,2\ell-i)}\cdot A^{-1}\bigr)(Z) J^\ell$
for any $Z\in \Uq(\mathfrak{g})$. Similarly, using moreover that $\sigma$ is a $\ast$-isomorphism and that $S$ and $\sigma$ commute, 
see \eqref{eq:defHopfstructureonUqg} we obtain 
$\bigl(\Phi^\ell_{\xi(n,j)}\cdot A^{-1}\bigr)^\ast (\sigma(Z)) = J^\ell \bigl(\Phi^\ell_{\xi(n,2\ell-i)}\cdot A^{-1}\bigr)^\ast (Z) J^\ell$.
Since $(\sigma\otimes\sigma)\circ \Delta = \Delta\circ \sigma$ and $(J^\ell)^2=1$, we find for $Z\in \Uq(\mathfrak{g})$ from 
these observations, cf. \eqref{eq:traceexplicitexpression}, 
\begin{equation*}
\begin{split}
&\tr \bigl( (\Phi^\ell_{\xi(n,i)}\cdot A^{-1}) (\Phi^\ell_{\xi(m,j)}\cdot A^{-1})^\ast\bigr) (\sigma(Z)) \\
&\quad = \tr \bigl( (\Phi^\ell_{\xi(n,2\ell-i)}\cdot A^{-1}) (\Phi^\ell_{\xi(m,2\ell-j)}\cdot A^{-1})^\ast\bigr) (Z).
\end{split}
\end{equation*}
By the first part of the proof of Theorem \ref{thm:ortho}
\begin{equation*}
\Bigl( (P_n(\psi))^\ast W(\psi) P_m(\psi)\Bigr)_{i,j}(\sigma(Z)) = 
\Bigl( (P_n(\psi))^\ast W(\psi) P_m(\psi)\Bigr)_{2\ell-i,2\ell-j}(Z).
\end{equation*}
Note that $\psi(\sigma(Z))= \psi(Z)$ by the symmetric expression of \eqref{eq:expressionpsi}. 
Again using $(\sigma\otimes\sigma)\circ \Delta = \Delta\circ \sigma$, we have
$p(\psi)(\sigma(Z))= p(\psi)(Z)$ for any polynomial $p$. It follows that 
\begin{equation}\label{eq:symmetryPnWPm}
\Bigl( (P_n(\psi))^\ast W(\psi) P_m(\psi)\Bigr)_{i,j} = 
\Bigl( (P_n(\psi))^\ast W(\psi) P_m(\psi)\Bigr)_{2\ell-i,2\ell-j}.
\end{equation}
For $n=m=0$, \eqref{eq:symmetryPnWPm} proves that $J$ is in the commutant algebra for $W$ as 
stated in Proposition \ref{prop:commutant}. 

Note that \eqref{eq:symmetryPnWPm}, after applying the Haar functional on the $\ast$-algebra generated by
$\psi$ as in Lemma \ref{lemma:sphericalelements}, also gives $JG_nJ=G_n$ as is immediately clear from the 
explicit expression for the squared norm matrix $G_n$ in Theorem \ref{thm:ortho} derived in 
Section \ref{subsec:orthogonality}. It moreover shows that $(JP_nJ)_{n\in\mathbb{N}}$ is 
a family of matrix-valued orthogonal polynomials with respect to the weight $W$. 
It is a consequence of Corollary \ref{cor:ortho-cgc-2}, see Section \ref{sec:genMVOP},  that $JP_nJ=P_n$, 
since $J \text{lc}(P_n)J = \text{lc}(P_n)$ by Corollary  \ref{cor:ortho-cgc-2}.

Note that we have now proved Proposition \ref{prop:commutant} except for the $\supset$-inclusion in 
the first line. This is done after the proof of Theorem \ref{thm:ortho} is completed at the end of 
Section \ref{subsec:weight}.

As a consequence of the discussion on symmetries, we can formulate the symmetry for $\hat{\Phi}^\ell_n$
in Lemma \ref{lem:symhatPhi}.

\begin{lemma}\label{lem:symhatPhi} With $J\colon e_n\mapsto e_{2\ell-n}$ we have
$J \hat{\Phi}^\ell_n(A^\lambda) J = \hat{\Phi}^\ell_n(A^\lambda)$ for all $\lambda \in \mathbb{Z}$. 
\end{lemma}

\begin{proof} This is a consequence  of initial observations in Section \ref{subsec:symmetries}. 
For $i,j\in\{0,\cdots, 2\ell\}$, using \eqref{eq:deffullsphericalfunction} and $\sigma(A^\lambda)=A^\lambda$ 
we obtain 
\begin{equation*}
\begin{split}
\bigl( J \hat{\Phi}^\ell_n(A^\lambda) J\bigr)_{i,j} 
&= \bigl( \hat{\Phi}^\ell_n(A^\lambda) \bigr)_{2\ell-i,2\ell-j} 
= \bigl( \Phi^\ell_{\xi(n,2\ell-j)}(A^\lambda) \bigr)_{\ell-i,\ell-i} \\
&= \bigl( J^\ell \Phi^\ell_{\xi(n,j)}(A^\lambda) J^\ell \bigr)_{\ell-i,\ell-i} 
= \bigl( \Phi^\ell_{\xi(n,j)}(A^\lambda) \bigr)_{i-\ell,i-\ell} \\
&= \bigl( \hat{\Phi}^\ell_n(A^\lambda)\bigr)_{i,j}.
\qedhere
\end{split}
\end{equation*}
\end{proof}

\section{The weight and orthogonality relations for the matrix-valued polynomials}\label{sec:weightorthorel}

In this section we complement the quantum group theoretic proofs of Section \ref{sec:quantumgrouprelatedpropssphericalf} 
of some of the statements of Section \ref{sec:mainresults} using mainly analytic techniques. 
In Section \ref{subsec:weight} we prove the statement on the expansion of the entries of the weight function 
in terms of Chebyshev polynomials of Theorem \ref{thm:ortho}. 
In Section \ref{subsec:ldu} we prove the LDU-decomposition of Theorem \ref{thm:ldu}. In 
Section \ref{subsec:ortho-calc} we prove Lemma \ref{lem:ortho-cgc-1} using a special case and induction
using Theorem \ref{thm:spherical_recurrence} in the induction step. 

\subsection{Explicit expressions of the weight}\label{subsec:weight}

In order to prove the explicit expansion of the matrix entries of the weight $W$ in terms of Chebyshev polynomials, we 
start with the expression of Corollary \ref{cor:thmbi-B-inv} for the matrix entries of the weight $W$.
After pairing with $A^\lambda$, we expand as a Laurent polynomial in $q^\lambda$ in 
Proposition \ref{prop:weightexpandedasLaurentpol}. 
Then we can use Lemma  \ref{lem:tediousequallity}, whose proof is presented in Appendix \ref{subapp:explicitweight}. 

\begin{proposition}\label{prop:weightexpandedasLaurentpol}
For $0\leq k,p\leq 2\ell$, $\lambda\in\mathbb{Z}$ we have 
\begin{equation*}
\begin{split}
&W(\psi)_{k,p}(A^\lambda) = \sum_{s=-\frac12(k+p)}^{\frac12(k+p)} d_s^\ell(k,p) q^{2s\lambda},
\qquad d_s^\ell(k,p) = d_{-s}^\ell(k,p), \\
&d_s^\ell(k,p) =  
\underset{s=j-i}{\sum_{i=-\frac12 k}^{\frac12 k}\sum_{j=-\frac12 p}^{\frac12 p}}
\sum_{n=-\ell}^\ell 
q^{2(i+j-n)+2(i+\frac12 k)(i-n+\ell-\frac12 k)+ 2(j+\frac12 p)(j-n+\ell-\frac12 p)} \\
& \qquad \qquad \qquad \times \frac{
  \qbin{k}{\frac12 k- i}_{q^2}
  \qbin{2\ell-k}{\ell-\frac12 k+n-i}_{q^2}
  \qbin{p}{\frac12 p- j}_{q^2}
  \qbin{2\ell-p}{\ell-\frac12 p+n-j}_{q^2}
}{
  \qbin{2\ell}{\ell-n}_{q^{2}}^{2}
}.
\end{split}
\end{equation*}
\end{proposition}

\begin{proof} We obtain from Corollary \ref{cor:thmbi-B-inv} 
\begin{equation}\label{eq:weightpairedAlambda}
\begin{split}
W(\psi)_{k,p}(A^\lambda) &= 
\sum_{m,n=-\ell}^\ell \bigl( \Phi^\ell_{\xi(0,k)}\cdot A^{-1}\bigr)_{m,n}(A^\lambda)\, 
\overline{\bigl( \Phi^\ell_{\xi(0,p)}\cdot A^{-1}\bigr)_{m,n}(A^{-\lambda})} \\
&= \sum_{m,n=-\ell}^\ell \bigl( \Phi^\ell_{\xi(0,k)}\bigr)_{m,n}(A^{\lambda-1})\, 
\overline{\bigl( \Phi^\ell_{\xi(0,p)}\bigr)_{m,n}(A^{-1-\lambda})}, 
\end{split}
\end{equation}
using that $A^\lambda$ is a group like element and $S(A^{\lambda})^\ast= A^{-\lambda}$. 
By Proposition \ref{prop:usefull}(ii) $W(\psi)_{k,p}(A^\lambda)$ is 
\begin{equation*}
\begin{split}
\sum_{n=-\ell}^\ell 
\sum_{i=-\frac12 k}^{\frac12 k}
\sum_{j=-\frac12 p}^{\frac12 p}
\left( C^{\frac12 k, \ell-\frac12 k,\ell}_{i,i-n,n} \right)^2
\left( C^{\frac12 p, \ell-\frac12 p,\ell}_{j,j-n,n} \right)^2
q^{2\lambda(j-i)} q^{2(i+j-n)},
\end{split}
\end{equation*}
where the Clebsch-Gordan coefficients are to be taken as zero in case 
$|i-n|>\ell-\frac12 k$, respectively $|j-n|>\ell-\frac12 p$. 
Now put $s=j-i$, then $s$ runs from $-\frac12(k+p)$ up to $\frac12(k+p)$, and 
we have the Laurent expansion 
\begin{equation*}
W(\psi)_{k,p}(A^\lambda) = \sum_{s=-\frac12(k+p)}^{\frac12(k+p)} d_s^\ell(k,p) q^{2s\lambda},
\end{equation*}
with 
\begin{equation*}
d_s^\ell(k,p) =  
\underset{s=j-i}{\sum_{i=-\frac12 k}^{\frac12 k}\sum_{j=-\frac12 p}^{\frac12 p}}
\sum_{n=-\ell}^\ell
\left( C^{\frac12 k, \ell-\frac12 k,\ell}_{i,i-n,n} \right)^2
\left( C^{\frac12 p, \ell-\frac12 p,\ell}_{j,j-n,n} \right)^2 q^{2(i+j-n)}.
\end{equation*}
Plugging in \eqref{eq:CGCinbottom} gives the explicit
expression for $d^\ell_s(k,p)$. 

In order to show that $d_s^\ell(p,k)= d_{-s}^\ell(p,k)$, we note that 
$p(\psi)(A^{\lambda})=
p(\mu(q^\lambda))$ is symmetric in $\lambda\leftrightarrow -\lambda$ for any polynomial $p$ and so
Corollary \ref{cor:thmbi-B-inv} implies the symmetry. 
\end{proof}

Note that for a matrix element of $P_n(\psi)^\ast W(\psi) P_m(\psi)$ a similar expression can be given, cf. the first part of the proof of Theorem \ref{thm:ortho}, but this is not required.
The symmetry in the Laurent expansion of $W(\psi)_{k,p}(A^{\lambda})$ does not seem to follow directly from known symmetries for the Clebsch-Gordan coefficients, see e.g. \cite[Ch.~3]{KlimS}. 

Now we can proceed with the proof of Theorem \ref{thm:ortho}, for which we need Lemma \ref{lem:tediousequallity}.

\begin{lemma} \label{lem:tediousequallity}
For $\ell \in \frac{1}{2} \NN$ and for $k,p\in\{0,\cdots,2\ell\}$ subject to $k\leq p$ and $k+p\leq 2\ell$ 
we have with the notation of Proposition \ref{prop:weightexpandedasLaurentpol} and 
Theorem \ref{thm:ortho},
\begin{equation*}
 \underset{2s=k+p-2(r+a)}{\sum_{r=0}^{k} \sum_{a=0}^{k+p-2r}} \alpha_r(k,p) = d^\ell_s(k,p).
\end{equation*}
\end{lemma}

Lemma \ref{lem:tediousequallity} contains the essential equality for the proof of the explicit expression
of the coefficients of the matrix entries of the weight of Theorem \ref{thm:ortho}. 
The proof of Lemma \ref{lem:tediousequallity} can be found in Appendix \ref{subapp:explicitweight}, and it is based
on a $q$-analogue of the corresponding statement for the classical case \cite[Thm.~5.4]{KoelvPR12}. 
Previously in 2011, Mizan Rahman (private correspondence) has informed one of us that he 
has obtained a $q$-analogue of the underlying summation formula for \cite[Thm.~5.4]{KoelvPR12}. 
It is remarkable that Rahman's $q$-analogue is different from the one needed here in Lemma 
\ref{lem:tediousequallity}.

\begin{proof}[Second part of the proof of Theorem \ref{thm:ortho}]
We prove the statement on the explicit expression of the matrix entries of the weight in 
terms of Chebyshev polynomials. By Corollary \ref{cor:thmbi-B-inv} we have
\begin{equation*}
\begin{split}
W(\psi)_{k,p}(A^{\lambda}) 
&= \sum_{r=0}^{k \wedge p} \alpha_r(k,p) U_{k+p-2r}(\mu(q^\lambda)) \\
&= \sum_{r=0}^{k\wedge p} \sum_{a=0}^{k+p-2r} \alpha_r(k,p)  q^{(k+p-2r-2a)\lambda} \\
&= \sum_{s=-\frac12(k+p)}^{\frac12(k+p)} \left( \underset{2s=k+p-2(r+a)}{\sum_{r=0}^{k\wedge p} \sum_{a=0}^{k+p-2r}} \alpha_r(k,p) \right) q^{2s\lambda}
\end{split}
\end{equation*}
for all $\lambda\in\mathbb{Z}$.
Since the coefficients $\alpha_r(k,p)$ are completely determined by $W(\psi)_{k,p}$ and since
$W(\psi)_{k,p} = W(\psi)_{2\ell-k,2\ell-k}=\bigl(W(\psi)_{p,k}\bigr)^\ast$, we can restrict to the case
$k\leq p$, $k+p\leq 2\ell$. For this case the result follows from 
Lemma \ref{lem:tediousequallity}, and hence the explicit expression for $\alpha_r(k,p)$ 
in Theorem \ref{thm:ortho} is obtained. 
\end{proof}

Note that proof of Theorem \ref{thm:ortho} is not yet complete, since we have to show that the 
weight is strictly positive definite for almost all $x\in [-1,1]$, see Section \ref{sec:genMVOP}. 
This will follow from the LDU-decomposition for the weight as observed in Corollary \ref{cor:thmldu}, but in
order to prove the LDU-decomposition of Theorem \ref{thm:ldu} we need the explicit 
expression for the coefficients $\alpha_t(m,n)$ of Theorem \ref{thm:ortho}. 

In Section \ref{subsec:symmetries} we observed that $J$ commutes with $W(x)$ for all $x$. 
In order to prove Proposition \ref{prop:commutant} we need to show that the commutant is not larger, 
and for this we need the explicit expression of $\alpha_t(m,n)$ of Theorem \ref{thm:ortho}.

\begin{proof}[Proof of Proposition \ref{prop:commutant}]
Let $Y$ be in the commutant, and write 
 $W(x) = \sum_{n = 0}^{2\ell} W_k U_n(x)$ for $W_k \in \Mat_{2\ell+1}(\CC)$ using Theorem \ref{thm:ortho}.
Then $[Y,W_k]=0$ for all $k$. The proof follows closely the proofs of \cite[Prop.~5.5]{KoelvPR12} 
and \cite[Prop.~2.6]{KoeldlRR}.
Note that $W_{2\ell}$ and $W_{2\ell-1}$ are symmetric and persymmetric (i.e. commute with $J$). Moreover $(W_{2\ell})_{m,n}$ is non-zero only for $m+n=2\ell$ and 
$(W_{2\ell-1})_{m,n}$ is non-zero only for $|m+n-2\ell|=1$.
From the explicit expression of the coefficients $\alpha_t(m,n)$ we find that 
all non-zero entries of $W_{2\ell}$ and $W_{2\ell-1}$ are different apart from the symmetry and persymmetry.
The proof of Proposition \ref{prop:commutant} can then be finished following 
\cite[Prop.~5.5]{KoelvPR12}. 
\end{proof}

\subsection{LDU-Decomposition}\label{subsec:ldu}

In order to prove the LDU-decomposition of Theorem \ref{thm:ldu} for the weight we need to prove the matrix identity 
termwise. So we are required to show that 
\begin{equation}\label{eq:forproofLDU}
W(x)_{m,n} = \sum_{k=0}^{m\wedge n} L(x)_{m,k} T(x)_{k,k} L(x)_{n,k}
\end{equation}
for the expression of $W(x)$ in Theorem \ref{thm:ortho} and for the expressions of 
$L(x)$ and $T(x)$ in Theorem \ref{thm:ldu}. Because of symmetry we can assume without loss of 
generality that $m\geq n$. 
Then \eqref{eq:forproofLDU} is equivalent to Proposition \ref{prop:ldu} after taking into
account the coefficients in the LDU-decomposition, so 
it suffices to prove Proposition \ref{prop:ldu} in order to obtain 
Theorem \ref{thm:ldu}.

\begin{proposition} \label{prop:ldu}
For $0 \leq n \leq m \leq 2\ell$, $\ell \in \frac{1}{2} \NN$ and with $\alpha_t(m, n)$ defined in Theorem \ref{thm:ortho},  
we have 
\begin{equation*}
\begin{split}
&\sum_{t = 0}^{n} \alpha_t(m, n) U_{m + n - 2t}(x) \\
& \qquad = \sum_{k = 0}^{n} \beta_k(m, n)
    \frac{w(x;q^{2k+2}|q^2)}{1-x^2}    
    C_{m - k}(x; q^{2k+2} | q^2) C_{n - k}(x; q^{2k + 2} | q^2), \\
&\beta_{k}(m, n)
  = \cfrac{(q^2; q^2)_m}{(q^2; q^2)_{m + k + 1}}
    \cfrac{(q^2; q^2)_{n}}{(q^2; q^2)_{n + k + 1}}
    (q^2; q^2)_{k}^{2} (1 - q^{4k + 2}) \\
    &\qquad \qquad \qquad \times \cfrac{
      (q^2; q^2)_{2\ell + k + 1}
      (q^2; q^2)_{2\ell + k}
    }{
      q^{2\ell}
      (q^2; q^2)_{2\ell}^{2}
    }.
\end{split}
\end{equation*}
\end{proposition}

Before embarking on the proof of Proposition \ref{prop:ldu}, note that each summand on
the right hand side of the expression of Proposition \ref{prop:ldu} is an even, respectively odd,
polynomial for $m+n$ even, respectively odd, since the continuous $q$-ultraspherical polynomials
are symmetric and since 
\begin{equation*} 
w(x;q^{2k+2}|q^2) = 4(1-x^2) \prod_{j=1}^k (1-2(2x^2-1)q^{2j}+q^{4j}),
\end{equation*}
see \eqref{eq:contqultraspherpolsspecialweight}, is an even polynomial with a factor $(1-x^2)$. 
In the proof of Proposition \ref{prop:ldu} we use Lemma \ref{lem:ldu1}. 

\begin{lemma} \label{lem:ldu1}
Let $0\leq k \leq m \leq n$ and $t\in \NN$, then the integral 
\begin{equation*}
\frac{1}{2\pi} \int_{-1}^{1} \frac{w(x;q^{2k+2}|q^2)}{\sqrt{1 - x^2}}
  C_{m-k}(x;q^{2k+2}|q^2) C_{n-k}(x;q^{2k+2}|q^2) U_{n+m-2t}(x) dx \\
\end{equation*}
is equal to zero for $t>m$ and for $0\leq t\leq m$ the integral above is equal to 
$C_k(m,n)R_k(\mu(t); 1, 1, q^{-2m - 2}, q^{-2n - 2}; q^2)$ with 
\begin{gather*}
C_k(m,n)=   \frac{q^{-2\binom{k}{2}}}{1 - q^{2k + 2}}
  \frac{(q^{2k + 2}; q^2)_{m + n - 2k}}{(q^{4k + 4}; q^2)_{m + n - 2k}}
  \frac{(q^{2k + 2}; q^2)_{m - k}}{(q^{2}; q^{2})_{m - k}}
  \frac{(q^{2k + 2}; q^2)_{n - k}}{(q^{2}; q^{2})_{n - k}} 
  \\
  \quad \times 
  \frac{
    (-1)^k (q^{4k + 4}; q^2)_{m + n - 2k} (q^2; q^2)_{k + 1}
  }{
    (q^2; q^2)_{m + n + 1}
  } .
\end{gather*}
\end{lemma}

In Lemma \ref{lem:ldu1} we use the notation \eqref{eqn:q-racah} for the $q$-Racah polynomials. 
Lemma \ref{lem:ldu1} shows that the expansion as in Proposition \ref{prop:ldu} is indeed valid, 
and it remains to determine the coefficients $\beta_k(m,n)$. 

The proof of Lemma \ref{lem:ldu1} follows the lines of the proof of \cite[Lem.~2.7]{KoelvPR13}, 
see Appendix \ref{subapp:forproofLDU}. 
The main ingredients are, cf. the proof in \cite{KoelvPR13}, the connection 
and linearisation coefficients for the continuous $q$-ultraspherical polynomials 
dating back to the work of L.J. Rogers (1894-5), 
see e.g. \cite[\S 10.11]{AndrAR}, \cite[\S 13.3]{Isma}, \cite[(7.6.14), (8.5.1)]{GaspR}. 
Write the product $C_{m-k}(x;q^{2k+2}|q^2) C_{n-k}(x;q^{2k+2}|q^2)$ as a sum over $r$ of continuous 
$q$-ultraspherical polynomials $C_{r}(x;q^{2k+2}|q^2)$ using the linearisation formula 
and write $U_{n+m-2t}$, which is a continuous $q$-ultraspherical
polynomial for $\beta=q$, in terms of $C_{s}(x;q^{2k+2}|q^2)$ using the connection formula.
The orthogonality relations for the continuous $q$-ultraspherical polynomials then 
give the integral in terms of a single series. The details are in Appendix \ref{subapp:forproofLDU}.
From this sketch of proof it is immediately clear that Lemma \ref{lem:ldu1} can be generalised 
to a more general statement. This is the content of Remark \ref{rmk:lemldu1}, whose proof is given in Appendix \ref{app:BproofsBHS}. 

\begin{remark}\label{rmk:lemldu1}\footnote{This result is wrong in the 
published version, and corrected here.}
For integers $0 \leq t$, $k \leq m \leq n$ and parameters $\alpha, \beta$, we have
\begin{align*}
\frac{1}{2\pi} & \int_{-1}^1 \frac{w(x;q^{k + 1}\alpha|q)}{\sqrt{1 - x^2}}
  C_{m-k}(x;q^{k+1}\alpha|q) C_{n-k}(x;q^{k+1}\alpha|q) C_{m + n - 2t}(x;\beta|q) dx \\
  &= C_{k,t}(m, n, \alpha, \beta) \pfq{4}{3}{
  q^{t-m}, \beta q^{-k-1}/\alpha, \beta q^{n-t}, \alpha\beta q^k
  }{
    \alpha\beta q^{n+1}, \beta, \beta q^{-1-m}/\alpha
  }{q}{q},
\end{align*}
%
%
where $C_{k,t}(m,n,\alpha,\beta)$ is
\begin{gather*}
\frac{(\alpha q^{k+1}, \alpha q^{k+1}, \alpha q^{m+n-t+2};q)_\infty}
{(q, \alpha q^{t+1}, \alpha^2q^{2+n+k};q)_\infty}
\frac{(\beta;q)_{n-t}}{(q;q)_{n-t}}
\frac{(\beta;q)_{m-t}}{(q;q)_{m-t}}
\frac{(\alpha\beta q^{1+n}, \alpha q^{t+2}/\beta;q)_{m-t}}{(q;q)_{m-k}}
\end{gather*}
%
\end{remark}

Note that the ${}_4\varphi_3$-series in Remark \ref{rmk:lemldu1} is balanced, but in general is 
not a $q$-Racah polynomial. 

In the proof of Proposition \ref{prop:ldu}, and hence of Theorem \ref{thm:ldu},
we need the summation formula involving $q$-Racah polynomials stated in 
Lemma \ref{lem:ldu2}. Its proof is also given in Appendix \ref{subapp:forproofLDU}. 

\begin{lemma} \label{lem:ldu2}
For $\ell \in \frac{1}{2}\NN$ and $m, n, k \in \NN$ with $0 \leq k \leq n \leq m$ we have
\begin{equation*}
\begin{split}
&\sum_{t = 0}^{m} (-1)^t 
  \frac{(q^{2m - 4\ell}; q^2)_{n - t}}{(q^{2m + 4}; q^2)_{n - t}}
  \frac{(q^{4\ell + 4 - 2t}; q^2)_{t}}{(q^2; q^2)_{t}}
  (1 - q^{2m + 2n + 2 - 4t}) q^{2\binom{t}{2} - 4\ell t} \\
  &\qquad \qquad \qquad \times R_k(\mu(t); 1, 1, q^{-2m-2}, q^{-2n-2}; q^2) \\
  &\quad = q^{n(n-1) - k(k+1) - 4n\ell} (-1)^{n + k} \\
  &\qquad \times \frac{
    (q^2; q^2)_{2\ell + k + 1} (q^2; q^2)_{2\ell - k}
  }{
    (q^2; q^2)_{2\ell + 1}
  }
  \frac{
    (1 - q^{2m + 2})
  }{
    (q^2; q^2)_{n} (q^2; q^2)_{2\ell - n}
  }.
\end{split}
\end{equation*}
\end{lemma}

With these preparations we can prove Proposition \ref{prop:ldu}, and hence the 
LDU-decomposition of Theorem \ref{thm:ldu}. 

\begin{proof}[Proof of Proposition \ref{prop:ldu}]
Since we have the existence of the expression in Proposition \ref{prop:ldu}, it suffices to calculate 
$\beta_k(m,n)$ having the explicit value of the $\alpha_t(m,n)$'s from Theorem \ref{thm:ortho}.
Multiply both sides by $\frac{1}{2\pi} \sqrt{1 - x^2} U_{m + n - 2t}(x)$ and integrate using the 
orthogonality for the Chebyshev polynomials, so that  Lemma \ref{lem:ldu1} gives
\begin{equation*} 
\frac{1}{4} \alpha_t(m, n) 
 = \sum_{k = 0}^n \beta_k(m, n) C_k(m, n) 
    R_k(\mu(t); 1, 1, q^{-2m-2}, q^{-2n-2}; q^2).
\end{equation*}
The $q$-Racah polynomials $R_k(\mu(t); 1, 1, q^{-m-1}, q^{-n-1}; q^2)$ satisfy the 
orthogonality relations 
\begin{gather*}
\sum_{t = 0}^{n} q^{2t} (1 - q^{2m + 2n  + 2 - 4t})
  R_i(\mu(t); 1, 1, q^{-2m - 2}, q^{-2n - 2}; q^2) \\
  \times R_j(\mu(t); 1, 1, q^{-2m - 2}, q^{-2n - 2}; q^2) \\
  = \delta_{i,j} q^{-2k(m + n + 1)} 
    \frac{(1 - q^{2m + 2})(1 - q^{2n + 2}}{(1 - q^{4k + 2})}
    \frac{(q^{2m + 4}, q^{2n + 4}; q^2)_{k}}{(q^{-2m}, q^{-2n}; q^2)_{k}}.
\end{gather*}
Using the orthogonality relations we find the explicit expression of $\beta_k(m,n)$ in terms
of $\alpha_t(m,n)$, and using 
the explicit expression of $\alpha_t(m, n)$ of Theorem \ref{thm:ortho} gives 
\begin{gather*}
\beta_k(m, n) = \frac{1}{4} \frac{q^{2k(m + n + 1)}}{C_k(m, n)}
  \frac{(1 - q^{4k + 2})}{(1 - q^{2m + 2})(1 - q^{2n + 2})}
  \frac{(q^{-2m}, q^{-2n}; q^2)_k}{(q^{2m + 4}, q^{2n + 4}; q^2)_{k}} \\
  \times \sum_{t = 0}^{n} q^{2t} (1 - q^{2m + 2n + 2 - 4t})
    R_k(\mu(t); 1, 1, q^{-2m - 2}, q^{-2n - 2}; q^2) \alpha_t(m, n).
\end{gather*}
This expression is summable by Lemma \ref{lem:ldu2}. Collecting the coefficients proves the 
proposition. 
\end{proof}

\begin{proof}[Last part of the proof of Theorem \ref{thm:ortho}]
Now that we have proved Theorem \ref{thm:ldu}, Corollary \ref{cor:thmldu} is immediate, since the coefficients of 
the diagonal matrix $T(x)$ are positive on $(-1,1)$. So the weight is strictly positive definite on $(-1,1)$,
which is the last step to be taken in the proof of Theorem \ref{thm:ortho}.  
\end{proof}

\subsection{Summation formula for Clebsch-Gordan coefficients}
\label{subsec:ortho-calc}

In this subsection we prove Lemma \ref{lem:ortho-cgc-1}, which has been used in 
the first part of the proof of Theorem \ref{thm:ortho}, see Section \ref{subsec:orthogonality}. 
The proof of Lemma \ref{lem:ortho-cgc-1} is somewhat involved, since we employ 
an indirect way using induction and Theorem \ref{thm:spherical_recurrence}.

\begin{proof}[Proof of Lemma \ref{lem:ortho-cgc-1}]
Assume for the moment that
\begin{equation} \label{eqn:assume-ortho-cgc1}
\sum_{i = -\ell}^{\ell} \left| C^{\ell_1, \ell_2, \ell}_{m, m - i, i} \right|^2 q^{-2i}
  = q^{2\ell_1 - 2\ell - 2m} \frac{(1 - q^{4\ell + 2})}{(1 - q^{4\ell_1 + 2})}.
\end{equation}
Assuming \eqref{eqn:assume-ortho-cgc1} the lemma follows, using $C^{\ell_1, \ell_2, \ell}_{m_1, m_2, i}=0$ if
$m_1-m_2\not=i$, 
\begin{align*}
&\sum_{i = -\ell}^{\ell} 
\sum_{m_1 = -\ell_1}^{\ell_1} 
\sum_{m_2 = -\ell_2}^{\ell_2}
|C^{\ell_1, \ell_2, \ell}_{m_1, m_2, i}|^2 q^{2(m_1 + m_2)} \\
&\quad = 
 \sum_{m_1 = -\ell_1}^{\ell_1} q^{4m_1}
    \sum_{i = -\ell}^{\ell} \left| C^{\ell_1, \ell_2, \ell}_{m_1, m_1-i, i} \right|^2 q^{-2i} \\
  &\quad = 
  \sum_{m_1 = -\ell_1}^{\ell_1} q^{2\ell_1 - 2\ell + 2m_1} \frac{(1 - q^{4\ell + 2})}{(1 - q^{4\ell_1 + 2})}
  = q^{-2\ell} \frac{(1 - q^{4\ell + 2})}{(1 - q^2)}. 
\end{align*}

It remains to prove \eqref{eqn:assume-ortho-cgc1}. We do this in case $\ell_1+\ell_2=\ell$. Put 
$(\ell_1, \ell_2) = (k/2, \ell - k/2)=\xi(0,k)$ for $k\in\mathbb{N}$ and $k\leq 2\ell$. 
Using the explicit expression for the Clebsch-Gordan coefficients \eqref{eq:CGCinbottom}, we find that in this 
special case the left hand side of \eqref{eqn:assume-ortho-cgc1} equals 
\begin{gather*}
\sum \frac{
  (q^2; q^2)_{k} (q^2; q^2)_{2\ell - k} (q^2; q^2)_{\ell - i} (q^2; q^2)_{\ell + i}
}{
  (q^2; q^2)_{k/2 - m} (q^2; q^2)_{k/2 + m} (q^2; q^2)_{\ell - k/2 - m + i} (q^2; q^2)_{\ell - k/2 + m - i} (q^2; q^2)_{2\ell}
} \\[0.3cm]
\times q^{-2i + 2(m - k/2)(m - i + \ell - k/2)},
\end{gather*}
where the sum runs over $i$ for $-\ell + k/2 + m \leq i \leq \ell - k/2 + m$.
Substitute $i \mapsto p - \ell + k/2 + m$ to see that this equals 
\begin{align*}
&\sum_{i = -\ell + k/2 + m}^{\ell - k/2 + m}
\frac{
  (q^2; q^2)_{\ell - i} (q^2; q^2)_{\ell + i}
}{
  (q^2; q^2)_{\ell - k/2 + m - i} (q^2; q^2)_{\ell - k/2 - m + i}
} q^{-2i(1 + m + k/2)} \\
&\quad =
\sum_{p = 0}^{\ell - k} 
\frac{
  (q^2; q^2)_{2\ell - p - k/2 - m} (q^2; q^2)_{p + k/2 + m}
}{
  (q^2; q^2)_p (q^2; q^2)_{2\ell - p - k}
} q^{-2(1 + m + k/2)(p - \ell + k/2 + m)}.
\end{align*}
Simplifying the expression we find that
\begin{gather*}
q^{-2(1 + m + k/2)(-\ell + k/2 + m)}
\frac{
  (q^2; q^2)_{2\ell - k/2 - m} (q^2; q^2)_{k/2 + m}
}{
  (q^2; q^2)_{2\ell - k}
} \\
\times \pfq{2}{1}{q^{-4\ell + 2k}, q^{2m + k + 2}}{q^{-4\ell + k + 2m}}{q^2}{q^{-2k-2}}
\end{gather*}
is equal to
\begin{equation*}
q^{-2(1 + m + k/2)(-\ell + k/2 + m)}
\frac{
  (q^2; q^2)_{2\ell - k/2 - m} (q^2; q^2)_{k/2 + m} (q^{-4\ell - 2};q^2)_{2\ell - k}
}{
  (q^2; q^2)_{2\ell - k} (q^{-4\ell + k + 2m}; q^2)_{2\ell - k}
},
\end{equation*}
since the ${}_2 \varphi_1$ can be summed by the reversed $q$-Chu-Vandermonde sum \cite[(II.7)]{GaspR}.
Putting everything together proves \eqref{eqn:assume-ortho-cgc1}, and hence for the case where $(\ell_1,\ell_2)=\xi(0,k)$ Lemma \ref{lem:ortho-cgc-1} , i.e. $\ell_1+\ell_2=\ell$. 

To prove Lemma \ref{lem:ortho-cgc-1} in general, we set 
\begin{align*}
f^{\ell}_{\ell_1, \ell_2}(\lambda) := \tr(\Phi^{\ell}_{\ell_1, \ell_2}(A^{\lambda}))
  = \sum_{i = -\ell}^{\ell} \sum_{m_1 = -\ell_1}^{\ell_1} \sum_{m_2 = -\ell_2}^{\ell_2}
    \left| C^{\ell_1, \ell_2, \ell}_{m_1, m_2, i} \right|^2 q^{-\lambda(m_1 + m_2)},
\end{align*}
hence it is sufficient to calculate $f^{\ell}_{\ell_1, \ell_2}(-2)$. We will show by 
induction on $n$ that $f^{\ell}_{\xi(n,k)}(-2)$ is independent of $(n,k)$, or 
equivalently that $f^{\ell}_{\ell_1, \ell_2}(-2)$ is independent of $(\ell_1,\ell_2)$.
Since we have established the case $n=0$, Lemma \ref{lem:ortho-cgc-1} then follows.

In order to perform the induction step we consider the recursion of Theorem \ref{thm:spherical_recurrence}.
Using $\varphi(A^\lambda) = \frac12(q^{1+\lambda}+q^{-1-\lambda})$ we see that $\varphi(1)=\varphi(A^{-2})=\frac12(q+q^{-1})$.
Taking the trace of Theorem \ref{thm:spherical_recurrence} at $A^0=1$
and using Proposition \ref{prop:usefull}(ii)
we find, after taking traces,
\begin{equation}\label{eq:sumofAsinspehericalrecurrence}
\frac12(q+q^{-1})(2\ell+1) = (A_{1/2,1/2}+A_{-1/2, -1/2} + A_{-1/2, 1/2} + A_{1/2, -1/2})(2\ell+1).
\end{equation}
Next we evaluate Theorem \ref{thm:spherical_recurrence} at $A^{-2}$ and we take traces, so,
using $\varphi(A^{-2})=\varphi(1)$, 
\begin{equation*} 
\frac12 (q+q^{-1})f^{\ell}_{\ell_1, \ell_2}(-2)  
= \sum_{i, j = \pm 1/2} A_{i, j} f^{\ell}_{\ell_1+i, \ell_2+j}(-2),
\end{equation*}
which we rewrite, assuming $f^{\ell}_{\ell_1, \ell_2}(-2)=F^\ell$ 
is independent of $(\ell_1,\ell_2)$ for $\ell_1+\ell_2\leq \ell +n$, 
so that $f^{\ell}_{\ell_1+\frac12, \ell_2+\frac12}(-2)$ is
\begin{equation*} 
\frac{1}{A_{1/2,1/2}}\left( \frac12(q+q^{-1}) 
- A_{1/2, -1/2} - A_{-1/2, 1/2} - A_{-1/2, -1/2}\right) F^\ell  = F_\ell,
\end{equation*}
by \eqref{eq:sumofAsinspehericalrecurrence} for the last equality. 
So the statement also follows for $\ell_1+\ell_2=\ell+n+1$, and the lemma follows.
\end{proof}

\section{\texorpdfstring{$q$}{q}-Difference operators for the matrix-valued polynomials}\label{sec:qdifferenceoperators}

We continue the study of $q$-difference operator for the matrix-valued orthogonal 
polynomials started in  Corollary \ref{cor:diffeqPn}. In particular, we show that the 
assumption on the invertibility of $\hat{\Phi}_0$ follows from the 
LDU-decomposition in Theorem \ref{thm:ldu}. Then we make the coefficients in 
the matrix-valued second order $q$-difference operator of Corollary 
\ref{cor:diffeqPn} explicit in Section \ref{subsec:poly_qdiff}. Comparing to the scalar-valued 
Askey-Wilson $q$-difference operators, see e.g. \cite{AndrAR}, \cite{GaspR}, \cite{Isma},
\cite{KoekLS}, we view the $q$-difference operator as a 
matrix-valued analogue of the Askey-Wilson $q$-difference operator. 
Next, having the matrix-valued orthogonal polynomials as eigenfunctions to 
the matrix-valued Askey-Wilson $q$-difference operator, we use this to obtain
an explicit expression of the matrix entries of the matrix-valued orthogonal 
polynomials in terms of scalar-valued orthogonal polynomials from the $q$-Askey scheme
by decoupling of the $q$-difference operator using the matrix-valued polynomial
$L$ in the LDU-decomposition of Theorem \ref{thm:ldu}. From this expression we can obtain an explicit
expression for the coefficients in the matrix-valued three-term recurrence
relation for the matrix-valued orthogonal polynomials, hence proving
Theorem \ref{thm:monic-3-term} and Corollary \ref{cor:monic-3-term}. 

\subsection{Explicit expressions of the $q$-difference operators}\label{subsec:poly_qdiff}

In order to make make Corollary \ref{cor:diffeqPn} explicit, we need to study 
the invertibility of the matrix $\hat{\Phi}^\ell_0(A^\lambda)$. 
For this we first use Theorem \ref{thm:bi-B-inv} and its Corollary \ref{cor:thmbi-B-inv}, and in particular
\eqref{eq:weightpairedAlambda}. Because of Proposition \ref{prop:usefull}(ii), this is only 
a single sum. In the notation of \eqref{eq:deffullsphericalfunction}, we find 
\begin{equation*}
\begin{split}
W(\psi)_{k,p}(A^\lambda) 
  &= \sum_{n=-\ell}^\ell (\hat{\Phi}^\ell_0)_{n,k}(A^{\lambda-1}) \overline{(\hat{\Phi}^\ell_0)_{n,p}(A^{-\lambda-1})} \\
  &= \left( \Bigl( \hat{\Phi}^\ell_0(A^{-\lambda-1})\Bigr)^\ast  \, \hat{\Phi}^\ell_0(A^{\lambda-1})\right)_{p,k}.
\end{split}
\end{equation*}
Taking the determinant gives, recalling $\mu(z)= \frac12(z+z^{-1})$, 
\begin{equation*}
\begin{split}
\overline{\det\left(\hat{\Phi}^\ell_0(A^{-\lambda-1}) \right)} &\det\left(\hat{\Phi}^\ell_0(A^{\lambda-1}) \right) 
= \det \bigl(W(\mu(q^\lambda))\bigr) = \det \bigl(T(\mu(q^\lambda))\bigr) 
\\ & = \prod_{k=0}^{2\ell} T(\mu(q^\lambda))_{k,k} =
\prod_{k=0}^{2\ell} 4 c_k(\ell) (q^{2+2\lambda}, q^{2-2\lambda};q^2)_k 
\end{split}
\end{equation*}
using the LDU-decomposition for the weight of Theorem \ref{thm:ldu} and \eqref{eq:contqultraspherpolsspecialweight}.
The right hand side is non-zero for $\lambda \in \mathbb{Z}$ unless $1\leq |\lambda|\leq 2\ell$. 
So $\hat{\Phi}^\ell_0(A^\lambda)$ is invertible for $\lambda\geq  2\ell$ or $\lambda < -2\ell-1$ or $\lambda=-1$, i.e.
for an infinite number of $\lambda\in \mathbb{Z}$ and it is meaningful to consider Corollary \ref{cor:diffeqPn}. 

\begin{proof}[Proof of Theorem \ref{thm:diff_eqn_Pn}]
Corollary \ref{cor:diffeqPn} gives a second-order $q$-difference equation for the matrix-valued orthogonal polynomials for
an infinite set of $\lambda$. So it suffices to check that for $i=1,2$ the matrix-valued functions 
$\tilde{M}_i$, $\tilde{N}_i$ of Corollary \ref{cor:diffeqPn} coincide with $\mathcal{M}_i$, $\mathcal{N}_i$,
where $\mathcal{N}_i(z)=\mathcal{M}_i(z^{-1})$, of 
Theorem \ref{thm:diff_eqn_Pn}, or 
\begin{equation}\label{eq:diffeqnidentities}
\begin{split}
\overline{\hat\Phi^\ell_0(A^{\lambda-1})} 
\mathcal{M}_i(q^{\lambda}) &= \overline{M_i(q^{\lambda-1}) \hat\Phi^\ell_0(A^{\lambda})}, \\
\overline{\hat\Phi^\ell_0(A^{\lambda-1})} \mathcal{N}_i(q^{\lambda}) &= \overline{N_i(q^{\lambda-1}) \hat\Phi^\ell_0(A^{\lambda-2})}, 
\end{split}
\end{equation}
where $M_i$, $N_i$ as in \eqref{eq:q_differenceeqnforPhin}, see  
Proposition \ref{prop:diffeqforhatPhi}, 
and $\mathcal{M}_i$, $\mathcal{N}_i$ as in Theorem \ref{thm:diff_eqn_Pn},
where $\mathcal{N}_i(z)=\mathcal{M}_i(z^{-1})$. 

By \eqref{eq:Lambda0expressedinMN} we need
\begin{equation}\label{eq:Lambda0expressedincurlyMN}
\Lambda_0(i)= \mathcal{M}_i(z) + \mathcal{N}_i(z) = \mathcal{M}_i(z) + \mathcal{M}_i(z^{-1}),
\end{equation}
which is an easy check using the explicit expressions of Theorem \ref{thm:diff_eqn_Pn}. 
Now \eqref{eq:Lambda0expressedincurlyMN} and \eqref{eq:q_differenceeqnforPhin} for $n=0$ 
show that any equation of 
\eqref{eq:diffeqnidentities} implies the other equation of \eqref{eq:diffeqnidentities}.
Indeed, assuming the second equation of \eqref{eq:diffeqnidentities} holds, 
then
\begin{gather*}
\overline{\hat\Phi^\ell_0(A^\lambda)}\mathcal{M}_i(q^{\lambda+1})
+ \overline{\hat\Phi^\ell_0(A^\lambda)}\mathcal{N}_i(q^{\lambda+1}) 
= \overline{\hat\Phi^\ell_0(A^\lambda)} \Lambda_0(i) \\
= \overline{M_i(q^\lambda) \hat\Phi^\ell_0(A^{\lambda+1})} + \overline{N_i(q^\lambda) \hat\Phi^\ell_0(A^{\lambda-1})} \\
= \overline{M_i(q^\lambda) \hat\Phi^\ell_0(A^{\lambda+1})} + \overline{\hat\Phi^\ell_0(A^\lambda)}\mathcal{N}_i(q^{\lambda+1})
\end{gather*}
implying the first equation of \eqref{eq:diffeqnidentities}. 

By Proposition \ref{prop:usefull}(ii) and \eqref{eq:deffullsphericalfunction} the matrix entries of $\hat\Phi^\ell_0(A^{\lambda})$
are Laurent series in $q^\lambda$. Setting $z=q^\lambda$, we see that in order to verify \eqref{eq:diffeqnidentities} entry-wise,
we need to check equalities for Laurent series in $z$.  

We first consider the second equality of \eqref{eq:diffeqnidentities} for $i=1$.
In this case the matrices $N_1$ and $\mathcal{N}_1$ 
are band-limited. 
Hence, the $(m,n)$-th entry of both sides of \eqref{eq:diffeqnidentities} involves either two or one terms, 
so we need to check 
\begin{equation}\label{eq:diffeqN1}
\begin{split}
&\Phi^\ell_{\xi(0,n-1)}(A^{\lambda-1})_{m-\ell,m-\ell}\Big\vert_{z=q^\lambda} \mathcal{N}_1(z)_{n-1,n} \\
&\qquad + \Phi^\ell_{\xi(0,n)}(A^{\lambda-1})_{m-\ell,m-\ell}\Big\vert_{z=q^\lambda} \mathcal{N}_1(z)_{n,n} \\ 
&\qquad\qquad  = N_1(\frac{z}{q})_{m,m}\Phi^\ell_{\xi(0,n)}(A^{\lambda-2})_{m-\ell,m-\ell}\Big\vert_{z=q^\lambda}.
\end{split}
\end{equation}
The proof of \eqref{eq:diffeqN1} involves the explicit expression of the spherical
functions in terms of Clebsch-Gordan coefficients using Proposition \ref{prop:usefull}(ii). It is given 
in Appendix \ref{subapp:forproofqdiffeq}. 

The statements for the second $q$-difference equation with $i=2$ follows from the symmetries of Proposition 
\ref{prop:diffeqforhatPhi} and Lemma \ref{lem:symhatPhi}. 
\end{proof}

The explicit expressions have been obtained initially by computer algebra, and then later the proof as
presented here and in Appendix \ref{subapp:forproofqdiffeq} has been obtained. 

\subsection{Explicit expressions for the matrix entries of the matrix-valued orthogonal polynomials}\label{subsec:explicitexpressions}

Having established the $q$-difference equations for the matrix-valued orthogonal polynomials
of Theorem \ref{thm:diff_eqn_Pn} and having the diagonal part of the LDU-decomposition of the weight in 
terms of weight functions for the continuous $q$-ultraspherical 
polynomials in Theorem  \ref{thm:ldu}, it is natural to look at the $q$-difference operators conjugated by 
the polynomial function $L^t$. It turns out that this completely decouples one of the
second order $q$-difference operators of Theorem \ref{thm:diff_eqn_Pn}.
This gives the opportunity to link the matrix entries of the matrix-valued 
orthogonal polynomials to continuous $q$-ultraspherical polynomials. 
In order to determine the coefficients we use the other $q$-difference operator and the 
orthogonality relations. 
Having such an explicit expression we can determine the three-term recurrence relation
for the monic matrix-valued orthogonal polynomials straightforwardly, and hence also
for the matrix-valued orthogonal polynomials $P_n$, since we already have determined
the leading coefficient in Corollary \ref{cor:ortho-cgc-2}.

The first step is to conjugate the second order 
$q$-difference operator $D_1$ of Theorem \ref{thm:diff_eqn_Pn} with the matrix $L^t$ 
of the LDU-decomposition of Theorem \ref{thm:ldu} into
a diagonal $q$-difference operator. This conjugation is inspired by the result of \cite[Theorem 6.1]{KoelvPR13}. 
This conjugation takes $D_2$ in a three-diagonal $q$-difference operator.
For any $n \in \mathbb{N}$, let $\mathcal{R}_n(x) = L^t(x) Q_n(x)$, where $Q_n(x) = P_n(x)\bigl(\text{lc}(P_n)\bigr)^{-1}$ 
denote the corresponding monic polynomial. Note that we have determined the leading coefficient $\text{lc}(P_n)$ in 
Corollary \ref{cor:ortho-cgc-2}.
Then $(\mathcal{R}_n)_{n\geq0}$ forms a family of matrix-valued polynomials, but note that the 
degree of $\mathcal{R}_n$ is larger than $n$, and that the  
leading coefficient of $\mathcal{R}_n$ is singular. 
Note that $\mathcal{R}_n$ satisfy the orthogonality relations 
\begin{equation}\label{eq:orthorelmathcalRn}
\begin{split}
\int_{-1}^1 \bigl(\mathcal{R}_n(x)\bigr)^\ast T(x) \mathcal{R}_m(x)\, \sqrt{1-x^2}dx 
  &=  \int_{-1}^1 \bigl(Q_n(x)\bigr)^\ast W(x) Q_m(x)\,  \sqrt{1-x^2}dx \\
  &= \delta_{m,n}\frac{\pi}{2} \bigl(\text{lc}(P_m)^\ast\bigr)^{-1} G_m \bigl(\text{lc}(P_m)\bigr)^{-1}.
\end{split}
\end{equation}

\begin{theorem}\label{thm:decoupling_diff_operator}
The polynomials $(\mathcal{R}_n)_{n \geq 0}$ are eigenfunctions of the $q$-difference operators
\begin{align*}
\mathcal{D}_i = \mathcal{K}_i(z) \eta_{q} + \mathcal{K}_i(z^{-1}) \eta_{q^{-1}},
\end{align*}
with eigenvalues $\Lambda_n(i)$, where
\begin{align*}
\mathcal{K}_1(z)
  &= \sum_{i=0}^{2\ell} 
    \frac{q^{1-i}}{(1-q^2)^2}
    \frac{(1 - q^{2i+2} z^2)}{(1 - z^2)} E_{i,i},  \\
\mathcal{K}_2(z) 
  &= -\sum_{i=1}^{2\ell}  
    q^{i - 2\ell + 1}
    \frac{
      (1 - q^{4\ell - 2i + 2})
    }{
      (1 - q^2)^2
    }
    \frac{z}{(1 - z^2)} E_{i, i-1} \\
  &\quad + \sum_{i=0}^{2\ell} 
    2 q^{i - 2\ell + 1}
    \frac{
      1
    }{
      (1 - q^2)^2
    }
    \frac{
      (1 + q^{4\ell + 2})
    }{
      (1 + q^{2i})(1 + q^{2i + 2})
    }
  \frac{
    (1 - q^{2i + 2} z^2)
  }{
    (1 - z^2)
  } E_{i,i} \\
  &\quad - \sum_{i=0}^{2\ell-1}  
    q^{i - 2\ell + 1}
    \frac{
      1
    }{
      (1 - q^2)^2
    } 
    \frac{
      (1 - q^{4\ell + 2i + 4})(1 - q^{2i + 2})^2
    }{
      (1 - q^{4i+6})(1 - q^{4i+2})(1 + q^{2i + 2})^2
    } \\
    &\qquad \qquad \times \frac{
      (1 - q^{2i+2}z^2)(1 - q^{2i+4}z^2)
    }{
      z (1 - z^2)
    } E_{i,i+1}. \\
\end{align*}
\end{theorem}

\begin{proof} We start by observing that the monic matrix-valued orthogonal polynomials $Q_n$ are 
eigenfunctions of the second order $q$-difference operators $D_i$ of Theorem \ref{thm:diff_eqn_Pn} for 
the eigenvalue $\text{lc}(P_n)\Lambda_n(i) \text{lc}(P_n)^{-1} =\Lambda_n(i)$, since the matrices are
diagonal and thus commute. By conjugation we find that $\mathcal{R}_n$ satisfy
\begin{equation*}
\mathcal{K}_i(z) \breve{\mathcal{R}}_n(qz) + \mathcal{K}_i(z^{-1})\breve{\mathcal{R}}_n(\frac{z}{q}) = 
\mathcal{R}_n(x) \Lambda_n(i),\,  
\mathcal{K}_i(z) = \breve{L}^t(z) \mathcal{M}_i(z) \bigl(\breve{L}^t(qz)\bigr)^{-1}
\end{equation*}
using the notation $\breve L^t(z)=L^t(\mu(z))$, etc., with $x=\mu(z) = \frac12(z+z^{-1})$ as before. 
It remains to calculate $\mathcal{K}_i(z)$ explicitly. 
We show in Appendix \ref{subapp:matricesconjugatedqdiffop} that the expressions for $\mathcal{K}_i$ are
correct by verifying
\begin{equation}\label{eq:matricesconjugatedqdiffop}
\mathcal{K}_i(z) \breve{L}^t(qz) = \breve{L}^t(z) \mathcal{M}_i(z),
\end{equation}
for $i=1,2$.
\end{proof}

\begin{lemma} \label{lem:expRn1}
For $n \in \NN$ and $0 \leq i,j \leq 2\ell$ we have
\begin{align*}
\mathcal{R}_n(x)_{i j} = \beta_n(i, j)\, C_{n + j - i}(x; q^{2i + 2} | q^2),
\end{align*}
where $C_{n}(x; \beta | q)$ are the continuous $q$-ultraspherical polynomials 
\eqref{eqn:cont_q_ultra_poly} and $\beta_n(i, j)$ is 
a constant depending on $i, j$ and $n$.
\end{lemma}

\begin{proof}
Evaluate $\mathcal{D}_1 \mathcal{R}_n(x) = \mathcal{R}_n(x) \Lambda_n(1)$ in entry $(i, j)$. 
Since $\mathcal{D}_1$ is decoupled, we get a $q$-difference equation for 
the polynomial $\bigl(\mathcal{R}_n\bigr)_{i,j}$, which, after
simplifying, is
\begin{equation*} 
\begin{split}
&\frac{(1 - q^{2i + 2} z^2)}{(1 - z^2)} \breve{\mathcal{R}}_n(qz)_{i j}
  + \frac{(1 - q^{2i + 2} z^{-2})}{(1 - z^{-2})} \breve{\mathcal{R}}_n(q^{-1}z)_{i j} \\
  &\qquad = q^{1+i} (q^{-j - n - 1} + q^{j + n + 1}) \breve{\mathcal{R}}_n(z)_{i j}.
\end{split}
\end{equation*}
All polynomial solutions of this $q$-difference are given by a multiple of the Askey-Wilson polynomials, $p_{n + j - i}(x; q^{i + 1}, -q^{i + 1}, q^{1/2}, -q^{1/2} | q)$, see \cite[\S 7.5]{GaspR}, \cite[\S 16.3]{Isma}, \cite[\S 14.1]{KoekLS}.
Apply the quadratic transformation, see \cite[(4.20)]{AskeW}, to see that the polynomial solutions are 
$p_{n + j - i}(x; q^{i + 1}, -q^{i + 1}, q^{i + 2}, -q^{i + 2} | q^2)$.
These polynomials are multiples of continuous $q$-ultraspherical polynomials $C_{n + j - i}(x; q^{2i + 2} | q^2)$,
\cite[\S 13.2]{Isma}, \cite[\S 7.4-5]{GaspR}, \cite[\S 14.10.1]{KoekLS}. 
Hence, the polynomial matrix-entries $\mathcal{R}_n(x)_{ij}$ are a multiple of $C_{n + j - i}(x; q^{2i + 2} | q^2)$.
\end{proof}

Our next objective is to determine the coefficients $\beta_n(i,j)$ of 
Lemma \ref{lem:expRn1}.
Having exploited that the matrix-valued polynomials $\mathcal{R}_n$ are eigenfunctions 
for the decoupled operator $\mathcal{D}_1$ of Theorem \ref{thm:decoupling_diff_operator},
we can use Lemma \ref{lem:expRn1} in  \eqref{eq:orthorelmathcalRn} to calculate 
the $(i,j)$-th coefficient of \eqref{eq:orthorelmathcalRn};
\begin{equation}\label{eq:orthcalRn-1}
\begin{split}
&\frac{2}{\pi} \sum_{k=0}^{2\ell} \overline{\beta_n(k,i)} \beta_m(k,j) c_k(\ell) \\
&\qquad \times \int_{-1}^1 C_{n+i-k}(x;q^{2k+2}|q^2)C_{m+j-k}(x;q^{2k+2}|q^2) \frac{w(x;q^{2k+2}|q^2)}{\sqrt{1-x^2}}\, dx \\
&\quad =\delta_{m,n} \delta_{i,j} (G_m)_{i,i} (\text{lc}(P_m))_{i,i}^{-2},
\end{split}
\end{equation}
using that $\text{lc}(P_m)$ and the squared norm matrix $G_m$ are diagonal matrices, see
Corollary \ref{cor:ortho-cgc-2} and Theorem \ref{thm:ortho}. The integral in 
\eqref{eq:orthcalRn-1} can be evaluated by 
\eqref{eq:ortho-contqultraspherpols}. In particular, the case $m+j=n+i$ of \eqref{eq:orthcalRn-1} 
gives the explicit orthogonality relations 
\begin{equation}\label{eq:orthcalRn-2}
\begin{split}
&\sum_{k=0}^{2\ell} \overline{\beta_{m+j-i}(k,i)} \beta_m(k,j)
  c_k(\ell) \\ 
  &\qquad \times \frac{
    (q^{2k+4}; q^2)_k
  }{
    (q^2; q^2)_{k}
  } 
  \frac{
    (q^{4k+4};q^2)_{m+j-k}
  }{
    (q^2;q^2)_{m+j-k}
  }
  \frac{
    (1 - q^{2k+2})
  }{
    (1 - q^{2m+2j+2})
  } \\
  &= \delta_{i,j} 4^{1-m} q^{-2\ell}
  \frac{
    (1 - q^{4\ell+2})^2
  }{
    (1 - q^{2m+2i+2})(1 - q^{4\ell-2i+2m+2})
  } \\
  &\qquad \times \frac{
    (q^2, q^{4\ell+4}; q^2)_{m}^2
  }{
    (q^{2i+2}, q^{4\ell-2i+2}; q^2)_{m}^2
  }.
\end{split}
\end{equation}

\begin{theorem} \label{thm:explicit_Rn}
We have
\begin{align*}
\mathcal{R}_n(x)_{i, j} 
  &= (-1)^i\, 2^{-n}
  \frac{
    (q^2, q^{4\ell + 4}; q^2)_n
  }{
    (q^{2j + 2}, q^{4\ell - 2j + 2}; q^2)_n
  }
  \frac{
    (q^{-4\ell}, q^{-2j - 2n}; q^2)_i
  }{
    (q^2, q^{4\ell + 4}; q^2)_i
  } \\
  & \qquad \times \frac{
    (q^2; q^2)_{n + j - i}
  }{
    (q^{4i + 4}; q^2)_{n + j - i}
  }
  q^{j(2i + 1) + 2i(2\ell + n + 1) - i^2} \\
  & \qquad \times
  R_i(\mu(j); 1, 1, q^{-2n - 2j - 2}, q^{-4\ell - 2}; q^2)
  C_{n + j - i}(x; q^{2i + 2}| q^2).
\end{align*}
\end{theorem}

\begin{proof}
From Theorem \ref{thm:decoupling_diff_operator} we have 
\begin{equation} \label{eqn:Rnqdiff2}
\mathcal{K}_2(z) \breve{\mathcal{R}}_n(qz) + \mathcal{K}_2(z^{-1}) \breve{\mathcal{R}}_n(q^{-1}z) = \breve{\mathcal{R}}_n(z) \Lambda_n(2).
\end{equation}
Evaluate \eqref{eqn:Rnqdiff2} in entry $(i, j)$ and use Lemma \ref{lem:expRn1} to find a three term recurrence relation in $i$ of $\beta_n(i, j)$,
\begin{align*}
\frac{(q^{-j - n - 1} + q^{j + n + 1})}{(q^{-1} - q)^2} \beta_n(i, j) 
  &\breve{C}_{n + j - i}(z; q^{2i + 2}| q^2) \\
= \beta_n(i + 1, j) 
  &\Bigl( \mathcal{K}_2(z)_{i, i+1} \breve{C}_{n + j - i - 1}(qz; q^{2i + 4} | q^2) \\
  &\qquad + \mathcal{K}_2(z^{-1})_{i, i+1} \breve{C}_{n + j - i - 1}(q^{-1}z; q^{2i + 4} | q^2) \Bigr) \\
+ \beta_n(i, j) 
  &\Bigl( \mathcal{K}_2(z)_{i i} \breve{C}_{n + j - i - 1}(qz; q^{2i + 2} | q^2) \\
  & \qquad + \mathcal{K}_2(z^{-1})_{i i} \breve{C}_{n + j - i}(q^{-1}z; q^{2i + 2} | q^2)\Bigr) \\
+ \beta_n(i - 1, j) 
  &\Bigl(\mathcal{K}_2(z)_{i, i-1} \breve{C}_{n + j - i + 1}(qz; q^{2i} | q^2) \\
  & \qquad + \mathcal{K}_2(z^{-1})_{i, i-1} \breve{C}_{n + j - i + 1}(q^{-1}z; q^{2i} | q^2) \Bigr)
\end{align*}
Multiply by $(1 - z^2)(1 - q^2)^2$ and evaluate the Laurent expansion at the leading coefficient in $z$ of degree $n + j - i + 3$.
The leading coefficient in $z$ of the continuous $q$-ultraspherical polynomial $\breve{C}_n(\alpha z; \beta| q)$ is
$\frac{(\beta; q)_n}{(q; q)_n} \alpha^n$. After a straightforward computation this leads to the three term recurrence relation
\begin{gather*} 
(1 + q^{4\ell - 2j + 2n + 2})\beta_n(i, j) \\
= q^{2i + 3} \frac{
    (1 - q^{4\ell + 2i + 4})(1 - q^{2i + 2})^2(1 - q^{2n + 2i + 2j + 4})
  }{
    (1 - q^{n + i + j + 2})(1 - q^{4i + 6})(1 - q^{4i + 2})(1 - q^{2i + 2})^2
  } \beta_n(i + 1, j) \\
 - 2 q^{2i + 2} \frac{(1 + q^{4\ell + 2})(1 - q^{2n + 2j + 2})}{(1 + q^{2i})(1 + q^{2i + 2})} \beta_n(i, j) \\
  + q^{2i + 1} (1 - q^{n + i + j + 1})(1 - q^{4\ell - 2i + 2})(1 - q^{2n + 2j - 2i + 2}) \beta_n(i -1, j). 
\end{gather*}
This recursion relation can be rewritten as the three-term recurrence relation for the $q$-Racah polynomials after rescaling, 
see \cite[Ch.~7]{GaspR}, \cite[\S 15.6]{Isma} \cite[\S 14.2]{KoekLS}. We identify  $(\alpha,\beta,\gamma,\delta)$ 
as in \eqref{eqn:q-racah} with  
$(1, 1, q^{-2n - 2j - 2}, q^{-4\ell - 2})$ in base $q^2$. This gives
\begin{gather*}
\beta_n(i, j) = \gamma_n(j) (-1)^i
  \frac{(q^{-4\ell}, q^{-2j - 2n}; q^2)_i}{(q^2, q^{4\ell + 4}; q^2)_i}
  \frac{(q^2; q^2)_{n + j - i}}{(q^{4i + 4}; q^2)_{n + j - i}}
  q^{2ji + 2i(2\ell + n + 1) - i^2} \\
\qquad\times R_i(\mu(j); 1, 1, q^{-2n - 2j - 2}, q^{-4\ell - 2}; q^2)
\end{gather*}
for some constant $\gamma_n(j)$ independent of $i$. 
Plugging this expression in \eqref{eq:orthcalRn-2} for $i=j$ gives $|\gamma_n(j)|^2$ by comparing with the 
explicit orthogonality relations for the $q$-Racah polynomials, see \cite[Ch.~7]{GaspR}, \cite[\S 15.6]{Isma} \cite[\S 14.2]{KoekLS}.

For $j\geq i$ we have $\mathcal{R}_{i,j}(x) = L_{j,i}(x) (x^n \text{Id} + \text{l.o.t})$, and since the explicit expression 
of $L_{j,i}$ shows that the leading coefficient (of degree $j-i$) is positive, we see that the 
leading coefficient (of degree $n+j-i$) of $\mathcal{R}_{i,j}$ in case $j\geq i$ is positive.
Since $\gamma_n(j)$ is independent of $i$, we take $i=0$, which shows that $\gamma_n(j)$ is positive.
\end{proof}

\begin{proof}[Proof of Theorem \ref{thm:explicit_Pn}]
Using Theorem \ref{thm:explicit_Rn} with the explicit inverse of $L(x)$ as given in Theorem \ref{thm:ldu} gives
an explicit expression for the matrix entries of $Q_n(x) = (L(x)^{-1})^t \mathcal{R}_n(x)$. Then we obtain 
the matrix entries of 
$P_n(x) = Q_n(x) \text{lc}(P_n)$ from this expression and Corollary \ref{cor:ortho-cgc-2}, stating that the 
leading coefficient is a diagonal matrix. 
\end{proof}

\subsection{Three-term recursion relation} 

The matrix-valued orthogonal polynomials satisfy a three-term recurrence relation, see Section \ref{sec:genMVOP}.
Moreover, Theorem \ref{thm:spherical_recurrence} shows that the three-term recurrence relation can in principle be obtained
from the tensor-product decomposition. However, in that case we obtain the coefficients of the matrices in the 
three-term recurrence relation in terms of sums of squares of Clebsch-Gordan coefficients, and this leads to 
a cumbersome result. 
In order to obtain an explicit expression for the three-term recurrence relation as in 
Theorem \ref{thm:monic-3-term} and Corollary \ref{cor:monic-3-term} we use the 
explicit expression obtained in Theorem \ref{thm:explicit_Pn} and Lemma \ref{lem:3-term-explicit}, which is
\cite[Lemma 5.1]{KoelvPR13}. Lemma \ref{lem:3-term-explicit} is only used to determine $X_n$. 

\begin{lemma} \label{lem:3-term-explicit}
Let $(Q_n)_{n \geq 0}$ be a sequence of monic (matrix-valued) orthogonal polynomials and write 
$Q_n(x) = \sum_{k = 0}^n Q^{n}_{k} x^k$, where $Q^{n}_{k} \in \Mat_N(\CC)$.
The sequence $(Q_n)_{n \geq 0}$ satisfies the three-term recurrence relation
\begin{equation*}
xQ_n(x) = Q_{n+1}(x) + Q_n(x) X_n + Q_{n-1}(x) Y_n,
\end{equation*}
where $Y_{-1} = 0$, $Q_0(x) = I$ and
\begin{equation*}
X_n = Q^n_{n-1} - Q^{n+1}_{n}, 
  \quad Y_n = Q^{n}_{n-2} - Q^{n+1}_{n-1} - Q^n_{n-1} X_n.
\end{equation*}
\end{lemma}

So we start by calculating the one-but-leading term in the monic matrix-valued orthogonal
polynomials. 

\begin{lemma} \label{lem:Qn-1-explicit}
For the monic matrix-valued orthogonal polynomials $(Q_n)_{n \geq 0}$ we have
\begin{align*}
Q^n_{n-1} &=
  - \sum_{j = 0}^{2\ell - 1} \frac{q}{2}
  \frac{
    (1 - q^{2n})(1 - q^{2j + 2})
  }{
    (1 - q^2)(1 - q^{2n + 2j + 2})
  } E_{j, j+1} \\
  &\quad - \sum_{j = 1}^{2 \ell}
  \frac{q}{2} \frac{
    (1 - q^{2n})(1 - q^{4\ell - 2j + 2})
  }{
    (1 - q^2)(1 - q^{4\ell - 2j + 2n + 2})
  } E_{j, j-1}.
\end{align*}
\end{lemma}

\begin{proof}
By Theorem \ref{thm:explicit_Rn} we have $(\mathcal{R}_n(x))_{i,j} = \beta_n(i, j) C_{n + j - i}(x; q^{2i + 2} | q^2)$ and $Q_n(x) = L^{t}(x)^{-1}\mathcal{R}_n(x)$, see Section \ref{subsec:explicitexpressions},  we have
\begin{equation}\label{eq:Qnn-1calculate}
\begin{split}
Q_n(x)_{i,j} &= \sum_{k = i}^{2\ell} q^{(2k + 1)(k - i)}
  \frac{
    (q^2; q^2)_k (q^2; q^2)_{k + i}
  }{
    (q^2; q^2)_{2k} (q^2; q^2)_{i}
  } \beta_n(i, j) \\
&\quad \times C_{k - i}(x; q^{-2k} | q^2) C_{n + j - k}(x; q^{2k + 2} | q^2),
\end{split}
\end{equation}
and this expression shows that $\deg(Q_n(x)_{i, j}) = n + j - i$. 
So in case $j-i\leq 0$, we can only have a contribution to $Q^n_{n-1}$ in case $i-j=0$ or $i-j=1$.
The first case does not give a contribution, since \eqref{eq:Qnn-1calculate} is even or odd for 
$n+j-i$ even or odd. So we only have to calculate the leading coefficient in 
\eqref{eq:Qnn-1calculate} for $i-j=1$. With the explicit value of $\beta_n(i,j)$ as in 
Section \ref{subsec:explicitexpressions} or Theorem \ref{thm:explicit_Pn} and 
Corollary \ref{cor:normalisation} we see that $Q^n_{n-1}$ in case $i-j=1$ gives
the required expression for $\bigl(Q^n_{n-1}\bigr)_{j,j-1}$.

On the other hand, by Proposition \ref{prop:commutant} and since $J\text{lc}(P_n)J=\text{lc}(P_n)$ by 
Corollary \ref{cor:ortho-cgc-2}, it follows that $JQ_n(x)J=Q_n(x)$. 
Therefore we find the symmetry of the entries of the monic matrix-valued polynomials $\bigl(Q_n(x)\bigr)_{i,j} = \bigl(Q_n(x)\bigr)_{2\ell-i,2\ell-j}$, so that the case 
$j-i\geq 0$ can be reduced to the previous case, and we get 
$\bigl(Q^n_{n-1}\bigr)_{j,j+1}=\bigl(Q^n_{n-1}\bigr)_{2\ell-j,2\ell-j-1}$.
\end{proof}

\begin{proof}[Proof of Theorem \ref{thm:monic-3-term}]
The explicit expression for $X_n$ follows from Lemma \ref{lem:3-term-explicit} and Lemma \ref{lem:Qn-1-explicit}.

By Theorem \ref{thm:ortho} and Corollary \ref{cor:ortho-cgc-2} we have the orthogonality relations
\begin{equation*} 
\begin{split}
\frac{2}{\pi}\int_{-1}^1 \bigl( Q_m(x)\bigr)^\ast W(x) Q_n(x) \sqrt{1 - x^2} dx 
&= \delta_{m, n} \bigl( \text{lc}(P_m)^\ast\bigr)^{-1} G_m \bigl( \text{lc}(P_m)\bigr)^{-1} \\
&= \delta_{m, n} G_m \bigl( \text{lc}(P_m)\bigr)^{-2}, 
\end{split}
\end{equation*}
since the matrices involved are diagonal and self-adjoint, hence pairwise commute. 
By the discussion in Section \ref{sec:genMVOP}, we have
\begin{equation*}
Y_n = G_{n-1}^{-1} \bigl( \text{lc}(P_{n-1})\bigr)^{2} G_n \bigl( \text{lc}(P_n)\bigr)^{-2},
\end{equation*}
and a straightforward calculation gives the required expression. 
\end{proof}

\subsection*{Acknowledgements}
We thank Stefan Kolb for useful discussions and his hospitality during the visit of 
the first author to Newcastle. 
We also thank the late Mizan Rahman for his help in 
proving the expression for the explicit coefficients in Theorem \ref{thm:ortho}, even though the 
final proof is different. 
The research of Noud Aldenhoven is
supported by the Netherlands Organisation for Scientific Research (NWO) under project number
613.001.005 and by Belgian Interuniversity Attraction Pole Dygest P07/18.
The research of Pablo Rom\'an is supported by the Radboud Excellence Fellowship
and by CONICET grant PIP 112-200801-01533 and by SeCyT-UNC. Pablo Rom\'an was also supported by the Coimbra Group Scholarships Programme at KULeuven in the period February-May 2014.

We would like to thank the anonymous referee for his comments and remarks, that have helped us to improve the paper. We thank Tom Koornwinder for comments and discussions, which made us realise that Remark \ref{rmk:lemldu1}
needed correction.

\begin{appendix}

\section{Branching rules and Clebsch-Gordan coefficients}\label{app:cgc}

\begin{proof}[Proof of Theorem \ref{thm:cgc}]
The Clebsch-Gordan decomposition for $\Uq(\mathfrak{su}(2))$ and the corresponding intertwiner involving Clebsch-Gordan coefficients
is well-known, see e.g. \cite[\S 3.4]{KlimS}. 
With the convention of the standard orthonormal bases as in Section \ref{sec:quantizeduniversalenvelopingalg} the Clebsch-Gordan coefficients
give the unitary intertwiner
\begin{equation}\label{eq:standardCGC}
\begin{split}
& \gamma_{\ell_1, \ell_2} = \bigoplus_{\ell = |\ell_1-\ell_2|}^{\ell_1+\ell_2} \gamma^\ell_{\ell_1, \ell_2} 
\colon \bigoplus_{\ell = |\ell_1-\ell_2|}^{\ell_1+\ell_2} \mathcal{H}^\ell \to \mathcal{H}^{\ell_1} \otimes \mathcal{H}^{\ell_2}, \\
& \gamma^\ell_{\ell_1, \ell_2} \colon \mathcal{H}^\ell \to \mathcal{H}^{\ell_1} \otimes \mathcal{H}^{\ell_2}, \quad
e^\ell_n \mapsto \sum_{n_1=-\ell_1}^{\ell_1}\sum_{n_2=-\ell_2}^{\ell_2} \mathcal{C}^{\ell_1,\ell_2,\ell}_{n_1,n_2,n} \, e^{\ell_1}_{n_1}\otimes e^{\ell_2}_{n_2}, \\
& \bigl( t^{\ell_1}\otimes t^{\ell_2}\bigr) \bigl(\Delta(X)\bigr) \circ \gamma_{\ell_1,\ell_2} = \gamma_{\ell_1,\ell_2} \circ 
\sum_{\ell = |\ell_1-\ell_2|}^{\ell_1+\ell_2} t^\ell(X), \quad \forall\, X\in \Uq(\mathfrak{su}(2)). 
\end{split}
\end{equation}
Identifying the generators $(K^{1/2},K^{-1/2},E,F)$ of the Hopf algebra as in \cite[\S 3.1.1-2]{KlimS}
with the generators $(k^{-1/2}, k^{1/2}, f,e)$ as in Section \ref{sec:quantizeduniversalenvelopingalg}, we see that
the Hopf algebra structures are the same. Moreover, in this case the representations $t^\ell$ as in 
Section \ref{sec:quantizeduniversalenvelopingalg} correspond precisely with the representations 
$T_{1\ell}$ of \cite[\S 3.2.1]{KlimS} including the choice of orthonormal basis. This gives
$\mathcal{C}^{\ell_1,\ell_2,\ell}_{n_1,n_2,n}= C_q(\ell_1,\ell_2,\ell;n_1,n_2,n)$, where 
the right hand side is the notation of the Clebsch-Gordan coefficient as in \cite[\S 3.4.2, (41)]{KlimS}.
The Clebsch-Gordan coefficients are explicitly known in terms of terminating basic hypergeometric orthogonal polynomials, the so-called $q$-Hahn polynomials 
\cite{KoekLS}, see 
\cite[\S 3.4]{KlimS}.

In order to obtain Theorem \ref{thm:cgc} from \eqref{eq:standardCGC}, we use Remark \ref{rmk:identification}(ii). With the $\ast$-algebra isomorphism 
$\Psi$ as in Remark \ref{rmk:identification} we have $t^\ell(\Psi(X))= J^\ell t^\ell(X) J^\ell$ for all $X\in\Uq(\mathfrak{su}(2))$, where 
the intertwiner $J^\ell \colon \mathcal{H}^\ell\to \mathcal{H}^\ell$, $J^\ell\colon e^\ell_p \mapsto e^\ell_{-p}$ is a unitary involution. 
Note that the representations of $\Uq(\mathfrak{su}(2))$ in \eqref{eq:defrepresentationUqsu2} and of $\mathcal{B}$ in \eqref{eq:representationsB} 
are not related via the map of Remark \ref{rmk:identification}(ii), but they are related by the same operator $J^\ell$. 
Theorem \ref{thm:cgc} now follows after setting $\beta^\ell_{\ell_1,\ell_2} = ( J^{\ell_1}\otimes\text{Id})\circ \gamma^\ell_{\ell_1,\ell_2} \circ J^{\ell}$, 
so that in particular $C^{\ell_1,\ell_2, \ell}_{n_1,n_2,n} = \mathcal{C}^{\ell_1,\ell_2, \ell}_{-n_1,n_2,-n}
=C_q(\ell_1,\ell_2,\ell;-n_1,n_2,-n)$. 
\end{proof}

The Clebsch-Gordan coefficients satisfy several symmetry relations, and we require, see \cite[\S 3.4.4(70)]{KlimS},
\begin{equation}\label{eq:CGCsymml1l2}
C^{\ell_1,\ell_2, \ell}_{n_1,n_2,n} = C^{\ell_2,\ell_1, \ell}_{n_2,n_1,-n}.
\end{equation}

We need explicit expressions of the Clebsch-Gordan coefficients for the case $\ell_1+\ell_2=\ell$, 
which follow from the explicit expressions in \cite[\S 3.4.2, p.~80]{KlimS}.
For fixed $\ell\in\frac12\NN$, let $-\ell \leq m \leq \ell$, and we consider the case
$\ell_1 = (\ell + m)/2$ and $\ell_2 = (\ell - m)/2$.
The Clebsch-Gordan coefficients in this case are given by
\begin{equation}\label{eq:CGCinbottom}
\left(C^{\frac12(\ell+m), \frac12(\ell-m), \ell}_{i, j, k} \right)^2
= q^{2(i + \frac12(\ell+m))(j + \frac12(\ell-m))} \dfrac{
  \qbin{\ell + m}{\frac12(\ell+m) - i}_{q^2} \qbin{\ell - m}{\frac12(\ell-m) - j}_{q^2}}
  {\qbin{2\ell}{\ell - k}_{q^2}},
\end{equation}
assuming $i-j=k$ and $i\in \{-\frac12(\ell+m), \frac12(\ell+m)+1, \cdots, \frac12(\ell+m)\}$, 
$j\in \{-\frac12(\ell-m), \frac12(\ell-m)+1, \cdots, \frac12(\ell-m)\}$, 
$k\in \{-\ell,-\ell+1, \cdots, \ell\}$. 

Observe that the unitarity gives
\begin{equation}\label{eq:CGCorthorel-partial}
\begin{split}
\delta_{m,n} &= \langle e^\ell_m, e^\ell_n\rangle = 
\sum_{n_1,m_1=-\ell_1}^{\ell_1} \sum_{n_2,m_2=-\ell_2}^{\ell_2}
C^{\ell_1,\ell_2,\ell}_{m_1,m_2,m} \overline{C^{\ell_1,\ell_2,\ell}_{n_1,n_2,n}} 
\langle e^{\ell_1}_{m_1}\otimes e^{\ell_2}_{m_2}, e^{\ell_1}_{n_1}\otimes e^{\ell_2}_{n_2} \rangle 
\\ &= \sum_{n_1=-\ell_1}^{\ell_1} \sum_{n_2=-\ell_2}^{\ell_2}
C^{\ell_1,\ell_2,\ell}_{n_1,n_2,m} C^{\ell_1,\ell_2,\ell}_{n_1,n_2,n} 
\end{split}
\end{equation}
using the fact that the Clebsch-Gordan coefficients are real, see \cite[\S 3.2.4]{KlimS}. 

We also need some of the simplest cases of Clebsch-Gordan coefficients, see \cite[\S 3.2.4, p.~75]{KlimS}, 
\begin{equation}\label{eq:CGCforl1isl2ishalf}
C^{1/2,1/2,0}_{1/2,1/2,0} = \frac{-1}{\sqrt{1+q^2}}, \qquad 
C^{1/2,1/2,0}_{-1/2,-1/2,0} =  \frac{q}{\sqrt{1+q^2}},
\end{equation}
and the other Clebsch-Gordan coefficients $C^{1/2,1/2,0}_{m,n,0}$ being zero. 
Another case required is a generalisation of  a case of \eqref{eq:CGCforl1isl2ishalf}, see \cite[\S 3.2.4, p.~80]{KlimS}
\begin{equation}\label{eq:CGCforl1isl2jisminell}
C^{\ell, \ell, 0}_{-\ell, -\ell, 0} = q^{2\ell}\sqrt{\frac{1 - q^2}{1 - q^{4\ell + 2}}}
\end{equation}
and more generally
\begin{equation}\label{eq:CGCforends}
C^{\ell_1, \ell_2, \ell}_{-\ell_1, -\ell_2, \ell_2 - \ell_1} = q^{\ell_1 + \ell_2 - \ell} \left(
\frac{
  (q^2; q^2)_{2\ell_1} (q^2; q^2)_{2\ell_2} 
}{
  (q^2; q^2)_{\ell_1 + \ell_2 + \ell + 1} (q^2; q^2)_{\ell_1 + \ell_2 - \ell}
} (1 - q^{4\ell + 2})
\right)^{\frac{1}{2}}.
\end{equation}

\section{Proofs involving only basic hypergeometric series}\label{app:BproofsBHS}

In Appendix \ref{app:BproofsBHS} we collect the proofs of various intermediate 
results only involving basic hypergeometric series. For these proofs we 
use the results of Gasper and Rahman \cite{GaspR}. 
In particular we follow the standard notation of \cite{GaspR}, and recall $0<q<1$. 

\subsection{Proofs of lemmas for Theorem \ref{thm:ortho}}\label{subapp:explicitweight}

Here we present the details of the proof of Lemma \ref{lem:tediousequallity}, which is 
used in order to prove the explicit expression of the matrix entries of the weight in
terms of Chebyshev polynomials in Theorem \ref{thm:ortho}. 
We start with two intermediate results needed in the proof. 
Lemma \ref{lem:q-sheppard} can be viewed as a $q$-analogue of Sheppard's result 
\cite[Cor.~3.3.4]{AndrAR}.

\begin{lemma} \label{lem:q-sheppard}
For $b, c, d, e \in \CC^{\times}$ we have
\begin{align*}
&\sum_{k = 0}^n (dq^k; q)_{n-k} (eq^k; q)_{n-k}
  \frac{(q^{-n}, b, c; q)_k}{(q;q)_k}
  \left( \frac{de}{bc} q^n \right)^k \\
&\quad = c^n \sum_{k = 0}^n q^{2k(n-1) - 2\binom{k}{2}}
  \left( \frac{de}{bc^2} \right)^{k}
  (d/c, e/c; q)_{n - k}
  \frac{
    (q^{-n}, c, bcq^{1-n} / (de); q)_{k}
  }{
    (q; q)_{k}
  } q^k.
\end{align*}
\end{lemma}

\begin{proof}
Applying first \cite[(III.13)]{GaspR} and next \cite[(III.12)]{GaspR} gives
\begin{align*}
\pfq{3}{2}{q^{-n}, b, c}{d, e}{q}{\frac{de}{bc} q^n}
  &= \frac{(e/c;q)_n}{(e;q)_n}
    \pfq{3}{2}{q^{-n}, c, d/b}{d, cq^{1-n}/e}{q}{q} \\
  &= c^n \frac{(d/c, e/c; q)_n}{(d, e; q)_n}
    \pfq{3}{2}{q^{-n}, c, bcq^{1 - n}/de}{cq^{1-n}/d, cq^{1-n}/e}{q}{\frac{q}{b}}.
\end{align*}
Multiplying with $(d, e; q)_n$ and expanding gives
\begin{align*}
\sum_{k = 0}^n & (dq^k, eq^k; q)_{n - k}
  \frac{(q^{-n}, b, c; q)_k}{(q; q)_k}
  \left( \frac{de}{bc} q^n \right)^k \\
&= c^n \sum_{k = 0}^n
  \frac{(d/c, e/c; q)_n}{(q^{1 - n} c/d, q^{1 - n} c/e; q)_{k}}
  \frac{
    (q^{-n}, c, bcq^{1 - n} /de; q)_{k}
  }{
    (q; q)_{k}
  } \left(\frac{q}{b}\right)^k \\
&= c^n \sum_{k = 0}^n q^{2k(n-1) - 2\binom{k}{2}}
  \left( \frac{de}{bc^2} \right)^{k}
  (d/c, e/c; q)_{n - k}
  \frac{
    (q^{-n}, c, bcq^{1-n} / (de); q)_{k}
  }{
    (q; q)_{k}
  } q^k. 
\qedhere
\end{align*}
\end{proof}

Lemma \ref{lem:qdiff_coef} gives a simple $q$-analogue of a Taylor expansion.

\begin{lemma}[a $q$-Taylor formula] \label{lem:qdiff_coef}
Let $B(q^{-M}) = \sum_{t = 0}^{N} A_t (-1)^{t} (q^{-M}; q)_t$ be a polynomial in $q^{-M}$, then
\begin{align*}
\left. 
  \cfrac{\Delta^{n}_{q} B(q^{-M})}{q^{\binom{n+1}{2}} (q^{-n}; q)_{n}}
\right|_{M=0} &= A_{n}, 
\end{align*}
where $\Delta_q$ is the $q$-shift operator $\Delta_{q} B(q^{-M}) = q^{M} (B(q^{-M}) - B(q^{-M-1}))$.
\end{lemma}

\begin{proof}
Define $f_t(x) = (x;q)_t$, then
\begin{align*}
\Delta_q f_t(q^{-M}) &= q^M(f_t(q^{-M}) - f_t(q^{-M-1})) \\
  &= q (1 - q^t) (q^{-M}; q)_{t-1}
  = q (1 - q^t) f_{t-1}(q^{-M}).
\end{align*}
Repeated application of $\Delta_q$ on $f_t(q^{-M})$ then gives
\begin{equation} \label{eqn:qdiff1}
\Delta_q^n f_t(q^{-M}) = q^n (q^t; q^{-1})_n f_{t - n}(q^{-M})
  = q^{n + tn - \binom{n}{2}} (-1)^n (q^{-t}; q)_n f_{t - n}(q^{-M}).
\end{equation}
Putting $M = 0$ in (\ref{eqn:qdiff1}) the expression is zero if $n > t$ and $n < t$.
Therefore
\begin{align*}
\Delta_q^n f_t(q^{-M}) \at[\big]{M = 0} 
  = \delta_{n, t} (-1)^n q^{\binom{n+1}{2}} (q^{-n}; q)_n.
\end{align*}
This gives the result.
\end{proof}

\begin{proposition} \label{prop:tediousequallity}
Let $\ell \in \frac{1}{2}\ZZ$ and $k,p \in \NN$ such that $k,p \leq 2\ell$, $k-p \leq 0$ and $k+p \leq 2\ell$.
Take $s \in \NN$ such that $s \leq p$ and define
\begin{equation*}
\begin{split}
e^{\ell}_s(k,p) = 
  q^{2ks + k + p - 2\ell}
  &\sum_{i = 0}^{p-s} q^{-2i(s+1)} \qbin{k}{i}_{q^2} \qbin{p}{i+s}_{q^2} \\
  &\quad \times \sum_{n = 0}^{k-s} q^{2n(-2i+k+p-s+1)} 
    \frac{
      \qbin{2\ell-k}{n+s}_{q^2} \qbin{2\ell-p}{n}_{q^2}
    }{
      \qbin{2\ell}{i+n+s}_{q^2}^2
    }.
\end{split}
\end{equation*}
Then we have
\begin{equation*} 
\begin{split}
e^{\ell}_s(k,p) &= q^{2p(2\ell+1) - p^2 - 2\ell + k} 
  \frac{(1 - q^{4\ell+2})}{(1 - q^{2k+2})}
  \frac{(q^2;q^2)_{2\ell-p} (q^2;q^2)_{p}}{(q^2;q^2)_{2\ell}} \\
  &\quad \times \sum_{T=0}^{p-s} q^{T^2 - (4\ell+3)T} (-1)^{p-T}
    \frac{(q^{2k-4\ell};q^2)_{p-T}}{(q^{2k+4};q^2)_{p-T}}
    \frac{(q^{4\ell-2T+4};q^2)_{T}}{(q^2;q^2)_{T}}.
\end{split}
\end{equation*}
\end{proposition}

\begin{proof}
The proof proceeds along the lines of the proof of \cite[Proposition A.1]{KoelvPR12}, and we only give a sketch.
We start by reversing the inner summation, using $M = 2\ell - k - s - m$ in $e^{\ell}_s(k,p)$ to get
\begin{equation}
\begin{split}
e^{\ell}_s(k,p) &= q^{2ks + k + p - 2\ell + 2(k+p-s+1)(2\ell-k-s)} \sum_{i=0}^{p-s} q^{2i(2k + s - 1)} \qbin{k}{i}_{q^2} \qbin{p}{i+s}_{q^2} \\
    &\qquad \times \sum_{M=0}^{k-s} q^{2M(2i - k - p + s - 1)} 
      \frac{
        \qbin{2\ell - k}{M}_{q^2} \qbin{2\ell - p}{2\ell - k - s - M}_{q^2}
      }{
        \qbin{2\ell}{2\ell - p - M + i}_{q^2}^2
      }
\end{split}
\label{eqn:begin_elskp}
\end{equation}
We rewrite the inner summation over $M$
\begin{equation} \label{eqn:invertedsumM}
\begin{split}
&\sum_{M=0}^{k-s} q^{2M(2i - k - p + s - 1)} 
  \frac{
    \qbin{2\ell - k}{M}_{q^2} \qbin{2\ell - p}{2\ell - k - s - M}_{q^2}
  }{
    \qbin{2\ell}{2\ell - p - M + i}_{q^2}^2
  } \\
&\quad = \frac{
  (q^2;q^2)_{2\ell-k} (q^2;q^2)_{2\ell-p} (q^2;q^2)_{2\ell-k+i} (q^2;q^2)_{k-i}
}{
  (q^2;q^2)_{2\ell}^2 (q^2;q^2)_{2\ell-k-s}
} \\
&\qquad \times \sum_{M=0}^{k-s} q^{-2M(k+s)} \frac{
  (q^{2k-4\ell+2s};q^2)_{M} (q^{2k-2i+2};q^2)_{M}
}{
  (q^2;q^2)_M (q^{-4\ell+2k-2i};q^2)_{M}
} B(q^{-2M}),
\end{split}
\end{equation}
where $B(q^{-2M}) = q^{2M(s+i-p)}(q^{4\ell-2k-2M+2};q^2)_i(q^{2k-2p+2s+2M+2};q^2)_{p-i-s}$ is a polynomial in $q^{-2M}$ of degree $p-s$ that depends on $\ell, k, p, i$ and $s$.
The polynomial $B(q^{-2M})$ has an expansion in $(-1)^t (q^{-2M};q^2)_t$ such that $B(q^{-2M}) = \sum_{t=0}^{p-s} A_t (-1)^t (q^{-2M}; q^2)_t$. By Lemma \ref{lem:qdiff_coef}, the coefficients $A_t$ are obtained by repeated application of the $q$-shift operator;
\begin{equation}
\left. \frac{
  \Delta_{q^2}^t B(q^{-2M})
}{
  q^{2\binom{t+1}{2}} (q^{-2t};q^2)_t
} \right|_{M=0} = A_t. \label{eqn:qdiffcoef}
\end{equation}
We substitute $B(q^{-2M}) = \sum_{t=0}^{p-s} A_t (-1)^t (q^{-2M}; q^2)_t$ into \eqref{eqn:invertedsumM} and we interchange the summations over $M$ and $t$. 
Then the sum over $M$ can be rewritten as a summable ${}_2\varphi_1$. 
Using the reversed $q$-Chu-Vandermonde summation \cite[(II.7)]{GaspR} we can rewrite \eqref{eqn:invertedsumM} as
\begin{equation} \label{eqn:newtsum}
\begin{split}
&\frac{
  (q^2;q^2)_{2\ell - k} (q^2;q^2)_{2\ell - p}
}{
  (q^2;q^2)_{2\ell}^2
}
\frac{
  (q^2;q^2)_{2\ell - k + i} (q^2;q^2)_{k-i}
}{
  (q^2;q^2)_{2\ell - k - s}
} \\
&\qquad \times \sum_{t=0}^{p-s} q^{-2t(k+s+t+1) + 2\binom{t}{2}} A_t
  \frac{
    (q^{2k - 4\ell + 2s}, q^{2k - 2i +2}; q^2)_{t} (q^{-4\ell-2}; q^2)_{2\ell - k - s - t}
  }{
    (q^{-4\ell + 2k - 2i}; q^2)_{2\ell - k - s}
  }.
\end{split}
\end{equation}
Substituting \eqref{eqn:qdiffcoef} and \eqref{eqn:newtsum} into \eqref{eqn:begin_elskp} and interchanging the sums over $i$ and $t$,
we find that $e^{\ell}_{s(k,p)}$ is 
\begin{align*}
&q^{2ks + k + p - 2\ell + 2(2\ell-k-s)(k+p-s+1)}
  \frac{
    (q^2;q^2)_{k} (q^2;q^2)_{p} (q^2;q^2)_{2\ell-k}^2 (q^2;q^2)_{2\ell-p}
  }{
    (q^2;q^2)_{2\ell}^2 (q^2;q^2)_{2\ell - k - s} (q^2;q^2)_s (q^2;q^2)_{p-s}
  } \\
  &\quad \times \sum_{t=0}^{p-s} 
    q^{-2t(k+s+t+1) - 2\binom{t}{2}} (q^{2k-4\ell+2s};q^2)_{t} (q^{-4\ell-2};q^2)_{2\ell-k-s-t} (q^{2k+2};q^2)_{t} \\
  &\qquad \times \left. \frac{\Delta_{q^2}^t}{q^{2\binom{t+1}{2}}(q^{-2t};q^2)_{t}} \right|_{M=0}
    q^{2M(s-p)} (q^{2k-2p+2s+2M+2};q^2)_{p-s} \\
  &\qquad \quad \times \sum_{i=0}^{p-s} q^{2i(-2\ell+p-s-t-1)}
  \frac{
    (q^{2s-2p}, q^{-2k}, q^{4\ell - 2k - 2M + 2};q^2)_i
  }{
    (q^2, q^{-2k-2M}, q^{-2k-2t}; q^2)_i
  }.
\end{align*}
The inner sum over $i$ is a ${}_3\varphi_2$-series and after some rewriting, it can be transformed using Lemma \ref{lem:q-sheppard}. 
The sum over $i$ becomes
\begin{align*}
&\sum_{i=0}^{p-s} q^{2i(-2\ell+p-s-t-1)}
  (q^{-2k-2M+2i}, q^{-2k-2t+2i}; q^2)_{p-s-i} \\
  &\qquad \qquad \qquad \times \frac{(q^{4\ell-2k-2M+2}, q^{-2k}, q^{2s-2p}; q^2)_i}{(q^2;q^2)_{i}} \\
&\qquad = q^{-2k(p-s)} \sum_{i=0}^{p-s} q^{2i(2p - 2s + k - t - 2\ell - 2) - 4\binom{i}{2}}
  (q^{-2t}, q^{-2M}; q^2)_{p-s-i} \\
  &\qquad \qquad \qquad \times \frac{(q^{-2p+2s}, q^{-2k}, q^{4\ell+2t-2p+2s+4}; q^2)_{i}}{(q^2;q^2)_i}.
\end{align*}
We observe that the evaluation of the $t$-order $q$-difference operator yields only one non-zero term in the sum over $i$, namely the
one corresponding to $i = p-s-t$. After simplifications we have
\begin{align*}
e^{\ell}_s(k,p) &= q^{(4\ell-2p+3)s + s^2 + k - 2\ell - p}
  \frac{(1 - q^{4\ell+2})}{(1 - q^{2k+2})}
  \frac{(q^2;q^2)_{2\ell-p}(q^2;q^2)_{p}}{(q^2;q^2)_{2\ell}} \\
  &\quad \times \sum_{t=0}^{p-s} (-1)^{s+t} q^{t^2 + t(4\ell-2p+2s+3)}
    \frac{
      (q^{2k-4\ell};q^2)_{s+t} (q^{4\ell-2p+2s+2t+4};q^2)_{p-s-t}
    }{
      (q^{2k+4};q^2)_{s+t} (q^2;q^2)_{p-s-t}
    }.
\end{align*}
Reversing the order of summation using $T = p - s - t$ proves the proposition.
\end{proof}

\begin{proof}[Proof of Lemma \ref{lem:tediousequallity}]
Recall from Proposition \ref{prop:weightexpandedasLaurentpol}, that we have
\begin{align*}
W(\psi)_{k,p}(A^{\lambda}) = \sum_{s = -\frac{1}{2}(k+p)}^{\frac{1}{2}(k+p)} d^{\ell}_{s}(k,p) q^{2s\lambda}.
\end{align*}
Therefore the coefficients of $d^{\ell}_{s}(k,p)$ and $\alpha^{\ell}_{i}(k,p)$ are related by
\begin{align*}
d^{\ell}_{s}(k,p) = \sum_{t=0}^{p-(s+\frac{1}{2}(p-k))} \alpha_{t}^{\ell}(k,p).
\end{align*}
Suppose that $p-k \leq 0$, $k + p \leq 2\ell$ and $s \geq \frac{1}{2}(k - p)$.
Using Proposition \ref{prop:weightexpandedasLaurentpol} we find the expression
\begin{align*}
d^{\ell}_{s + \frac{1}{2}(k-p)}(k, p) = q^{2ks + k + p - 2\ell}
  &\sum_{i = 0}^{p-s} q^{-2i(s+1)} \qbin{k}{i}_{q^2} \qbin{p}{i+s}_{q^2} \\
  &\quad \times \sum_{n = 0}^{2\ell - k - s} q^{2n(-2i+k+p-s+1)} 
    \frac{
      \qbin{2\ell-k}{n+s}_{q^2} \qbin{2\ell-p}{n}_{q^2}
    }{
      \qbin{2\ell}{i+n+s}_{q^2}^2
    }.
\end{align*}
By taking $e^{\ell}_s(k,p) = d^{\ell}_{s + \frac{1}{2}(k-p)}$, Proposition \ref{prop:tediousequallity} shows that $d^{\ell}_{s+\frac{1}{2}(k-p)}(k,p)$ is equal to
\begin{equation}
\label{eqn:finaldlkp}
\begin{split}
&q^{2p(2\ell+1) - p^2 - 2\ell + k} 
  \frac{(1 - q^{4\ell+2})}{(1 - q^{2k+2})}
  \frac{(q^2;q^2)_{2\ell-p} (q^2;q^2)_{p}}{(q^2;q^2)_{p}} \\
  &\quad \times \sum_{t=0}^{p-s} q^{t^2 - (4\ell+3)t} (-1)^{k-t}
    \frac{(q^{2k-4\ell};q^2)_{p-t}}{(q^{2k+4};q^2)_{p-t}}
    \frac{(q^{4\ell-2t+4};q^2)_{t}}{(q^2;q^2)_{t}}.
\end{split}
\end{equation}
Comparing (\ref{eqn:finaldlkp}) with the explicit expression of $\alpha^{\ell}_{i}(k,p)$ yields the statement of the lemma.
\end{proof}

\subsection{Proofs of lemmas for Theorem \ref{thm:ldu}}\label{subapp:forproofLDU}

Here we present the proof of the technical lemmas needed in the proof of Theorem \ref{thm:ldu}, in which 
the LDU-decomposition of the weight matrix is presented. 

For completeness we start by recalling the linearisation and connection relations for the 
continuous $q$-ultraspherical polynomials, see \cite[\S 10.11]{AndrAR}, \cite[\S 13.3]{Isma}, \cite[(7.6.14), (8.5.1)]{GaspR}.
The connection coefficient formula is
\begin{equation}\label{eq:connectionformulacontqultrapols}
C_n(x;\gamma|q) = \sum_{k=0}^{\lfloor n/2\rfloor} \frac{1-\beta q^{n-2k}}{1-\beta} \beta^k 
\frac{(\gamma/\beta;q)_k(\gamma;q)_{n-k}}{(q;q)_k (\beta q;q)_{n-k}} C_{n-2k}(x;\beta|q),
\end{equation}
and the linearisation formula is
\begin{equation}\label{eq:linearisationformulacontqultrapols}
\begin{split}
 C_n(x;\beta|q)C_m(x;\beta|q) &= \sum_{k=0}^{m \wedge n} 
\frac{1-\beta q^{m+n-2k}}{1-\beta}
\frac{(q;q)_{m+n-2k}(\beta;q)_{m-k}}{(\beta^2;q)_{m+n-2k}(q;q)_{m-k}} \\
&\quad \times \frac{(\beta;q)_{n-k}(\beta;q)_k(\beta^2;q)_{m+n-k}}{(q;q)_{n-k}(q;q)_k(q\beta;q)_{m+n-k}}
C_{m+n-2k}(x;\beta|q).
\end{split}
\end{equation}

The proof of Lemma \ref{lem:ldu1} follows the lines of the proof of \cite[Lem.~2.7]{KoelvPR13} closely. 

\begin{proof}[Proof of Lemma \ref{lem:ldu1}]
In order to evaluate the integral of Lemma \ref{lem:ldu1}, we observe that the weight function in the integral
is the weight function \eqref{eq:ortho-contqultraspherpols}
for the continuous $q$-ultraspherical (in base $q^2$) with $\beta=q^{2k+2}$. Rewrite the 
product of the two continuous $q$-ultraspherical (in base $q^2$) with $\beta=q^{2k+2}$ as
a sum over $i$ of $C_{m + n - 2k - 2i}(x; q^{2k + 2} | q^2)$
using \eqref{eq:linearisationformulacontqultrapols}. Since the Chebyshev polynomials
can be viewed as continuous $q$-ultraspherical polynomials, 
which in base $q^2$ is $U_{m + n - 2t}(x) = C_{m + n - 2t}(x; q^2 | q^2)$, we can use
\eqref{eq:connectionformulacontqultrapols} to write the Chebyshev polynomial as a sum
of continuous $q$-ultraspherical polynomials with $\beta=q^{2k+2}$.
Plugging in the two summations, and next using the orthogonality 
relations \eqref{eq:ortho-contqultraspherpols} shows that we can evaluate the 
integral of Lemma \ref{lem:ldu1} as a single sum;
\begin{equation}\label{eq:prooflemldu1-1}
\begin{split}
& q^{(2k + 2)(k - t)}
  \frac{
    (q^{2k + 2}, q^{2k + 4}; q^2)_{\infty}
  }{
    (q^2, q^{4k + 4}; q^2)_{\infty} 
  } \\
  & \times \sum_{r = 0 \vee t - k}^{t \wedge m - k}
    \frac{(1 - q^{2m + 2n - 2k + 2 - 4r})}{(1 - q^{2m + 2n - 2k + 2 - 2r})}
    \frac{(q^{2k + 2}; q^2)_{r}}{(q^2; q^2)_{r}}
    \frac{(q^{2k + 2}; q^2)_{m - k - r}}{(q^2; q^2)_{m - k - r}}
    \frac{(q^{2k + 2}; q^2)_{n - k - r}}{(q^2; q^2)_{n - k - r}} \\
    &\qquad\quad \times \frac{
      (q^{4k + 4}; q^2)_{m + n - 2k - r}
    }{
      (q^{2k + 2}; q^2)_{m + n - 2k - r}
    }
    \frac{(q^{-2k}; q^2)_{k + r - t}}{(q^{2}; q^2)_{k + r - t}}
    \frac{
      (q^2; q^2)_{m + n - t - k - r}
    }{
      (q^{2k + 4}; q^2)_{m + n - t - k - r}
    } q^{(2k + 2)r}. 
\end{split}
\end{equation}
We consider two cases; $k \geq t$ and $k \leq t$.
If $k \geq t$, note that 
\begin{equation*}
\frac{1 - q^{2m + 2n - 2k + 2 - 4r}}{1 - q^{2m + 2n - 2k + 2 - 2r}}
  = q^{-2r} \frac{
    (q^{k - m - n + 1}, -q^{k - m - n + 1}, q^{2k - 2m - 2n - 2}; q^2)_{r}
  }{
    (q^{k - m - n - 1}, -q^{k - m - n - 1}, q^{2k - 2m - 2n}; q^2)_{r}
  },
\end{equation*}
so that we can rewrite  \eqref{eq:prooflemldu1-1} as a terminating very-well-poised ${}_8\varphi_7$-series.
Explicitly, 
\begin{gather*}
  \frac{
    (q^{2k + 2}, q^{2k + 4}; q^2)_{\infty}
  }{
    (q^2, q^{4k + 4}; q^2)_{\infty}
  }
  \frac{(q^{2k + 2}; q^2)_{m - k}}{(q^2; q^2)_{m - k}}
  \frac{(q^{2k + 2}; q^2)_{n - k}}{(q^2; q^2)_{n - k}} \\
  \times \frac{(q^{4k + 4}; q^2)_{m + n - 2k}}{(q^{2k + 2}; q^2)_{m + n - 2k}}
  \frac{(q^{-2k}; q^2)_{k - t}}{(q^2; q^2)_{k - t}} 
  \frac{(q^2; q^2)_{m + n - t - k}}{(q^{2k + 4}; q^2)_{m + n - t - k}}
  q^{(2k + 2)(k - t)} \\
   \times {}_8W_7(q^{2k - 2m - 2n - 2}; q^{2k - 2m}, q^{2k - 2n}, q^{2t - 2m - 2n - 2}, q^{2k + 2}, q^{-2t}; q^2, q^{-2k}).
\end{gather*}
Here we use the notation of \cite[\S 2.1]{GaspR}. 
Using Watson's transformation formula \cite[(III.18)]{GaspR} and recalling the definition \eqref{eqn:q-racah} of the $q$-Racah polynomials we can rewrite the ${}_8 W_7$ as a balanced ${}_4 \varpi_3$;
\begin{gather*}
\frac{
  (q^{2k - 2m - 2n}, q^{-2t}; q^2)_{t}
}{
  (q^{2k - 2t + 2}, q^{-2m - 2n - 2}; q^2)_{t}
}  \pfq{4}{3}{
  q^{-2k}, q^{2t - 2m - 2n - 2}, q^{2k + 2}, q^{-2t}
}{
  q^{-2m}, q^{-2n}, q^2
}{q^2}{q^2} \\
= \frac{
  (q^{2k - 2m - 2n}, q^{-2t}; q^2)_{t}
}{
  (q^{2k - 2t + 2}, q^{-2m - 2n - 2}; q^2)_{t}
} R_k(\mu(t); 1, 1, q^{-2m - 2}, q^{-2n - 2}; q^2).
\end{gather*}
Simplifying the $q$-shifted factorials gives the required expression.

For the second case $k \leq t$, we proceed similarly. After applying Watson's transformation 
we also employ Sears' transformation for a terminating balanced ${}_4 \varphi_3$ series \cite[(III.15)]{GaspR} 
in order to recognise the expression for the $q$-Racah polynomial. 
\end{proof}

The proof of the generalisation of Lemma \ref{lem:ldu1} in Remark \ref{rmk:lemldu1} is sketched here.

\begin{proof}[Proof of Remark \ref{rmk:lemldu1}]
We denote the integral by $I^{m,n}_{k,t}(\alpha,\beta)$. 
Let us first assume $\beta=\alpha q^{k+1}$. We use the 
linearisation \eqref{eq:linearisationformulacontqultrapols} to expand
the product of the first two $q$-ultraspherical polynomials and then 
using the orthogonality 
\eqref{eq:ortho-contqultraspherpols}, we obtain a single term which 
is non-zero only for $k\leq t\leq n$. In this case, we find after simplification 
\begin{gather*}
I^{m,n}_{k,t}(\alpha,\alpha q^{k+1}) 
= \frac{(\alpha q^{k+1};q)_{m-t}}{(q;q)_{m-t}} 
\frac{(\alpha q^{k+1};q)_{n-t}}{(q;q)_{n-t}} 
\frac{(\alpha q^{k+1};q)_{t-k}}{(q;q)_{t-k}}
\frac{(\alpha q^{m+n-t+2}, \alpha q^{k+1};q)_\infty}{(\alpha^2 q^{m+n+k-t+2}, q;q)_\infty}.
\end{gather*}
In the general case, we first use the connection formula 
\eqref{eq:connectionformulacontqultrapols} to rewrite 
$C_{m+n-2t}(x;\beta | q)$ in terms of continuous $q$-ultraspherical
of parameter $\alpha q^{k+1}$ and then we use the special case above. This 
gives a finite sum, which, after simplification, can be written as 
a very-well-poised ${}_8\varphi_7$-series. Explicitly, we find
\begin{gather*}
I^{m,n}_{k,t}(\alpha,\beta) = 
\frac{(\beta;q)_{m+n-2t}}{(\alpha q^{k+1};q)_{m+n-2t}} 
\frac{(\alpha q^{k+1};q)_{m-t}}{(q;q)_{m-t}} 
\frac{(\alpha q^{k+1};q)_{n-t}}{(q;q)_{n-t}}
\frac{(\alpha q^{k+1}, \alpha q^{k+1},\alpha q^{m+n-t+2}, q^{t-k+1};q)_\infty}
{(\alpha^2q^{m+n+k-t+2},\alpha q^{t+1}, q,q;q)_\infty} \\
\qquad \times 
{}_8W_7(q^{2t-1-m-n-k}/\alpha; 
\beta q^{-1-k}/\alpha, q^{t-m}, q^{t-n}, \alpha q^{t+1}, 
q^{t-m-n-1}/\alpha; q, \frac{q}{\alpha\beta q^k})
\end{gather*}
using the notation ${}_8W_7$ for very-well-poised series, see \cite[Ch.~2]{GaspR}.  
Transforming the very-well-poised ${}_8W_7$ to a terminating balanced
${}_4\varphi_3$-series by Watson's transformation formula
 \cite[(2.5.1)]{GaspR}  (with $a,b,c,d,e,f$ replaced by
 $q^{2t-1-m-n-k}/\alpha, q^{t-n}, q^{t-m-n-1}/\alpha, \beta q^{-1-k}/\alpha,
 \alpha q^{t+1}, q^{t-m}$) gives, again after simplification, 
\begin{gather*}
I^{m,n}_{k,t}(\alpha,\beta) = (\alpha\beta q^{1+n};q)_{m-t} 
\frac{(\alpha q^{k+1};q)_{m-t}}{(q;q)_{m-t}} 
\frac{(\beta;q)_{n-t}}{(q;q)_{n-t}} 
\frac{(\alpha q^{k+1}, \alpha q^{k+1},\alpha q^{m+n-t+2}, q^{t-k+1};q)_\infty}
{(\alpha^2q^{n+k+2},\alpha q^{t+1}, q,q;q)_\infty} \\
\times 
\pfq{4}{3}{q^{t-m}, \beta q^{-1-k}/\alpha, \alpha q^{t+1}, q^{1+n-k}}
{q^{t-k+1},q^{t-m-k}/\alpha, \alpha \beta q^{1+n}}{q}{q}
\end{gather*}
Note that $k>t$ is possible and the ${}_4\varphi_3$-series has $q^{t-k+1}$ as a lower parameter, which 
is understood to combine with the factor $(q^{t-k+1};q)_\infty$. We transform
the balanced ${}_4\varphi_3$-series by Sears' transformation 
\cite[(III.15)]{GaspR} (with $q^{-n}, a,b,c,d,e,f$ replaced by 
$q^{t-m}, \beta q^{-1-k}/\alpha, \alpha q^{t+1}, q^{1+n-k}, 
\alpha \beta q^{1+n}, q^{t-k+1},q^{t-m-k}/\alpha$). This removes 
the apparent singularity and gives the ${}_4\varphi_3$-series as 
in Remark \ref{rmk:lemldu1}. The factor is obtained manipulating 
the shifted factorials. 
\end{proof}

In the proof of the LDU-decomposition of Theorem \ref{thm:ldu} we have also used 
Lemma \ref{lem:ldu2}, whose proof is presented next. 

\begin{proof}[Proof of Lemma \ref{lem:ldu2}]
We first write the $q$-Racah polynomial in the left hand side of Lemma \ref{lem:ldu2}
as a ${}_4\varphi_3$-series and interchange the sums to get
\begin{equation}
\label{eq:proof_lem_ldu}
\begin{split}
&\sum_{j = 0}^{k} \frac{
    (q^{-2k}, q^{2k + 2}; q^2)_{j}
  }{
    (q^{2}, q^{2}, q^{-2m}, q^{-2n}; q^2)_{j} 
  } q^{2j}
\sum_{t = j}^{n} (-1)^t
  \frac{(q^{2m - 4\ell}; q^2)_{n - t}}{(q^{2m + 4}; q^2)_{n - t}}
  \frac{(q^{4\ell + 4 - t}; q^2)_{t}}{(q^2; q^2)_{t}} \\
  &\qquad \qquad \times (q^{-2t}, q^{-2m - 2n + 2t - 2}; q^2)_{j}
  (1 - q^{2m + 2n + 2 - 4t}) q^{2\binom{t}{2} - (4\ell + 4)t},
\end{split}
\end{equation}
which has a well-poised structure. 
Relabeling $t = j + p$ gives that the inner sum is equal to
\begin{gather*}
\frac{(q^{-2m - 2n - 2})_{j}}{(q^{4\ell - 2m - 2n + 2}; q^2)_{j}}
\frac{(q^{-4\ell - 2}; q^2)_{j}}{(q^2; q^2)_{j}} \\
\times (q^{-2j}; q^2)_j (q^{-2m - 2n - 2 + 2j}; q^2)_{j}
(1 - q^{2m + 2n + 2 - 4j}) q^{(4\ell + 6)j} \\
\times \sum_{p = 0}^{n - j}
  \frac{
    (q^{-4\ell - 2 + 2j}, q^{2j - m - n + 1}, -q^{2j - m - n + 1}, 
      q^{-2m -2n - 2 + 4j}; q^2)_{p} 
  }{
    (q^{2}, q^{4\ell - 2m - 2n + 2 + 2j}, q^{2j - m - n - 1},
      -q^{2j - m - n - 1}; q^2)_{p}
  } q^{(4\ell - 2j + 2)p}.
\end{gather*}
Multiplying the inner sum with $\frac{(q^{2j - 2m}, q^{2j - 2n}; q^2)_p}{(q^{2j - 2n}, q^{2j - 2n}; q^2)_p}$, we can rewrite the sum as a very-well-poised ${}_6 \varphi_5$-series
\begin{align*}
{}_6 W_5(q^{4j - 2m - 2n - 2}; q^{-2-4\ell+2j}, q^{2j-2n}, q^{2j-2m}; q^2, q^{4\ell - 2j + 2}).
\end{align*}
This very-well-poised ${}_6 \varphi_5$ series can be evaluated by using \cite[(II.21)]{GaspR} as
\begin{align*}
\frac{
  (q^{4j - 2m - 2n}, q^{4\ell - 2n + 2}; q^2)_{n - j}
}{
  (q^{4\ell - 2m - 2n + 2j + 2}, q^{2j - 2n}; q^2)_{n - j}
}.
\end{align*}
Straightforward calculations show that \eqref{eq:proof_lem_ldu} is given by
\begin{align*}
\frac{(q^{2m - 4\ell}; q^2)_{n}}{(q^{2m + 4}; q^2)_{n}}
(1 - q^{2m + 2n + 2})
&\frac{(q^{-2m - 2n}; q^2)_{n}}{(q^{4\ell - 2m - 2n + 2}; q^2)_{n}}
\frac{(q^{4\ell + 2 - 2n}; q^2)_{n}}{(q^{-2n}; q^2)_{n}} \\
&\times \sum_{j = 0}^{k}
  \frac{
    (q^{-2k}, q^{2k + 2}, q^{-2\ell - 2}; q^2)_{j}
  }{
    (q^2, q^2, q^{-2\ell}; q^2)_{j}
  } q^{2j}.
\end{align*}
The inner sum can be rewritten as a balanced ${}_3 \varphi_2$.
Using the $q$-Saalsch\"utz transformation \cite[(II.12)]{GaspR},
we find that \eqref{eq:proof_lem_ldu} reduces to
\begin{equation}
\label{eq:proof_lem_ldu2}
\frac{(q^{2m - 4\ell}; q^2)_{n}}{(q^{2m + 4}; q^2)_{n}}
(1 - q^{2m + 2n + 2})
\frac{(q^{-2m - 2n}; q^2)_{n}}{(q^{4\ell - 2m - 2n + 2}; q^2)_{n}}
\frac{(q^{4\ell + 2 - 2n}; q^2)_{n}}{(q^{-2n}; q^2)_{n}}
\frac{
  (q^{-2k}, q^{4\ell + 4}; q^2)_{k}
}{
  (q^2, q^{4\ell - 2k + 2}; q^2)_{k}
}.
\end{equation}
Finally, simplifying the $q$-Pochhammer symbols we prove that \eqref{eq:proof_lem_ldu2} is equal to the right hand side of Lemma \ref{lem:ldu2}. 
\end{proof}

\subsection{Proof of \eqref{eq:diffeqN1}.}
\label{subapp:forproofqdiffeq}
Here we verify that \eqref{eq:diffeqN1} is valid entry-wise. It follows from Theorem \ref{thm:diff_eqn_Pn}, Propositions \ref{prop:usefull} and \ref{prop:diffeqforhatPhi}, and the fact that $\mathcal{N}_i(z)=\mathcal{M}_i(z^{-1})$, that the $(m,n)$-th entry of the left hand side of \eqref{eq:diffeqN1} is given by
\begin{multline}
\label{eq:laurent_series_qdiff_eq_1}
\sum_{k=0}^{n-1} q^{\ell-m-2n+2k+3} (1-q^{2n})  \left( C^{\frac{n-1}{2}, \ell-\frac{n-1}{2},\ell}_{-\frac{n-1}{2}+k,\ell-m-\frac{n-1}{2}+k,m-\ell}\right)^2 \frac{z^{-\ell+m+n-2k}}{(1-q^2)^2(1-z^2)} \\
- \sum_{k=0}^{n} q^{\ell-m-2n+2k+1}  \left( C^{\frac{n}{2}, \ell-\frac{n}{2},\ell}_{-\frac{n}{2}+k,\ell-m-\frac{n}{2}+k,m-\ell}\right)^2 \frac{(z^2-q^{2n+2})z^{-\ell+m+n-2k}}{(1-q^2)^2(1-z^2)}.
\end{multline}
On the other hand, the $(m,n)$-th entry of the right hand side of \eqref{eq:diffeqN1} is given by
\begin{equation}
\label{eq:laurent_series_qdiff_eq_2}
\sum_{k=0}^{n} q^{2\ell-m-2n+4k+1}  \left( C^{\frac{n}{2}, \ell-\frac{n}{2},\ell}_{-\frac{n}{2}+k,\ell-m-\frac{n}{2}+k,m-\ell}\right)^2 \frac{(1-q^{-2}z^2) z^{-\ell+m+n-2k}}{(1-q^2)^2(1-z^2)}.
\end{equation}
If we multiply \eqref{eq:laurent_series_qdiff_eq_1} and \eqref{eq:laurent_series_qdiff_eq_2} by $(1-q^2)^2(1-z^2)$, they become Laurent series in $z^s$. By equating the coefficients of $z^s$, the proof of \eqref{eq:diffeqN1} boils down to prove the following equality:
\begin{align*}
&q^{\ell-m-2n+2k+3}(1-q^{2n}) \left( C^{\frac{n-1}{2}, \ell-\frac{n-1}{2},\ell}_{-\frac{n-1}{2}+k,\ell-m-\frac{n-1}{2}+k,m-\ell}\right)^2 \\
&\qquad - q^{\ell-m-2n+2k+3}  \left( C^{\frac{n}{2}, \ell-\frac{n}{2},\ell}_{-\frac{n}{2}+k+1,\ell-m-\frac{n}{2}+k+1,m-\ell}\right)^2 \\
&\qquad + q^{\ell-m+2k+3}  \left( C^{\frac{n}{2}, \ell-\frac{n}{2},\ell}_{-\frac{n}{2}+k,\ell-m-\frac{n}{2}+k,m-\ell}\right)^2  \\
&\quad =  q^{\ell-m-2n+4k+3} \left( C^{\frac{n}{2}, \ell-\frac{n}{2},\ell}_{-\frac{n}{2}+k+1,\ell-m-\frac{n}{2}+k+1,m-\ell}\right)^2 \\
&\qquad - q^{\ell-m-2n+4k+5} \left( C^{\frac{n}{2}, \ell-\frac{n}{2},\ell}_{-\frac{n}{2}+k,\ell-m-\frac{n}{2}+k,m-\ell}\right)^2.
\end{align*} 
The last equation is proven by using \eqref{eq:CGCinbottom} and performing some simple manipulation of the $q$-binomial coefficients.

\subsection{Proof of \eqref{eq:matricesconjugatedqdiffop}.}\label{subapp:matricesconjugatedqdiffop}

The expressions for $\mathcal{K}_1$ and $\mathcal{K}_2$ can be obtained using the inverse of $L$ given in 
Theorem \ref{thm:ldu}, so that $\mathcal{K}_1$ and $\mathcal{K}_2$ are uniquely determined.
So it suffices to check \eqref{eq:matricesconjugatedqdiffop}. 

In order to prove \eqref{eq:matricesconjugatedqdiffop} we need to distinguish between the cases $i=1$ and $i=2$,
since there is no symmetry between the two cases. 

The case $i=1$ of \eqref{eq:matricesconjugatedqdiffop} is
$\breve{L}^t(z) \mathcal{M}_1(z) = \mathcal{K}_1(z) \breve{L}^t(qz)$. 
Consider the $(m,n)$-entry of $\mathcal{K}_1(z) \breve{L}^t(qz)-\breve{L}^t(z) \mathcal{M}_1(z)$; 
\begin{equation}\label{eq:case1-1}
\mathcal{K}_1(z)_{m,m} \breve{L}(qz)_{n,m} - \mathcal{M}_1(z)_{n-1,n}\breve{L}(z)_{n-1,m}(z) - \mathcal{M}_1(z)_{n,n}\breve{L}(z)_{n,m}(z).
\end{equation}
The convention is that matrices with negative labels are zero. 
In case $m>n$, \eqref{eq:case1-1} is zero, since $L$ is lower triangular. 
In case $m=n$, we see that $\mathcal{K}_1(z)_{m,m}=\mathcal{M}_1(z)_{m,m}$ implies that \eqref{eq:case1-1} is zero as well. Note that this also covers the case $n=0$. 
Assume $0\leq m<n$, and multiplying \eqref{eq:case1-1} by $(1-q^2)^2(1-z^2)$, and using the explicit expressions from Theorems \ref{thm:ldu} and \ref{thm:decoupling_diff_operator} and canceling common factors, 
we need to show that the following expression equals zero
\begin{align*}
&q^{-m}(1-q^{2+2m}z^2)\, \breve{C}_{n-m}(qz;q^{2m+2}|q^2) \\
&\quad + q^{-n} z \, (1-q^{2n+2m+2}) \, \breve{C}_{n-m-1}(z;q^{2m+2}|q^2) \\
&\quad - q^{-n}(1-q^{2+2n}z^2) \, \breve{C}_{n-m}(z;q^{2m+2}|q^2).
\end{align*}
Now we can use the Laurent expansion of \eqref{eqn:cont_q_ultra_poly} to rewrite the expression as
\begin{multline*}
q^{-m}(1-q^{2+2m}z^2)\, \sum_{k=0}^{n-m} \frac{(q^{2m+2};q^2)_k (q^{2m+2};q^2)_{n-m-k}}{(q^2;q^2)_k (q^2;q^2)_{n-m-k}} (qz)^{n-m-2k} \\
+ q^{-n}\, (1-q^{2n+2m+2}) \, \sum_{k=0}^{n-m-1} \frac{(q^{2m+2};q^2)_k (q^{2m+2};q^2)_{n-m-1-k}}{(q^2;q^2)_k (q^2;q^2)_{n-m-1-k}} z^{n-m-2k} \\
- q^{-n}(1-q^{2+2n}z^2) \, \sum_{k=0}^{n-m} \frac{(q^{2m+2};q^2)_k (q^{2m+2};q^2)_{n-m-k}}{(q^2;q^2)_k (q^2;q^2)_{n-m-k}} z^{n-m-2k}.
\end{multline*}
It is a straightforward calculation to show that the coefficient of $z^{n-m-2k}$, where $k\in\{-1,0,\cdots, n-m\}$, equals zero. 
This proves the case $i=1$ of \eqref{eq:matricesconjugatedqdiffop}.

To prove the case $i=2$ of \eqref{eq:matricesconjugatedqdiffop}, we evaluate $\mathcal{K}_2(z) \breve{L}^t(qz)-\breve{L}^t(z) \mathcal{M}_2(z)$ in the $(m,n)$-entry, which is slightly more complicated than the corresponding case for $i=1$
\begin{equation}\label{eq:case2-1}
\begin{split}
&\mathcal{K}_2(z)_{m,m-1} \breve{L}(qz)_{n,m-1} + 
\mathcal{K}_2(z)_{m,m} \breve{L}(qz)_{n,m} + 
\mathcal{K}_2(z)_{m,m+1} \breve{L}(qz)_{n,m+1} \\ 
&\qquad\qquad - \breve{L}(z)_{n+1,m} \mathcal{M}_2(z)_{n+1,n} - 
\breve{L}(z)_{n,m} \mathcal{M}_2(z)_{n,n}. 
\end{split}
\end{equation}
If $m>n+1$ all terms vanish in \eqref{eq:case2-1}, because of the lower triangularity of $L$. In case $m=n+1$, 
\eqref{eq:case2-1} reduces to $\mathcal{K}_2(z)_{n+1,n} - \mathcal{M}_2(z)_{n+1,n}$, which is indeed 
zero. 
Suppose $m \leq n$, we expand (\ref{eq:case2-1}) under the convention that continuous $q$-ultraspherical polynomials with negative degree are zero.
Expand the continuous $q$-ultraspherical polynomials in $z$ in (\ref{eq:case2-1}) and take out the term
\begin{equation*}
\frac{1}{(1 - z^2)} \frac{1}{(1 - q^2)^2} 
  \frac{(q^2;q^2)_n (q^2;q^2)_{2m+1}}{(q^2;q^2)_{m+n+1} (q^2;q^2)_m}
\end{equation*}
By a long but straightforward calculation we can show that the coefficient of $z^{n-m-2k}$, $k \in \{-1, 0, \ldots, n-m+1\}$ equals zero.
This proves the case $i = 2$.

\end{appendix}

\end{document}